\documentclass[a4paper]{article}

\usepackage{cite}

\usepackage{amsmath,amssymb,amsthm}
\usepackage{hyperref}
\usepackage{graphicx}
\usepackage{mathrsfs}

\numberwithin{equation}{section}

\newcommand{\coloneqq}{:=}

\DeclareMathOperator{\supp}{supp}
\DeclareMathOperator{\dist}{dist}
\DeclareMathOperator{\Spec}{spec}
\DeclareMathOperator{\Specess}{spec_\mathrm{ess}}
\DeclareMathOperator{\Specdisc}{spec_\mathrm{disc}}

\DeclareMathOperator{\Ran}{Ran}
\DeclareMathOperator{\Span}{span}

\renewcommand{\tilde}{\widetilde}

\newtheorem{theorem}{Theorem}[section]
\newtheorem{proposition}[theorem]{Proposition}
\newtheorem{lemma}[theorem]{Lemma}
\newtheorem{Corollaires}[theorem]{Corollary}

\newtheorem{openp}[theorem]{Open problem}

\theoremstyle{definition}

\newtheorem{definition}[theorem]{Definition}
\newtheorem{Remarques}[theorem]{Remark}
\newtheorem*{theorem*}{Theorem}
\newtheorem*{Corollaires*}{Corollary}

\title{\bf Spectral asymptotics for Robin Laplacians on polygonal domains}

\author{\sc Magda Khalile}

\date{}

\begin{document}
\maketitle

\begin{abstract}

Let $\Omega$  be a curvilinear polygon and $Q^\gamma_{\Omega}$ be the Laplacian in $L^2(\Omega)$, $Q^\gamma_{\Omega}\psi=-\Delta \psi$, with the Robin boundary condition $\partial_\nu \psi=\gamma \psi$, where $\partial_\nu$ is the outer normal  derivative and $\gamma>0$. We are interested in the behavior of the eigenvalues of $Q^\gamma_\Omega$ as $\gamma$ becomes large. 
We prove that the asymptotics of the first eigenvalues of $Q^\gamma_\Omega$ is determined at the leading order by those of model operators associated with the vertices: the Robin Laplacians acting on the tangent sectors associated with $\partial \Omega$. In the particular case of a polygon with straight edges 
 the first eigenpairs are exponentially close to those of the model operators. 
Finally, we prove a
Weyl asymptotics for the eigenvalue counting function of $Q^\gamma_\Omega$ for a threshold depending on $\gamma$, and show that the leading term is the same as for smooth domains.
\end{abstract}

\section{Introduction}

Let $\Omega\subset\mathbb R^2$ be a Lipschitz domain. For $\gamma>0$, we consider the Robin Laplacian $Q_\Omega^\gamma$ acting on $L^2(\Omega)$ 
as 
\[
Q^\gamma_\Omega\psi=-\Delta \psi, \quad \frac{\partial \psi}{\partial \nu} = \gamma\psi \text{ on } \partial \Omega,
\]
where $\nu$ is the outer unit normal. More rigorously, if $\Omega$ is either bounded or with a suitable behavior at infinity, the sesquilinear form 
\[
q_\Omega^\gamma(\psi,\psi)=\int_{\Omega} \lvert \nabla \psi\rvert^2 dx -\gamma\int_{\partial\Omega} \lvert \psi\rvert^2 ds,\quad \psi \in H^1(\Omega),
\]
where $s$ denotes the arc length of $\partial\Omega$,
is closed and semibounded from below and hence defines a unique self-adjoint operator which is denoted by $Q^\gamma_\Omega$. The boundary of $\Omega$ is either compact or non-compact. In the latter case, some additional assumptions are needed on $\partial \Omega$, see \cite{EM, PPeff}, to ensure the existence of discrete eigenvalues. In the following, we assume that $\partial\Omega$ is such that the discrete spectrum of $Q^\gamma_\Omega$ is not empty and we denote by $E_n(Q_\Omega^\gamma)$ its discrete eigenvalues counted the multiplicities and ordered in the increasing way. 
The problem involving Robin Laplacians appears in several applications as the study of reaction-diffusion equations in the long-time asymptotics, see \cite{LOS}, or the estimation of the critical temperature of superconductors, see \cite{GS}.
  
  In this paper, we are interested in the asymptotics of these eigenvalues as the parameter $\gamma$ goes to $+\infty$.
It is easy to see that  $E_n(Q_\Omega^\gamma)\to-\infty$ as $\gamma\to+\infty$ for each $n$. Moreover, by the standard Sobolev trace theorems, see for example \cite[Theorem 1.5.1.10]{GRB}, we know that there exists a constant $C_\Omega>0$ such that $
E_1(Q_\Omega^\gamma)\geq -C_\Omega\gamma^2
$ for  $\gamma$ large enough if $\Omega$ is bounded.

In the past few decades, more precise estimates have created a lot of interest and it was particularly pointed out that the behavior of the eigenvalues is sensitive to the regularity of the boundary.
As shown in \cite{LP, LOS, bp}, for a large class of domains $\Omega$ there exists a constant $C_\Omega\geq 1$ such that 
$E_1(Q_\Omega^\gamma) \sim -C_\Omega\gamma^2$ as $\gamma\to +\infty$. If $\partial\Omega$ is $C^1$, then $C_\Omega=1$ as proved in 
  \cite{LZ}. Later, it was proved in \cite{DK} that this asymptotics holds for any $E_n(Q^\gamma_\Omega)$.
Under additional smoothness assumptions, more precise results are obtained \cite{HK, HKR, KovP, PPeff}. In particular, in \cite{EMP, PP} it was shown that for each fixed $n$, and for large $\gamma$ there holds
\[
E_n(Q_\Omega^\gamma)=-\gamma^2-\kappa_{\max} \gamma+o(\gamma),
\]
where $\kappa_{\max}$ denotes the maximum of the curvature of $\partial \Omega$.
Furthermore, if 
 $\mathcal N(Q_\Omega^\gamma, \lambda)$ 
 denotes the number of eigenvalues of $Q_\Omega^\gamma$ in $(-\infty, \lambda)$, the following Weyl-type asymptotics was proved in \cite{HKR} for smooth bounded $\Omega$, 
\begin{align}
\label{weyl1}
\mathcal N(Q_\Omega^\gamma, E\gamma^2) = \gamma \frac{\lvert \partial \Omega\rvert \sqrt{E+1}}{\pi} +R_1(\gamma),\quad  R_1(\gamma)=O(1) \text{ as } \gamma\to +\infty,
\end{align}
for all $E\in(-1,0)$ and
\begin{align}
\label{weyl2}
\mathcal N(Q_\Omega^\gamma, -\gamma^2+\lambda\gamma) =
\frac{\sqrt{\gamma}}{\pi}
\int_{\partial\Omega}\sqrt{(\kappa(s)+\lambda)_+} ds+R_2(\gamma), \quad R_2(\gamma)=o(\sqrt{\gamma})\text{ as } \gamma\to +\infty,
\end{align}
for all $\lambda\in\mathbb R$, where $\kappa$ is the curvature of $\partial\Omega$ and $(x)_+\coloneqq\max(x,0)$. Higher dimensional analogues were considered in \cite{KKR}.

Few informations are available for non-smooth domains $\Omega\subset\mathbb R^2$.
By \cite{LP}, if $\Omega$ is a (suitably defined) curvilinear polygon which smallest angle is $2\alpha$, then 
\[
C_\Omega= \frac{1}{\sin^2\alpha} \text{ if } \alpha <\frac{\pi}{2}, \quad C_\Omega=1 \text{ otherwise}.
\] 
More precise asymptotics were only given for very specific $\Omega$ \cite{HP, Mc, Pan, P}. For a more detailed discussion of available results, we refer to the recent review paper \cite{BFK}, which also contains a number of interesting open problems. In particular, the following question was asked, see  
\cite[Open problem~4.19]{BFK}:
\begin{openp}\label{Openp}
Suppose that $\Omega\subset\mathbb R^2$ is a bounded, piecewise smooth domain having $L\geq1$ corners with half-angles $\alpha_1\leq...\leq \alpha_L<\displaystyle\frac{\pi}{2}$. Is it true that the first $L$ eigenvalues have the asymptotic behavior
\[
E_n(Q_\Omega^\gamma) \sim -\frac{\gamma^2}{\sin^2 \alpha_n},\quad \text{as}\quad \gamma\to +\infty,
\]
for $n=1,...,L$ ? How does $E_n(Q_\Omega^\gamma)$ behave for fixed $n\geq L$ ? Investigate the corresponding situation in higher dimensions and for more general $\Omega$.
\end{openp}
In the present paper we show, in particular, that the conjecture is not true as stated, and we propose and prove a correct version.

Let us pass to a description of the main results. 
Let $\Omega\subset\mathbb R^2$ be a curvilinear polygon with $C^4$ smooth sides (see Definition~\ref{defcurvpol} for a rigorous description). If $v$ is a vertex of $\Omega$ (that is a point at which the boundary is not smooth) we denote by $2\alpha_v \in(0,\pi)\cup (\pi,2\pi)$ the angle formed by the one-sided tangents at $v$ and introduce the set of the convex vertices by 
\[
\mathcal V\coloneqq\{v\in\partial \Omega : \alpha_v\in (0,\pi/2)\}.
\]
Denote by  $U_v$ the infinite sector of half aperture $\alpha_v$ given by
\[
U_v\coloneqq\{(x_1,x_2)\in\mathbb R^2 : \lvert \arg(x_1+ix_2)\rvert <\alpha_v\},
\]
and consider the associated Robin Laplacians $T_v\coloneqq Q^1_{U_v}$. This operator was studied in \cite{KP, LP} and we recall some of the results: 
the essential spectrum of $T_v$ does not depend on the half-angle of $U_v$, $\Specess(T_v)=[-1,+\infty)$, and the discrete spectrum is non-empty if and only if $\alpha_v<\frac{\pi}{2}$. Moreover,
if $\alpha_v <\frac{\pi}{2}$ then, $E_1(T_v)=-\frac{1}{\sin^2 \alpha_v}$, the discrete spectrum is finite,
\begin{align}
\label{KP1}
\mathcal N(T_v,-1)\to +\infty \text{ and } E_n(T_v)\to -\infty \text{ as } \alpha_v\to 0.
\end{align}
In addition, for all $\alpha_v\in[\frac{\pi}{6},\frac{\pi}{2})$, we have 
\begin{align}
\label{KP2}
\mathcal N(T_v,-1)=1.
\end{align}

We define the model operator
\[
T^\oplus\coloneqq \bigoplus_{v\in\mathcal V} T_v, \text{ and } \mathcal N^\oplus\coloneqq\sum_{v\in\mathcal V} \mathcal N_v,
\]
where $\mathcal N_v\coloneqq \mathcal N(T_v,-1)$.

Our main results are as follows. First, we discuss the behavior of the $\mathcal N^\oplus$ first eigenvalues of $Q^\gamma_\Omega$ as $\gamma$ becomes large. 

\begin{theorem}
\label{Intro_Thasvp}
For any $n\in\{1,...,\mathcal N^\oplus\}$ there holds
\begin{align*}
E_n(Q^\gamma_{\Omega})= E_n(T^\oplus) \gamma^2 +r(\gamma), \text{ as } \gamma \to +\infty,
\end{align*}
where $r(\gamma)=O(\gamma^{4/3})$, and one can take $r(\gamma)=O(e^{-c\gamma})$ with $c>0$ if $\Omega$ is a polygon with straight edges. 
\end{theorem}
For a precise statement see Theorem~\ref{asexp} for polygons with straight edges and Theorem~\ref{asexpcurv} for the general case. 
In addition, we show in Theorem~\ref{eigenspace} that, if $\Omega$ is a polygon with straight edges, the $\mathcal N^\oplus$ first associated eigenfuctions are localized near the convex vertices of $\Omega$. 

By Theorem~\ref{Intro_Thasvp}, we see that the conjecture stated in Open problem~\ref{Openp} becomes false if $E_2(T_v)<E_1(T_w)$ for some $v,w\in\mathcal V$, which happens for $\alpha_v$ small enough due to \eqref{KP1}.
 However, it is possible to find a setting for which the conjecture holds true.
 
 \begin{Corollaires}
Let $\Omega \subset \mathbb R^2$ be a curvilinear polygon having $L\geq 1$ convex vertices with half angles $\displaystyle \frac{\pi}{6} \leq \alpha_1\leq ...\leq \alpha_L <\displaystyle \frac{\pi}{2}$. Then, 
 \[
  E_n(Q_\Omega^\gamma)= -\frac{\gamma^2}{\sin^2\alpha_n}+O\left(\gamma^{\frac{4}{3}}\right), \quad \gamma\to+\infty,
 \]
for all $n=1,...,L$.
\end{Corollaires}
The proof follows immediatly from \eqref{KP2}. In particular, we have the following asymptotics for regular polygons.

\begin{Corollaires}
 Let $\Omega\subset \mathbb R^2$ be a regular polygon having $L\geq3$ edges. Then,
 \[
  E_n(Q_\Omega^\gamma)=-\frac{\gamma^2}{\sin^2\frac{(L-2)\pi}{2L}}
  +O\left(e^{-c\gamma}\right),\quad \gamma\to+\infty,
 \]
for all $n=1,...,L$.
\end{Corollaires}

In Theorem~\ref{weylcurvpol} we discuss the asymptotic behavior of the eigenvalue counting function of $Q^\gamma_\Omega$ as $\gamma\to+\infty$. 
\begin{theorem}
\label{Intro_asweyl}
The asymptotics \eqref{weyl1} and \eqref{weyl2} hold true when $\Omega$ is a curvilinear polygon with respectively $R_1(\gamma)=O(\gamma^\theta)$ for any $\theta\in(0,\frac{1}{2})$ and $R_2(\gamma)=O(\gamma^{\frac{1}{4}})$.
\end{theorem}
This result particularly means that the vertices do not contribute to the Weyl law at the leading order.

Finally in Section~\ref{prospects}, we discuss the second question in Open problem~\ref{Openp}. We prove that for each fixed $j\geq1$,
\begin{align}
\label{Intro_propprospects}
E_{\mathcal N^\oplus+j}(Q^\gamma_\Omega)=-\gamma^2+o(\gamma^2), \text{ as }\gamma\to +\infty,
\end{align}
see Proposition~\ref{nexteigen}.

The main tool in our proofs is the min-max characterization of the eigenvalues. The proof of the asymptotics of the first eigenvalues uses the idea of \cite{BD} in which a Schr\"odinger with magnetic field acting on curvilinear polygons is considered. It mainly relies on the construction of weak quasi-modes, thanks to the eigenfunctions of the model operator. The estimates on these weak quasi-modes are obtain using their decay property proved in \cite{KP}. In the particular case of a polygon with straight edges, these functions are true quasi-modes, namely they belong to the domain of the operator $Q^\gamma_\Omega$. 
It will allow us to use a spectral approximation result in order to prove the exponential decay of the remainder in the asymptotics and then to use a result of closeness of subspaces, see e.g \cite{HS}, to prove that linear combinations of quasi-modes are exponentially close, in a sense, to the associated eigenfunctions. 
To prove the Weyl-type asymptotics, we first use a partition of unity and a Dirichlet bracketing in order to remove the corners from the domain $\Omega$. 
We are then lead to study separately the corners and the rest of $\Omega$. 
We show that the corners do not contribute to the asymptotics at the leading order using the same kind of arguments as \cite{KK}. Then, the first term in the asymptotics comes from the study of the rest of the domain. To prove this, we adapt the sketch of the proof of \cite{Pan}. The idea, inspired by the proof of a Weyl law of a Schr\"odinger operator in \cite{RS}, consists in a reduction to a well chosen neighborhood of the boundary. The proof of the asymptotics \eqref{Intro_propprospects} directly follows from a combination of the preceding results.

In Section~\ref{Prel}, we recall some properties of one-dimensional operators and of Robin Laplacians acting on infinite sectors as they will play a crucial role in our study. We also introduce the model operator $T^\oplus$.
Section~\ref{RLpol} is devoted to the study of polygons with straight edges: we prove Theorem~\ref{Intro_Thasvp} for the particular case of polygons and the result on the associated eigenfunctions.
Section~\ref{curvpol} is devoted to the study of general curvilinear polygons: we prove Theorem~\ref{Intro_Thasvp} for curvilinear polygons and Theorem~\ref{Intro_asweyl}.
In Section~\ref{prospects}, we give the proof of the asymptotics \eqref{Intro_propprospects}.
Finally in Appendix~\ref{speccorol}, we recall the proof of a spectral approximation result used in Section~\ref{RLpol}.

\section{Preliminaries}\label{Prel}

\subsection{Min-max principle}
\textbf{General notation.} If $A$ is a self-adjoint, semibounded from below operator acting on a Hilbert space $\mathcal H$ of domain $D(A)$, 
we denote by $a$ the associated sesquilinear form of domain $D(a)$.
 For $\lambda \in\mathbb R$, $\mathcal N(A,\lambda)$ denotes the number of eigenvalues, 
counting the multiplicities, 
of $A$ in $(-\infty,\lambda)$ if $\Specess(A)\cap (-\infty,\lambda)=\emptyset$, 
and $\mathcal N(A,\lambda)=+\infty$ otherwise. 
We denote by $\Spec(A)$, $\Specdisc(A)$, $\Specess(A)$ respectively the spectrum of $A$, its discrete spectrum and its essential spectrum. 
By $E_n(A)$ we denote its $n$th discrete eigenvalue, when ordered in the non-decreasing order and counting the multiplicities.

Let $A$ be a self-adjoint operator acting on a Hilbert space $\mathcal H$ of infinite dimension.
We assume that $A$ is semi-bounded from below, $A\ge -c$
, $c\in \mathbb R$, and denote
\[
\Sigma:=\begin{cases}
\inf \Specess A, & \text{ if }\Specess (A)\ne \emptyset,\\
+\infty, & \text{ if }\Specess (A)=\emptyset.
\end{cases}
\]
Recall that $D(a)$, equipped with the scalar product 
$\langle u,v\rangle_a:=a(u,v)+(c+1)\langle u,v\rangle$,
is a Hilbert space.
The following result, giving a variational characterization of eigenvalues,  is a standard tool of the spectral theory, see e.g. \cite[Section XIII.1]{RS}.
\begin{theorem}[Min-max principle]
Let $n \in \mathbb N$ and $D$ be a dense subspace of the Hilbert space $D(a)$. 
Let $\Lambda_n(A)$
be the $n$th \emph{Rayleigh quotient} of $A$, which is defined by
\[
\Lambda_n(A):=\sup_{\psi_1,...,\psi_{n-1} \in \mathcal H}\inf_{\substack{\varphi \in D,\varphi \neq 0\\ \varphi \perp \psi_j, j=1,...,n-1}}\frac{a(\varphi,\varphi)}{\langle \varphi,\varphi \rangle}\equiv \inf_{\substack{G\subset D \\ \dim G=n}} \sup_{\substack{\varphi \in G\\\varphi \neq 0}} \frac{a(\varphi,\varphi)}{\langle \varphi,\varphi \rangle},
\]
then one and only one of the following assertions is true:
\begin{enumerate}
\item $\Lambda_n(A)<\Sigma$ and $E_n(A)=\Lambda_n(A)$.
\item $\Lambda_n(A)=\Sigma$ and $\Lambda_m(A)=\Lambda_n(A)$ for all $m\ge n$.
\end{enumerate}
\end{theorem}

\subsection{Auxiliary one-dimensional operators}\label{auxop}

In this section, we recall some results on one-dimensional Laplacians acting on an interval.

\begin{proposition}{\cite[Lemma A.2]{HP}}\label{LemmaHP}
For $\gamma>0$ and $l>0$, denote by $\mathscr D_{\gamma,l}$ the operator acting on $L^2(0,l)$ as $f\mapsto -f''$ with
\[
 D(\mathscr D_{\gamma,l})\coloneqq\{f\in H^2(0,l), -f'(0)-\gamma f(0)=f(l)=0\}.
\]
Then, $E_1(\mathscr D_{\gamma,l})<0$ iff $\gamma l >1$, and in that case it is the unique negative eigenvalue. Moreover, for a fixed $l>0$ one has
\[
 E_1(\mathscr D_{\gamma,l})=-\gamma^2+4\gamma^2 e^{-2\gamma l} +O( e^{-4\gamma l}), \quad\text{as}\quad \gamma\to +\infty.
\] 
\end{proposition}
 
\begin{proposition}{\cite[Lemma 3]{Pan}}\label{LemmaPan}
 For $\gamma,\beta,l>0$, denote by $\mathscr R_{\gamma,\beta,l}$ the operator acting on $L^2(0,l)$ as $f\mapsto -f''$ with 
 \[
  D(\mathscr R_{\gamma,\beta,l})\coloneqq\{ f\in H^2(0,l), -f'(0)-\gamma f(0)= f'(l)-\beta f(l)=0\}.
 \]
If $\gamma>2\beta$ and $\gamma l>1$ then $E_1(\mathscr R_{\gamma,\beta,l})$ is the unique negative eigenvalue and 
\[
 -\gamma^2-  123 \gamma^2e^{-2\gamma l} < E_1(\mathscr R_{\gamma,\beta,l}) < -\gamma^2.
\]
\end{proposition}

\subsection{Robin Laplacian on infinite sectors}\label{infiniteSect}

For $\alpha\in(0,\pi)$, we define $U(\alpha)$ the infinite sector of opening $2\alpha$,
 \[
U(\alpha)\coloneqq\{(x_1,x_2)\in\mathbb R^2, \lvert \arg(x_1+ix_2)\rvert <\alpha\}.
\]
Denote by $T^{\gamma,\alpha}$ the Robin Laplacian acting on $L^2\left(U(\alpha)\right)$ 
as $T^{\gamma,\alpha} \psi= -\Delta \psi$ on $U(\alpha)$, 
with the Robin boundary condition $\partial_\nu \psi=\gamma \psi$ 
on $\partial U(\alpha)$ 
where $\nu$ stands for the unit outward normal and $\gamma >0$. 
The operator $T^{\gamma,\alpha}$ is defined as the unique self-adjoint operator associated with the sesquilinear form
\begin{align*}
t^{\gamma,\alpha}(\psi,\psi)=\int_{U(\alpha)} \lvert \nabla \psi\rvert ^2 dx-\gamma \int_{\partial U(\alpha)} \lvert \psi \rvert^2 ds, \quad \psi \in H^1\left(U(\alpha)\right).
\end{align*}
As mentioned above, this operator will play a particular role in our study and we will use some of its spectral properties gathered in \cite{LP, KP}.
For the reader's convenience, we recall some of them in this section.
\begin{theorem}\label{propKP}
For all $\alpha \in(0,\pi)$ and $\gamma>0$, $\Specess(T^{\gamma,\alpha})=[-\gamma^2,+\infty)$ and the discrete spectrum of $T^{\gamma,\alpha}$ is non-empty if and only if $\alpha<\displaystyle\frac{\pi}{2}$. Moreover, 
\begin{itemize}
\item if $\alpha \in (0,\displaystyle \frac{\pi}{2})$, then 
$
E_1(T^{\gamma,\alpha})=-\displaystyle\frac{\gamma^2}{\sin^2\alpha},
$
$u(x_1,x_2)=\exp(-\gamma\displaystyle\frac{x_1}{\sin \alpha})$ is an associated eigenfunction, and $\mathcal N(T^{\gamma,\alpha},-\gamma^2) <+\infty;$
\item for all $\alpha\in[\displaystyle\frac{\pi}{6},\displaystyle\frac{\pi}{2})$, we have $\mathcal N(T^{\gamma,\alpha},-\gamma^2)=1.$
\end{itemize}
\end{theorem}
In  \cite[Theorem 4.1]{KP}, an estimate on the Rayleigh quotients of $T^{\gamma,\alpha}$ as $\alpha$ is small is obtained which has as a direct consequence the following proposition.

\begin{proposition}
There exists $\kappa>0$ such that $\mathcal N(T^{\gamma,\alpha},-\gamma^2) \geq \kappa/\alpha$ as $\alpha$ is small. In particular
\[
\mathcal N(T^{\gamma,\alpha},-\gamma^2)\to +\infty,\quad \text{as}\quad \alpha\to 0.
\]
\end{proposition}

Some following results are based on the decay property of the associated eigenfunctions \cite[Theorem 5.1]{KP}.

\begin{theorem}\label{Agmon}
Let $E$ be a discrete eigenvalue of $T^{\gamma,\alpha}$ and $\psi$ be an associated eigenfunction. Then, for any $\epsilon \in (0,1)$ there exists $C_\epsilon>0$ such that we have 
\[
\int_{U(\alpha)} \left ( \lvert \nabla \psi\rvert^2 +\lvert \psi \rvert^2\right) e^{2 (1-\epsilon)\sqrt{-\gamma^2-E}\lvert x \rvert} dx <C_\epsilon	.
\]
\end{theorem}

Notice that, since the domain $U(\alpha)$ is invariant by dilations, a simple change of variables tells us that $T^{\gamma,\alpha}$ is unitarily equivalent to $\gamma^2T^{1,\alpha}$. In particular $E_n(T^{\gamma,\alpha})=\gamma^2 E_n(T^{1,\alpha})$ and $\mathcal N(T^{\gamma,\alpha},-\gamma^2)=\mathcal N(T^{1,\alpha}, -1)$. Let us denote by $(\psi^{1}_n)_n$ the normalized eigenfunctions of $T^{1,\alpha}$.
Then, $(\psi^{\gamma}_n)_n$ defined as 
\begin{align}
\label{eigenFsector}
 \psi^{\gamma}_n(x)\coloneqq \gamma \psi^{1}_n(\gamma x),
\end{align}
 are the eigenfunctions of $T^{\gamma,\alpha}$ satisfying 
\[ 
\lVert \psi^{\gamma}_n\rVert_{L^2\left(U(\alpha)\right)}=1, \quad
\lVert \nabla \psi^{\gamma}_n\rVert_{L^2\left(U(\alpha)\right)}=\gamma\lVert \nabla\psi^{1}_n\rVert_{L^2\left(U(\alpha)\right)}, \quad
\lVert \psi^{\gamma}_n\rVert_{L^2\left(\partial U(\alpha)\right)}=\sqrt{\gamma} \lVert \psi^{1}_n \rVert_{L^2\left(\partial U(\alpha)\right)}.
\]

\subsection{Definition of curvilinear polygons}\label{sectdefcurvpol}

Let us introduce a rigorous definition of the domains we consider.

\begin{definition}
\label{defcurvpol}
Let $\Omega\subset\mathbb R^2$ be a bounded open set. We say that $\Omega$ is a \emph{curvilinear polygon} if $\partial \Omega$ is Lipschitz and if there exists $M\geq 1$ non-intersecting connected arcs $\Gamma_k$, $k=1,...,M$, such that 
\[
\partial\Omega= \bigcup_{k=1}^M \overline{\Gamma_k},
\]
and if we denote by $l_k$ the length of $\Gamma_k$ and by $\gamma_k$ a parametrization of $\overline{\Gamma_k}$ by the arc length then $\gamma_k\in C^4([0,l_k])$. Moreover, 
if two components $\Gamma_k$, $\Gamma_j$ intersect at some point $v\coloneqq\Gamma_k(l_k)=\Gamma_j(0)$, then two cases are allowed: either $\overline{\Gamma_k}\cup\overline{\Gamma_j}$ is $C^4$ near $v$ and then $v$ is called a regular point of $\partial \Omega$, or the corner opening angle at $v$, called $\alpha_v$, measured inside $\Omega$ and formed by the one-sided tangents at $v$ belongs to $(0,\pi)\cup(\pi,2\pi)$. In the latter case, $v$ is called a vertex of $\Omega$. 
\end{definition}
Notice that cusps (zero angles) are not allowed by our definition as the boundary is Lipschitz. 

We introduce the set of convex vertices of $\Omega$ by
\[
\mathcal V\coloneqq\left \{ v\in\partial \Omega, \text{ $v$ is a vertex of $\Omega$ and } \alpha_v\in(0,\frac{\pi}{2}) \right\}.
\]
It is then easy to see that, for each $v\in\mathcal V$ there exists $r_v>0$ and $F_v$ a $C^2$-diffeomorphism satisfying the following conditions:
\begin{itemize}
\item[(a)] $F_v : \Omega \cap B(v,r_v) \to U(\alpha_v)\cap B(0,r_v)$,
\item[(b)] $F_v(\overline{\Omega\cap B(v,r_v)})= \overline{U(\alpha_v)\cap B(0,r_v)}$,
\item [(c)] $ F_v(v)=0$ and $\nabla F_v(v)= I_2$,
\end{itemize}
where $I_2$ stands for the identity matrix in two dimensions, $B(v,r_v)$ is the ball of center $v$ and radius $r_v$ in $\mathbb R^2$, $\nabla F_v$ is the Jacobian matrix of $F_v$.
We say that $U(\alpha_v)$ is the tangent sector of $\Omega$ at $v$.

\subsection{Model operator}\label{modelop}

In this section we introduce the model operator and some important notation which will be used in the whole paper.

Let $\Omega$ be a curvilinear polygon.
 For $v\in\mathcal V$, we denote by $(\psi^{\gamma,v}_n)_n$ the normalized eigenfunctions of $T^{\gamma,\alpha_v}$. In the following, we use the simpler notation
\[
T_v\coloneqq T^{1,\alpha_v},\text{ and } U_v\coloneqq U(\alpha_v),
\]
and we introduce $\mathcal N_v\coloneqq \mathcal N(T_v,-1)$. 

We define the model operator $T^\oplus$ as the direct sum of Robin Laplacians on tangent sectors associated with the convex vertices of $\partial \Omega$, 
\[
T^{\oplus}\coloneqq \bigoplus_{v\in\mathcal V} T_v,
\]
acting on $\bigoplus_{v\in\mathcal V}L^2(U_v)$.
Then $\Spec(T^{\oplus})=\bigcup_{v\in\mathcal V} \Spec(T_v)$. We denote  by $\mathcal N^\oplus\coloneqq\sum_{v\in\mathcal V} \mathcal N_v$, and $\Lambda\coloneqq\{ \lambda_l, 1\leq l \leq K^\oplus\}$ the eigenvalues of $T^{\oplus}$ ordered in the increasing way and counted \emph{without} multiplicity, namely : $\lambda_1<\lambda_2<...<\lambda_K^\oplus$. For $1\leq l \leq K^\oplus$ we introduce 
\[
\mathcal S_l\coloneqq \{(n,v) :  v\in\mathcal V, 1\leq n\leq\mathcal N_v: E_n(T_v)=\lambda_l\},
\]
and $m_l\coloneqq \# \mathcal S_l$. Defined like this, $m_l$ is then the multiplicity of $\lambda_l$ as an eigenvalue of $T^\oplus$ and $\sum_{l=1}^{K^\oplus} m_l =\mathcal  N^\oplus$. 
Finally we denote by $E^{\max}\coloneqq E_{\mathcal N^\oplus}(T^\oplus)$.

\section{Robin Laplacian on polygons}\label{RLpol}

We begin our study with the particular case of $\Omega$ being 
a bounded connected polygon with straight edges, namely each $\Gamma_k$ in Definition~\ref{defcurvpol} is a segment. As there is no ambiguity, we denote $Q^\gamma\coloneqq  Q^\gamma_\Omega$.
For each $v\in \mathcal V$, there exists $\tilde U_v$ an infinite sector of half aperture $\alpha_v$ and of vertex $v$ such that, for $r>0$ small enough,
\[
\Omega \cap B(v,r)=\tilde U_v\cap B(v,r),
\]
and there exists $F_v$ a rotation composed by a translation satisfying
\[
 U_v=\{ F_v(x), x\in \tilde U_v\}.
\]

\subsection{Description of quasi-modes}\label{descriptqm}

Let $v\in\mathcal V$. For $n\in\{1,...,\mathcal N_v\}$ we set
$
\phi^{\gamma,v}_n\coloneqq \psi^{\gamma,v}_n \circ F_v.
$
Then, $\phi^{\gamma,v}_n\in H^2(\tilde U_v)$ and satisfies the Robin boundary condition on $\partial\tilde U_v$ with the Robin parameter $\gamma$. 
Let us introduce 
\[
\rho_v\coloneqq \frac{\dist(v,\mathcal V \backslash \{v\})}{2}, \text{ and } \rho\coloneqq \displaystyle \frac{\min_{v\in \mathcal V} \rho_v}{2}.
\]
Let $\varphi \in C^{\infty} (\mathbb R_+)$ be a smooth cut-off function satisfying 
$0\leq \varphi \leq 1$, 
$\varphi(t)=1$ if $0\leq t\leq 1$, 
and $\varphi(t) =0$ if $t\geq 2$. 
We introduce the smooth radial cut-off function $\chi_v$ defined as follows:
\[
\chi_v(x)=\varphi\left(\frac{\lvert x - v\rvert}{ \rho}\right), \quad x\in \Omega.
\]
 Notice that, for $v\neq v'$, $\supp \chi_v \cap \supp \chi_{v'} =\emptyset$.
Finally we set, for $1\leq n \leq \mathcal N_v$,
\[
\tilde \phi^{\gamma,v}_n\coloneqq \phi^{\gamma,v}_n \chi_v \text{ on } \Omega.
\]

\begin{proposition}\label{qm}
For any $\epsilon\in(0,1)$, 
there exists $C >0$ such that 
for $v\in\mathcal V$, $n\in\{1,...,\mathcal N_v\}$ 
and $\gamma>0$ we have $\tilde \phi^{\gamma,v}_n\in D(Q^\gamma)$, and
\begin{align}
\label{qm1}
1-C e^{-2\gamma(1-\epsilon)\sqrt{-1-E^{\max}} \rho} &\leq \lVert \tilde \phi^{\gamma,v}_n \rVert^2_{L^2(\Omega)} \leq 1, \\
\label{qm2}
\frac{\lVert Q^\gamma \tilde\phi^{\gamma,v}_n-\gamma^2E_n(T_{v})\tilde \phi^{\gamma,v}_n\rVert^2_{L^2(\Omega)}}{\lVert \tilde\phi^{\gamma,v}_n\rVert^2_{L^2(\Omega)}}&\leq C e^{-2\gamma(1-\epsilon) \sqrt{-1-E^{\max}}\rho}.
\end{align}
\end{proposition}

\begin{Remarques}
In the sequel we denote by $C$ all the constants depending eventually on $\epsilon$ and not on $v$. 
If a constant $C(v)$ depending on $v\in\mathcal V$ appears, as $\# \mathcal V $ is finite, it is sufficient to take
 $C\coloneqq\max_{v\in\mathcal V} C(v)$.
\end{Remarques}

\begin{proof}[Proof of Proposition \ref{qm}]
We start proving \eqref{qm1}.
We immediately see that $\lVert \tilde \phi^{\gamma,v}_n\rVert ^2_{L^2(\Omega)}\leq \lVert  \psi^{\gamma,v}_n\rVert^2_{L^2(U_{v})}=1.$ On the other hand, we have
\begin{align*}
\lVert \tilde \phi^{\gamma,v}_n\rVert^2_{L^2(\Omega)} \geq \int_{\Omega\cap B(v,\rho)} \lvert  \phi_n^{\gamma,v}\rvert^2dx 
=\int_{U_v} \lvert  \psi^{\gamma,v}_n \rvert^2 dx - \int_{U_v\backslash B(v,\rho)} \lvert  \psi ^{\gamma,v}_n\rvert^2 dx.
\end{align*}
We now can apply Theorem~\ref{Agmon} to $\psi^{\gamma,v}_n$ to get:
\[
\int_{U_v\backslash B(v,\rho)} \lvert  \psi ^{\gamma,v}_n\rvert^2 dx \leq C
e^{-2(1-\epsilon)\gamma\sqrt{-1-E^{\max}}\rho},
\]
which gives us the lower bound for $\lVert \tilde \phi^{\gamma,v}_n\rVert^2_{L^2(\Omega)}$ and concludes the proof of \eqref{qm1}.

To prove that $\tilde \phi^{\gamma,v}_n \in D(Q^\gamma)$ 
we have to show that $-\Delta \tilde \phi^{\gamma,v}_n\in L^2(\Omega)$, which is easily checked as $\chi_v$ is smooth, 
and that $\partial_\nu \tilde \phi^{\gamma,v}_n =\gamma \tilde\phi^{\gamma,v}_n$ on $\partial \Omega$. 
As $\chi_v$ is radial, $\partial_\nu \chi_v=0 $ on $\partial \left( \Omega \cap B(v,2\rho)\right) \backslash \partial B(v,\rho)$ and then $\tilde \phi^{\gamma,v}_n$ satisfies the Robin boundary condition. 
Thus we can write 
\[
Q^\gamma \tilde \phi^{\gamma,v}_n=-\Delta (\phi^{\gamma,v}_n)\chi_v-\Delta(\chi_v) \phi^{\gamma,v}_n -2\nabla \phi^{\gamma,v}_n \nabla\chi_v,
\]
and for all $x\in \supp \chi_v$,
\[
-\Delta \phi^{\gamma,v}_n=\gamma^2E_n(T_{v}) \phi^{\gamma,v}_n.
\]
Using the fact that $\supp \Delta (\chi_v) \subset \supp \nabla \chi_v \subset B(v,2\rho)\backslash \overline{B(v,\rho)}$ and Theorem~\ref{Agmon}, we obtain 
\[
\lVert \Delta(\chi_v) \phi^{\gamma,v}_n \rVert^2_{L^2(\Omega)} \leq \lVert \Delta(\chi_v)\rVert^2_\infty C e^{-2(1-\epsilon)\gamma\sqrt{-1-E^{\max}}\rho},
\]
and,
\[
\lVert \nabla \phi^{\gamma,v}_n \nabla \chi_v \rVert^2_{L^2(\Omega)} \leq \lVert \nabla \chi_v\rVert^2_\infty  C e^{-2(1-\epsilon)\gamma\sqrt{-1-E^{\max}}\rho}.
\]
Gathering the two previous inequalities gives us
\begin{align}
\label{qm2_1}
\lVert Q_\gamma\tilde\phi^{\gamma,\alpha_s}_n-\gamma^2 E_n(T_{\alpha_s})\tilde\phi^{\gamma,\alpha_s}_n\rVert^2_{L^2(\Omega)} \leq C e^{-2(1-\epsilon)\gamma\sqrt{-1-E^{\max}}\rho}.
\end{align}
Putting \eqref{qm1} and \eqref{qm2_1} together finishes the proof.
\end{proof}

\begin{Corollaires}
For any $\epsilon \in(0,1)$ there exists a constant  $C >0$ such that  
\[
\dist(\gamma^2 E_n(T_v), \Spec(Q^\gamma)) \leq C e^{-(1-\epsilon)\gamma\sqrt{-1-E^{\max}}\rho}.
\]
\end{Corollaires}

\begin{proof}
This is a consequence of the spectral theorem due to \eqref{qm2}.
\end{proof}

\subsection{Properties of quasi-modes}\label{propqm}

In order to prove Theorem~\ref{asexp} we will need some properties satisfied by the quasi-modes gathered in the following lemma.
\begin{lemma}\label{pq}
Let $\epsilon\in(0,1)$. There exists a constant $C>0$ such that, for all $v\in\mathcal V$,
for all $n\in\{1,...,\mathcal N_v\}$ and 
for all $i\neq j$, $(i,j)\in\{1,...,\mathcal N_v\}^2$ we have for $\gamma$ large enough,
\begin{align}
\label{pq2}
\left \lvert q^{\gamma}(\tilde \phi^{\gamma,v}_n,\tilde \phi^{\gamma,v}_n) -\gamma^2 E_n(T_v)\right \rvert &\leq \gamma C e^{-2(1-\epsilon)\gamma \sqrt{-1-E^{\max}}\rho}, \\
\label{pq1}
\left \lvert \langle \tilde \phi^{\gamma,v}_i,\tilde \phi^{\gamma,v}_j\rangle_{L^2(\Omega)} \right \rvert &\leq C e^{-2(1-\epsilon)\gamma \sqrt{-1-E^{\max}}\rho},\\
\label{pq3}
\left \lvert q^{\gamma}(\tilde \phi^{\gamma,v}_i,\tilde \phi^{\gamma,v}_j) \right \rvert &\leq \gamma C e^{-2(1-\epsilon)\gamma \sqrt{-1-E^{\max}}\rho}.
\end{align}
\end{lemma}

\begin{proof}
We start proving \eqref{pq2}. Let us first expand $q^\gamma(\tilde \phi^{\gamma,v}_n,\tilde \phi^{\gamma,v}_n)$:
\begin{multline*}
q^\gamma(\tilde \phi^{\gamma,v}_n,\tilde \phi^{\gamma,v}_n) = \int_{\Omega} \lvert \chi_v\rvert^2 \lvert \nabla \phi^{\gamma,v}_n \rvert^2 dx -\gamma \int_{\partial \Omega} \lvert \chi_v\rvert^2 \lvert \phi^{\gamma,v}_n\rvert^2 ds \\
+ \int_{\Omega} \lvert \nabla \chi_v\rvert^2 \lvert \phi^{\gamma,v}_n\rvert^2 dx +2\Re \int_{\Omega } \chi_v \nabla \phi^{\gamma,v}_n\overline{\nabla \chi_v \phi^{\gamma,v}_n} dx.
\end{multline*}
As $\supp \nabla \chi_v \subset B(v,2\rho)\backslash\overline{ B(v,\rho)}$, 
we can use Theorem~\ref{Agmon} to bound the cross-term:
\begin{align*}
\Big \lvert \int_{\Omega} \lvert \nabla \chi_v\rvert^2 \lvert \phi^{\gamma,v}_n\rvert^2 dx 
&+2\Re \int_{\Omega } \chi_v \nabla \phi^{\gamma,v}_n \overline{\nabla \chi_v \phi^{\gamma,v}_n} dx \Big\rvert \\
&\leq \left (  2\lVert \nabla \chi_v\rVert^2_\infty \int_{U_v\backslash B(v,\rho)} \lvert \psi^{\gamma,v}_n\rvert^2 dx + \lVert \chi_v\rVert^2_\infty \int_{U_v\backslash B(v,\rho)} \lvert \nabla\psi^{\gamma,v}_n\rvert^2 dx \right ) \\
&\leq C e^{-2(1-\epsilon)\gamma \sqrt{-1-E^{\max}} \rho}.
\end{align*}
We now focus on the main term:
\[
\int_{\Omega} \lvert \chi_v\rvert^2 \lvert \nabla \phi^{\gamma,v}_n\rvert^2 dx \leq \int_{U_v} \lvert \nabla \psi^{\gamma,v}_n\rvert^2 dx,\text{ and } \int_{\partial \Omega} \lvert \chi_v\rvert^2 \lvert \phi^{\gamma,v}_n\rvert^2 ds \leq \int_{\partial U_v} \lvert \psi ^{\gamma,v}_n\rvert^2 ds.
\]
On the other hand,
\begin{align*}
\int_{\Omega} \lvert\chi_v \rvert^2 \lvert \nabla \phi^{\gamma,v}_n\rvert^2 dx &\geq \int_{U_v\cap B(v,\rho)} \lvert \nabla \psi^{\gamma,v}_n\rvert^2 dx\\
&\geq \int_{U_v} \lvert \nabla \psi^{\gamma,v}_n\rvert^2 dx -C e^{-2(1-\epsilon)\gamma \sqrt{-1-E^{\max}} \rho},
\end{align*}
and,
\begin{align*}
\int_{\partial \Omega} \lvert \chi_v\rvert^2 \lvert\phi^{\gamma,v}_n\rvert^2 ds &\geq \int_{\partial(U_v\cap B(v,\rho))\backslash \partial B(v,\rho)} \lvert \psi^{\gamma,v}_n\rvert^2 ds \\
&= \int_{\partial U_v} \lvert \psi^{\gamma,v}_n\rvert^2 ds -\int_{\partial(U_v \backslash B(v,\rho))\backslash\partial B(v,\rho)} \lvert \psi^{\gamma,v}_n\rvert^2 ds.
\end{align*}
Notice that, as $U_v\backslash \overline{B(v,\rho)}$ is a Lipschitz domain there exists $K>0$ such that 
\begin{align}
\label{lip}
\lVert \psi^{\gamma,v}_n\rVert^2_{L^2(\partial(U_v \backslash \overline{B(v,\rho)}))} \leq K \lVert \psi^{\gamma,v}_n\rVert^2_{H^1(U_v\backslash \overline{B(v,\rho)})}.
\end{align}
Then, using \eqref{lip} and Theorem~\ref{Agmon} we obtain 
\[
\int_{\partial \Omega} \lvert \chi_v\rvert^2 \lvert \phi^{\gamma,v}_n\rvert^2 ds \geq \int_{\partial U_v} \lvert\psi^{\gamma,v}_n\rvert^2 ds -C e^{-2(1-\epsilon)\gamma\sqrt{-1-E^{\max}}\rho}.
\]
As $t^{\gamma,\alpha_v}(\psi^{\gamma,v}_n,\psi^{\gamma,v}_n)=\gamma^2 E_n(T_v)$,
\begin{align}
\label{estbil}
\left \lvert \int_{\Omega} \lvert \chi_v\rvert^2 \lvert \nabla \phi^{\gamma,v}_n \rvert^2 dx -\gamma \int_{\partial \Omega} \lvert \chi_v\rvert^2 \lvert \phi^{\gamma,v}_n\rvert^2 ds -\gamma^2 E_n(T_v)\right\rvert \leq \gamma C e^{-2(1-\epsilon)\gamma\sqrt{-1-E^{\max}}\rho},
\end{align}
which concludes the proof combining \eqref{estbil} with the estimate on the cross-term.

 Let us now prove \eqref{pq1}. As $i\neq j$, $\displaystyle \int_{U_v} \psi^{\gamma,v}_i \overline{\psi^{\gamma,v}_j}dx=0$. Then,
\begin{align*}
\langle \tilde\phi^{\gamma,v}_i,\tilde \phi^{\gamma,v}_j\rangle_{L^2(\Omega)}=\int_{U_v\backslash B(v,2\rho)} \left(\lvert \chi_v\circ F_v\rvert^2 -1\right ) \psi^{\gamma,v}_i \overline{\psi^{\gamma,v}_j} dx.
\end{align*}
We can conclude, as 
$\left \lvert \lvert \chi_v\rvert^2 -1\right \rvert \leq 1$,
 using Cauchy-Schwarz inequality and Theorem~\ref{Agmon}.

 To finish, let now focus on \eqref{pq3}. Let $i\neq j$, we have 
\[
q^\gamma(\tilde \phi^{\gamma,v}_i,\tilde \phi^{\gamma,v}_j) =\int_{\Omega} \lvert \chi_v\rvert^2 \nabla \phi^{\gamma,v}_i  \overline{\nabla \phi^{\gamma,v}_j }  dx  -\gamma\int_{\partial \Omega} \lvert\chi_v\rvert^2\phi^{\gamma,v}_i \overline{\phi^{\gamma,v}_j}ds + I(\gamma),
\]
where 
\[
I(\gamma)= \int_{\Omega} \left \{\lvert \nabla \chi_v\rvert^2 \phi^{\gamma,v}_i \overline{\phi^{\gamma,v}_j} + \phi^{\gamma,v}_i \nabla \chi_v  \overline{\nabla \phi^{\gamma,v}_j \chi_v} + \chi_v \nabla \phi^{\gamma,v}_i \overline{\nabla \chi_v \phi^{\gamma,v}_j}  \right \}dx.
\]
Using the fact that $\supp \nabla \chi_v \subset B(v,2\rho)\backslash \overline{B(v,\rho)}$, 
Cauchy-Schwarz and Theorem~\ref{Agmon} we get 
\begin{align}
\label{estI}
\lvert I(\gamma)\rvert \leq C e^{-2(1-\epsilon)\gamma\sqrt{-1-E^{\max}}\rho}.
\end{align}
By the spectral theorem we have $t^{\gamma}_v (\psi^{\gamma,v}_i,\psi^{\gamma,v}_j)=0$, which implies 
\begin{multline*}
\int_{\Omega} \lvert \chi_v\rvert^2 \nabla \phi^{\gamma,v}_i  \overline{\nabla \phi^{\gamma,v}_j }  dx  -\gamma\int_{\partial \Omega} \lvert\chi_v\rvert^2\phi^{\gamma,v}_i \overline{\phi^{\gamma,v}_j}ds \\
= \int_{U_v\backslash \overline{B(v,2\rho)}} (\lvert \chi_v\circ F_v \rvert^2-1) \nabla \psi^{\gamma,v}_i  \overline{\nabla \psi^{\gamma,v}_j }  dx
-\gamma\int_{\partial (U_v\backslash \overline{B(v,2\rho)})\backslash \partial B(v,2\rho)} (\lvert\chi_v\circ F_v\rvert^2-1)\psi^{\gamma,v}_i \overline{\psi^{\gamma,v}_j}ds.
\end{multline*}
We can use the same arguments as before and \eqref{lip} to obtain 
\begin{align}
\label{estII}
\left \lvert \int_{\Omega} \lvert \chi_v\rvert^2 \nabla \phi^{\gamma,v}_i \overline{\nabla \phi^{\gamma,v}_j }  dx  -\gamma\int_{\partial \Omega} \lvert\chi_v\rvert^2\phi^{\gamma,v}_i \overline{\phi^{\gamma,v}_j}ds\right \rvert \leq \gamma C e^{-2(1-\epsilon)\gamma\sqrt{-1-E^{\max}}\rho}.
\end{align}
Putting \eqref{estI} and \eqref{estII} together finishes the proof of \eqref{pq3}.
\end{proof}

\begin{lemma}\label{li}
For $\gamma$ large enough the family $(\tilde \phi^{\gamma,v}_n)_{(n,v)\in \displaystyle\cup_{l=1}^{K^\oplus}\mathcal S_l}$ is linearly independent.
\end{lemma}

\begin{proof}
Let us denote by $G$ the Gramian matrix associated with $(\tilde \phi^{\gamma,v}_n)_{(n,v)\in \cup_{l=1}^{K^\oplus} S_l}$ 
which entries are 
$G_{i,j}=\langle\tilde  \phi^{\gamma,v_i}_{n_i},\tilde\phi^{\gamma,v_j}_{n_j}\rangle$, 
where $(n_i,v_i),(n_j,v_j) \in \displaystyle \cup_{l=1}^{K^\oplus} \mathcal S_l$. 
First, the diagonal is simply composed of $1+o(1)$ as $\gamma \to +\infty$, according to \eqref{qm1}. 
Secondly, if $(n_i,v_i)\neq (n_j,v_j)$ then $v_i =v_j$ and $n_i\neq n_j$ or $v_i\neq v_j$. 
In the first case, we already know by \eqref{pq1} that $G_{i,j}=o(1)$ as $\gamma\to+\infty$. 
In the second case, $\supp \chi_{v_i}\cap \supp \chi_{v_j}=\emptyset$, then $G_{i,j}=0$.
Necessarily, $\det(G)=1+o(1)$ as $\gamma \to +\infty$. In particular, $\det(G)\neq 0$ for $\gamma$ large enough, which gives us the result.
\end{proof}

\subsection{Asymptotic behavior of the first eigenvalues on polygons}\label{sectasvppol}

In this section we prove Theorem~\ref{Intro_Thasvp} for polygons with straight edges.

\begin{theorem}\label{asexp}
Let $\Omega\subset \mathbb R^2$ be a polygon with straight edges. 
For any $\epsilon \in (0,1)$ there exists $C >0$ such that for all $n\in\{1,...,\mathcal N^\oplus\}$ and for $\gamma$ large enough, 
\[
\lvert E_n(Q^\gamma)-\gamma^2E_n(T^{\oplus})\rvert \leq C e^{-(1-\epsilon)\gamma\sqrt{-1-E^{\max}}\rho}.
\]
\end{theorem}

\begin{proof}
The proof of Theorem~\ref{asexp} requires two steps. First, we prove an upper bound and a lower bound for the eigenvalues of $Q^\gamma$, using respectively the properties of the quasi-modes and a partition of unity. Secondly, to prove the exponential decay of the remainder we use a spectral approximation result. 

Recall that $\Lambda$ is the set of the eigenvalues of the operator $T^\oplus$ ordered in the increasing way and counted \emph{without} multiplicity, $K^\oplus\coloneqq \#\Lambda$ and we denote by $m_l$ the multiplicity of $\lambda_l \in \Lambda$ as an eigenvalue of $T^\oplus$, see Section~\ref{modelop}.
\begin{proposition}\label{asvp}
For any $\epsilon \in(0,1)$ there exist $C >0$ and $c>0$ such that, for all $0\leq l \leq K^\oplus$ and for $\gamma$ large enough,
\begin{align}
\label{asvp1}
E_{m_1+...+m_l}(Q^\gamma) &\leq \gamma^2\lambda_l +C \gamma^2 e^{-2(1-\epsilon)\gamma \sqrt{-1-E^{\max}}\rho},\\
\label{asvp2}
E_{m_0+...+m_l+1}(Q^\gamma) &\geq \gamma^2 \lambda_{l+1}-c,
\end{align}
with the convention $m_0=0$.
\end{proposition}

\begin{proof}
We begin proving \eqref{asvp1}.
Let $l\in\{1,...,K^\oplus\}$ be fixed.
 In the sequel we denote $d\coloneqq \displaystyle \sum _{j=1}^l m_j.$ By the min-max principle:
\[
E_d(Q^\gamma)=\inf_{\substack{F\subset D(q^\gamma)\\ \dim(F)=d}} \sup_{\substack{\psi\in F\\ \psi \neq 0}}\frac{q^{\gamma}(\psi,\psi)}{\lVert \psi\rVert^2}.
\]
We introduce 
\[
\mathcal F^\gamma\coloneqq \Span \{\tilde \phi^{\gamma,v}_n, (n,v)\in \cup_{j=1}^l \mathcal S_j\}.
\]
For simplicity we denote by $(\tilde \phi_1,...,\tilde \phi_d)$ the elements of $\{\tilde \phi^{\gamma,v}_n, (n,v)\in \displaystyle\cup_{j=1}^l \mathcal S_j\}$. By Lemma~\ref{li}, $\dim(\mathcal F ^\gamma)=d$ for $\gamma$ large enough and
\begin{align}
\label{minmax1}
E_d(Q^\gamma)\leq \sup_{\substack{ \psi \in \mathcal F^\gamma \\ \psi \neq 0}}\frac{q^{\gamma}( \psi, \psi)}{\lVert \psi\rVert^2} 
=\sup_{\substack{(c_1,...,c_d)\in \mathbb C^d\\ (c_1,...,c_d)\neq(0,...,0)}} \frac{q^{\gamma}(\sum_{j=1}^d c_j \tilde \phi_j,\sum_{j=1}^d c_j  \tilde\phi_j)}{\lVert \sum_{j=1}^d c_j \tilde\phi_j\rVert^2}.
\end{align}
Expanding the numerator we get
\[
q^\gamma(\sum_{j=1}^d c_j \tilde \phi_j,\sum_{j=1}^d c_j \tilde \phi_j)=\sum_{j=1}^l \lvert c_j\rvert^2 q^\gamma(\tilde \phi_j,\tilde \phi_j) +2\Re \sum_{j<k} c_j \overline{c_k} q^{\gamma}(\tilde \phi_j,\tilde \phi_k).
\]
We can use \eqref{pq2} and \eqref{pq3} to obtain 
\begin{align*}
q^\gamma(\sum_{j=1}^d c_j \tilde \phi_j,\sum_{j=1}^d c_j \tilde \phi_j) \leq \sum_{j=1}^l \lvert c_j \rvert^2 (\gamma^2\lambda_l&+\gamma C e^{-2(1-\epsilon)\gamma \sqrt{-1-E^{\max}}\rho} )\\
&+2 \sum_{j<k} \lvert c_j c_k\rvert \gamma C e^{-2(1-\epsilon)\gamma \sqrt{-1-E^{\max}}\rho}.
\end{align*}
As $\displaystyle \sum_{j<k} \lvert c_j c_k\rvert \leq d \displaystyle \sum_j \lvert c_j\rvert^2$, we can write 
\begin{align}
\label{major1}
q^{\gamma}(\sum_{j=1}^d c_j \tilde \phi_j,\sum_{j=1}^d c_j \tilde \phi_j) \leq \left(\gamma^2\lambda_l + \gamma C e^{-2(1-\epsilon)\gamma \sqrt{-1-E^{\max}}\rho}\right ) \sum_{j=1}^d \lvert c_j \rvert^2.
\end{align}
The denominator expands as 
\[
\lVert \sum_{j=1}^d c_j \tilde \phi_j\rVert^2=\sum_{j=1}^d \lvert c_j \rvert^2 \lVert \tilde \phi_j\rVert^2 
+ 2 \Re \sum_{j<k} c_j \overline{c_k}\langle \tilde \phi_j,\tilde \phi_k \rangle.
\]
Then, using \eqref{qm1} and \eqref{pq1} we have 
\begin{align}
\label{major2}
\left \lvert \lVert \sum_{j=1}^d c_j \tilde \phi_j \rVert^2 - \sum_{j=1}^l \lvert c_j\rvert^2\right \rvert \leq C e^{-2(1-\epsilon)\gamma \sqrt{-1-E^{\max}}\rho} \left( \sum_{j=1}^l \lvert c_j \rvert^2 \right).
\end{align}
Combining \eqref{major1} and \eqref{major2} we first get :
\begin{align*}
\frac{q^{\gamma}(\sum_{j=1}^d c_j\tilde  \phi_j,\sum_{j=1}^d c_j \tilde \phi_j)}{\lVert \sum_{j=1}^d c_j \phi_j\rVert^2} 
&\leq \left(\frac{\gamma^2\lambda_l}{\lVert \sum_{j=1}^d c_j \tilde \phi_j\rVert^2}  + \frac{\gamma C e^{-2(1-\epsilon)\gamma \sqrt{-1-E^{\max}}\rho} }{\lVert \sum_{j=1}^d c_j \tilde \phi_j\rVert^2} \right )\sum_{j=1}^d \lvert c_j\rvert^2 \\
&\leq \frac{\gamma^2\lambda_l}{1+C e^{-2(1-\epsilon)\gamma \sqrt{-1-E^{\max}}\rho} }+ \frac{\gamma C e^{-2(1-\epsilon)\gamma \sqrt{-1-E^{\max}}\rho}}{1-C e^{-2(1-\epsilon)\gamma \sqrt{-1-E^{\max}}\rho}}.
\end{align*}
Recall that there exists 
$(n,v) \in \mathcal S_l$ such that $\gamma^2\lambda_l=\gamma^2 E_n(T^1_v)$. Then, 
$-\gamma^2\lambda_l\leq -\gamma^2 \min_{v\in\mathcal V} E_1(T^1_v)$ 
and one has, for $\gamma$ large enough
\[
\frac{q^{\gamma}(\sum_{j=1}^d c_j \tilde \phi_j,\sum_{j=1}^d c_j \tilde \phi_j)}{\lVert \sum_{j=1}^d c_j \tilde \phi_j\rVert^2} \leq
\gamma^2\lambda_l + C \gamma^2 e^{-2(1-\epsilon)\gamma \sqrt{-1-E^{\max}}\rho}.
\]
This concludes the proof of \eqref{asvp1} thanks to \eqref{minmax1}. 

We now focus on the lower bound. Here $l\in\{0,.., K^\oplus\}$ and 
$d\coloneqq\displaystyle\sum_{j=0}^l m_j$.
Using the same $(\chi_v)_{v\in\mathcal V}$ as before, we define 
$
\chi_0\coloneqq 1-\sum_{v\in\mathcal V} \chi_v$ and for $v\in\mathcal V\cup\{0\}$,
\[
\tilde \chi_v(x)\coloneqq \frac{\chi_v(x)}{\sqrt{\sum_{v\in\mathcal V\cup\{0\}}\chi_v^2(x)}},\quad x\in\Omega,
\]
such that $\sum_{v\in\mathcal V\cup\{0\}} \tilde \chi_v^2= 1$ on $\Omega$.
Then, for all $\psi \in D(q^\gamma)$ we have
\begin{align}
\label{IMSpol}
q^\gamma(\psi,\psi)=\sum_{v\in\mathcal V \cup \{0\}} q^{\gamma}(\psi \tilde\chi_v,\psi\tilde \chi_v) -\sum_{v\in\mathcal V\cup \{0\}} \lVert \psi\nabla \tilde \chi_v\rVert^2_{L^2(\Omega)}.
\end{align}
Let us introduce some notation. 
Let $V(x)\coloneqq  \sum_{v\in\mathcal V\cup \{0\}} \lvert \nabla \tilde \chi_v(x)\rvert^2$, $\Omega_0\coloneqq \Omega\backslash  \bigcup_{v\in\mathcal V}\overline{ B(v,\rho)}$ 
and for $v\in\mathcal V$ we denote 
$\Omega_{v,2\rho}\coloneqq \Omega\cap B(v,2\rho)$ and 
$\Gamma_{v,2\rho}\coloneqq \partial \Omega_{v,2\rho}\backslash \partial B(v,2\rho)$.
By definition of $V$, there exists $c>0$ such that for all $x\in\Omega$,
\begin{align}
\label{majVpol}
\lVert V\rVert_{\infty} \leq c.
\end{align}
We also introduce
\[
q^{\gamma,V}_v(\psi,\psi)=\int_{\Omega_{v,2\rho}}\left(\lvert \nabla \psi \rvert^2-V(x) \lvert \psi\rvert^2\right) dx -\gamma \int_{\Gamma_{v,2\rho}} \lvert \psi \rvert^2 ds,
\]
where $D(q^{\gamma,V}_v)\coloneqq \{ \psi \in H^1(\Omega_{v,2\rho}), \psi(x)=0 \text{ for } x\in \partial \Omega_{v,2\rho} \backslash \Gamma_{v,2\rho}\}$ and 
\[
q^{\gamma,V}_0(\psi,\psi)=\int_{\Omega_0} \left( \lvert \nabla \psi\rvert^2 -V(x)\lvert \psi\rvert^2 \right) dx -\gamma\int_{\partial \Omega_0\cap \partial \Omega} \lvert \psi \rvert^2 ds,
\]
where $D(q^{\gamma,V}_0)\coloneqq \{ \psi \in H^1(\Omega_0), \psi(x)=0 \text{ for } x\in \partial \Omega_0\backslash \partial \Omega\}$.
Notice that, if 
$\psi \in \bigoplus_{v\in\mathcal V\cup \{0\}} D(q^{\gamma,V}_v)$, then $\psi \in D(q^\gamma)$ and $\lVert \psi \rVert^2_{L^2(\Omega)}= \sum_{v\in\mathcal V\cup\{0\}} \lVert \psi \tilde\chi_v \rVert^2_{L^2(\Omega)}$. By the min-max principle and \eqref{IMSpol} we can write for all $n\in\mathbb N$,
\[
E_n(Q^\gamma)\geq E_n(\bigoplus_{v\in\mathcal V\cup \{0\}} Q^{\gamma,V}_v).
\]
By definition, $\Omega_0\cap \mathcal V =\emptyset$. 
Moreover, $\Omega_0$ does not depend on $\gamma$ and we can extend it in a smooth way to obtain a domain with a Lipschitz, $C^4$ boundary which we call $\Omega^{\text{reg}}$. 
We define
\[
q^{\gamma,\text{reg}}(\psi,\psi)=\int_{\Omega^{\text{reg}}}\lvert \nabla \psi \rvert^2 dx -\gamma \int_{\partial \Omega^{\text{reg}}} \lvert \psi\rvert^2 ds,\quad \psi\in H^1(\Omega^{\text{reg}}).
\]
By \cite[Theorem 1]{Pan}, we know that there exists $ C >0$ such that, for $\gamma$ large enough,
\begin{align*}
E_1(Q^{\gamma,\text{reg}}) \geq -\gamma^2 - C\gamma.
\end{align*}
In addition, by the min-max principle and \eqref{majVpol} we also have, for all $n \in\mathbb N$, 
$
E_n(Q^{\gamma,V}_0) \geq E_n(Q^{\gamma,\text{reg}})-c.
$
Then, for $\gamma$ large enough,
\begin{align}
\label{minreg}
E_1(Q^{\gamma,V}_0)\geq -\gamma^2 -C \gamma -c.
\end{align} 
On the other hand,  extending $\psi \in D(q^{\gamma,V}_v)$ by zero and using the min-max principle and \eqref{majVpol}, one can write for all $n\leq \mathcal N_v$,
\[
E_n(Q^{\gamma,V}_v)\geq \gamma^2 E_n(T_v)-c.
\]
In particular,
\begin{align*}
E_{d+1}(\bigoplus_{v\in\mathcal V} Q^{\gamma,V}_v) \geq \gamma^2 E_{d+1}(T^{\oplus})-c =\gamma^2\lambda_{l+1} -c.
\end{align*}
Combining it with \eqref{minreg}, we finally get
$
E_{d+1}(Q^\gamma)\geq \gamma^2 \lambda_{l+1} -c,
$
since $\lambda_{l+1}<-1$. This concludes the proof of \eqref{asvp2}.
\end{proof}
This proposition tells us that the eigenvalues of $Q^\gamma$ are gathered in clusters. For each $n\in\{1,...,\mathcal N^\oplus\}$, there exists $l\in\{0,...,K^\oplus-1\}$ such that $m_0+m_1+...+m_l+1\leq n \leq m_0+...+m_{l+1}$. Then, $\lambda_{l+1}=E_n(T^\oplus)$ and 
\[
E_n(Q^\gamma)\in I_{l+1}\coloneqq (\gamma^2\lambda_{l+1}-c, \gamma^2\lambda_{l+1} +C \gamma^2 e^{-2(1-\epsilon)\gamma \sqrt{-1-E^{\max}}\rho}),
\]
for $\gamma$ large enough.
Notice that, as $\lambda_l <\lambda_{l+1}$ we have $I_l\cap I_{l+1}=\emptyset$ for large $\gamma$: the $I_l$ are disjoint sets.

In order to conclude, we can now state the spectral approximation result, which proof is recalled in Appendix~\ref{speccorol}.
\begin{proposition}\label{nvp}
Let $A$ be a self-adjoint operator acting on a Hilbert space $\mathcal H$ and $\lambda \in\mathbb R$. If there exist $\psi_1,...,\psi_n \in D(A)$  linearly independent and $\eta >0$ such that 
\[
\lVert (A-\lambda)\psi_j\rVert \leq \eta \lVert\psi_j\rVert,\quad j=1,...,n,
\]
then, 
\[
\dim \Ran P_A (\lambda- n^{3/2}\eta \sqrt{\frac{\beta_{\max}}{\beta_{\min}}},\lambda+ n^{3/2}\eta \sqrt{\frac{\beta_{\max}}{\beta_{\min}}} ) \geq n,
\]
where $\beta_{\min}$ (resp. $\beta_{\max}$) is the minimal (resp. maximal) eigenvalue of the Gramian matrix of the family $(\psi_j)_j$ and $P_A(a,b)$ stands for the spectral projection of $A$ on the interval $(a,b)$.  
In particular, if $\Specess(A)\cap (\lambda-n^{3/2}\eta \sqrt{\frac{\beta_{\max}}{\beta_{\min}}},\lambda+n^{3/2}\eta \sqrt{\frac{\beta_{\max}}{\beta_{\min}}} )=\emptyset$, there exist at least $n$ eigenvalues of $A$ in  $(\lambda-n^{3/2}\eta \sqrt{\frac{\beta_{\max}}{\beta_{\min}}},\lambda+n^{3/2}\eta \sqrt{\frac{\beta_{\max}}{\beta_{\min}}} )$.
\end{proposition}
In order to apply Proposition~\ref{nvp}, let us recall \eqref{qm2}. 
For all $1\leq l\leq K^\oplus$, and for all $(n,v)\in\mathcal S_l$, 
\[
\lVert (Q^\gamma-\gamma^2\lambda_l)\tilde \phi^{\gamma,v}_n\rVert \leq C e^{-(1-\epsilon)\gamma \sqrt{-1-E^{\max}}\rho}\lVert \tilde\phi^{\gamma,v}_n\rVert.
\]
 Let $\eta\coloneqq C e^{-(1-\epsilon)\gamma \sqrt{-1-E^{\max}}\rho}$. Then by Proposition~\ref{nvp},
\[
\dim \Ran P_{Q^\gamma} (\gamma^2\lambda_l-m_l^{3/2}\eta \sqrt{\frac{\beta_{\max}}{\beta_{\min}}},\gamma^2\lambda_l+m_l^{3/2}\eta \sqrt{\frac{\beta_{\max}}{\beta_{\min}}}) \geq m_l,
\]
for $\gamma$ large enough as $(\tilde \phi^{\gamma,v}_n)_{(n,v)\in S_l}$ is linearly independent by Lemma~\ref{li}. Notice that, by \eqref{qm1} and \eqref{pq1}, if we denote by $(\beta_i)_{i=1}^{m_l}$ the eigenvalues of the Gramian matrix of $(\tilde \phi^{\gamma,v}_n)_{(n,v)\in S_l}$ with $\beta_1=\beta_{\min}$ and $\beta_{m_l}=\beta_{\max}$, then for all $1\leq i \leq m_l$,
\begin{align}
\label{closevp}
\lvert \beta_i -1\rvert \leq C e^{-2(1-\epsilon)\gamma \sqrt{-1-E^{\max}}\rho},
\end{align}
and $\sqrt{\displaystyle \frac{\beta_{\max}}{\beta_{\min}}}=O(1)$, as $\gamma\to+\infty$. 
Moreover, as $\Specess(Q^\gamma)=\emptyset$ the operator $Q^\gamma$ admits at least $m_l$ eigenvalues in 
$\tilde I_l\coloneqq (\gamma^2\lambda_l-m_l^{3/2}\eta \sqrt{\frac{\beta_{\max}}{\beta_{\min}}},\gamma^2\lambda_l+{m_l}^{3/2}\eta \sqrt{\frac{\beta_{\max}}{\beta_{\min}}})$. But, as $\tilde I_l\cap I_{l+1} =\emptyset$, $\tilde I_l\cap \tilde I_{l-1}=\emptyset $ and $\tilde I_l \cap I_l \neq \emptyset$, we can conclude by the previous corollary that $Q^\gamma$ admits exactly $m_l$ eigenvalues in $\tilde I_l$ and these eigenvalues correspond to the $l$th cluster mentioned above, which concludes the proof.

\end{proof}

\subsection{Approximation of eigenspaces}\label{appeigen}

We are now going to prove that the corresponding eigenfunctions are, in a sense, exponentially close to linear combinations of the quasi-modes $(\tilde \phi^{\gamma,v}_j)$. Let us first introduce the distance between subspaces of a Hilbert space which main properties are gathered in \cite{HS}. Let $E$ and $F$ be closed subspaces of a Hilbert space $\mathcal H$. We define the non-symmetric distance between $E$ and $F$ as:
\[
d(E,F)\coloneqq \inf_{x \in E\backslash\{0\}}\frac{\dist(x,F)}{\lVert x \rVert},
\]
where $\dist(x,F)=\inf_{y\in F} \lVert x-y \rVert$. The following theorem will be the main argument in the proof of closeness of quasi-modes. 
\begin{theorem}[\cite{HS}]\label{THeigen}
Let $A$ be a self-adjoint operator in a Hilbert space $\mathcal H$. Let $I\subset \mathbb R$ be a compact interval, $\psi_1,...,\psi_n \in D(A)$ be linearly independent and $\mu_1,...,\mu_n \in I$. Suppose that there exists $\eta >0$ such that
\[
A\psi_j=\mu_j\psi_j+r_j \quad \text{with}\quad \lVert r_j\rVert \leq \eta,
\]
and $a>0$ such that $\Spec(A)\cap \left((I+B(0,2a))\backslash I\right)=\emptyset$. Then, if $F$ is the space spanned by $\psi_1,...,\psi_n$ and $E\coloneqq \Ran P_A(I)$, 
\[
\dist(F,E) \leq \frac{\sqrt{n} \eta}{a \sqrt{ \beta_{\min}}},
\]
where $\beta_{\min}$ stands for the minimal eigenvalue of the Gramian matrix of $(\psi_j)_j$.
\end{theorem}

We introduce
\[
\mathcal F^\gamma_l\coloneqq \Span\{\tilde \phi^{\gamma,v}_j, (j,v)\in \mathcal S_l\}.
\]
We can now state the theorem on the eigenspaces.

\begin{theorem}\label{eigenspace}
For any $\epsilon \in(0,1)$, there exists $C>0$ such that, for all $l\in\{1,...,K^\oplus\}$ and for $\gamma$ large enough,
\[
\dist(\mathcal F^\gamma_l, \mathcal E^\gamma_l) \leq C e^{-(1-\epsilon)\gamma \sqrt{-1-E^{\max}}\rho},
\]
where $\mathcal E^\gamma_l\coloneqq \Ran P_{Q^\gamma}(\gamma^2\lambda_l -C e^{-(1-\epsilon)\gamma \sqrt{-1-E^{\max}}\rho},\gamma^2\lambda_l +C e^{-(1-\epsilon)\gamma \sqrt{-1-E^{\max}}\rho} )$.
\end{theorem}

\begin{proof}
By \eqref{qm2_1}, there exists $r_{n,v} \in L^2(\Omega)$ such that 
\[
Q^\gamma\tilde \phi^{\gamma,v}_n=\gamma^2\lambda_l\tilde  \phi^{\gamma,v}_n + r_{n,v},
\]
and $\lVert r_{n,v} \rVert \leq C e^{-(1-\epsilon)\gamma \sqrt{-1-E^{\max}}\rho}$, for any $(n,v)\in \mathcal S_l$. 
Let $a\coloneqq \min(\displaystyle \frac{\lambda_{l+1}-\lambda_l}{4},\displaystyle \frac{\lambda_l -\lambda_{l-1}}{4} )$.
By Theorem~\ref{asexp}, we know that if we denote by $\tilde I_l \coloneqq (\gamma^2\lambda_l - C e^{-(1-\epsilon)\gamma \sqrt{-1-E^{\max}}\rho}, \gamma^2\lambda_l + C e^{-(1-\epsilon)\gamma \sqrt{-1-E^{\max}}\rho})$, then $\Spec(Q^\gamma)\cap \left ( (\tilde I_l +B(0,2a)) \backslash \tilde I_l\right )=\emptyset$. Then by Theorem~\ref{THeigen},
\[
\dist(\mathcal F^\gamma_l,\mathcal E^\gamma_l)\leq \frac{\sqrt{m_l} C e^{-(1-\epsilon)\gamma \sqrt{-1-E^{\max}}\rho}}{a \sqrt{\beta_{\min}}}.
\]
As by \eqref{closevp}, $\lvert \beta_{\min}-1\rvert \leq C e^{-2(1-\epsilon)\gamma \sqrt{-1-E^{\max}}\rho}$, we can conclude the proof.
\end{proof}

\section{Robin Laplacian on curvilinear polygons}\label{curvpol}

In the following, $\Omega$ is a general curvilinear polygon in the sense of Definition~\ref{defcurvpol}. We still denote $Q^\gamma\coloneqq Q^\gamma_{\Omega}$.

In the next section, we introduce some test-functions which will play the role of the quasi-modes we used in the proofs for polygons with straight edges.

\subsection{Description and properties of weak quasi-modes}\label{descriptqmcurv}

Recall that $\psi^{\gamma,v}_n$ are the eigenfunctions of the Robin Laplacian acting on the infinite sector $U_v$ introduced in Section~\ref{modelop}.
For $v\in\mathcal V$ and $1\leq n\leq \mathcal N_v$ we set 
\[
\phi^{\gamma,v}_n\coloneqq \psi^{\gamma,v}_n \circ F_v,
\]
where $F_v$ is the $C^2$-diffeomorphism maping $\overline{B(v, r_v)\cap \Omega}$ onto $\overline{B(0,r_v)\cap U_v}$, see Section~\ref{sectdefcurvpol}. 
Then, $\phi^{\gamma,v}_n \in H^1(\Omega\cap B(v,r_v))$.
Let $\varphi \in C^\infty(\mathbb R_+)$ be a smooth cut-off function satisfying $0\leq \varphi\leq 1$, $\varphi(t)=1$ if $0\leq t\leq 1$, and $\varphi=0$ if $t\geq 2$. We introduce the smooth radial cut-off function $\chi^\gamma_v$ defined as follows:
\begin{align}
\label{chigamma}
\chi^\gamma_v(x)= \varphi(\lvert x-v\rvert \gamma^\beta), \quad x\in \Omega,
\end{align}
where $\beta \in (1/2,1)$ will be chosen later. Notice that for $\gamma$ large enough $\supp \chi^\gamma_v \subset B(v,r_v)$ and $\supp \chi^\gamma_v \cap \supp \chi^\gamma_{v'} =\emptyset$ for $v\neq v'$. In the following, $\gamma$ will be supposed large enough such that these conditions are satisfied.
We define
\[
\tilde \phi^{\gamma,v}_n\coloneqq \phi^{\gamma,v}_n \chi^\gamma_v\quad \text{on}\quad \Omega.
\]
Defined as above, $\tilde \phi^{\gamma,v}_n \in D(q^\gamma)\coloneqq H^1(\Omega)$ but $\tilde \phi^{\gamma,v}_n$ does not belong to the domain of the operator $Q^\gamma$ as it does not satisfy the Robin boundary condition: we call it a \emph{weak quasi-mode}. 

In order to list some properties of the weak quasi-modes we will need some additional results and notations.

\noindent\textbf{Notation.} 
\begin{itemize}
\item[(a)] If $D \subset \mathbb R^2$, we set $D_{y,r}\coloneqq D\cap B(y,r)$. 
\item[(b)] Let $g: \mathbb R^2\to \mathbb R^2$  be $C^1$. We denote by $Jg$ the determinant of the Jacobian matrix of $g$, namely
 \[
 Jg\coloneqq \det(\nabla g).
\]
\end{itemize}

\begin{lemma}\label{Lemmacov}
Let $v \in \mathcal V$, $\psi \in H^1(U_v)$ and $\phi\coloneqq \psi\circ F_v $. Then, $\phi\in H^1(\Omega_{v,2\gamma^{-\beta}})$ for $\gamma$ large enough and,
\begin{align}
\label{est1}
\left \lvert \int_{\Omega_{v,2\gamma^{-\beta}}} \lvert \phi(x)\rvert^2 dx - \int_{(U_v)_{0,2\gamma^{-\beta}}} \lvert \psi(u) \rvert^2 du\right \rvert \leq C\gamma^{-\beta} \int_{(U_v)_{0,2\gamma^{-\beta}}} \lvert \psi(u)\rvert^2 du,
\end{align}
\begin{align}
\label{est2}
\left \lvert \int_{\Omega_{v,2\gamma^{-\beta}}} \lvert \nabla \phi(x)\rvert^2 dx -\int_{(U_v)_{0,2\gamma^{-\beta}}} \lvert \nabla \psi(u)\rvert^2 du\right \rvert \leq C\gamma^{-\beta}  \int_{(U_v)_{0,2\gamma^{-\beta}}} \lvert \nabla \psi(u)\rvert^2 du, 
\end{align}
\begin{align}
\label{est3}
\left \lvert \int_{\Gamma_{v,2\gamma^{-\beta}}} \lvert \phi(s)\rvert^2 ds -\int_{\Sigma_{v,2\gamma^{-\beta}}}\lvert \psi(s)\rvert^2 ds \right \rvert \leq C \gamma^{-\beta} \int_{\Sigma_{v,2\gamma^{-\beta}}} \lvert \psi(s)\vert^2 ds,
\end{align}
where $\Gamma_{v,r}\coloneqq \partial \Omega_{v,r} \backslash \partial B(v,r)$ and 
$\Sigma_{v,r}\coloneqq\partial (U_v)_{0,r}\backslash \partial B(0,r)$.
\end{lemma}

\begin{proof}
 We first want to estimate the $L^2$-norm of $\phi$. By change of variables,
\[
\int_{\Omega_{v,2\gamma^{-\beta}}} \lvert \phi(x)\rvert^2 dx = \int_{(U_v)_{0,2\gamma^{-\beta}}} \lvert \psi(u)\vert^2 \lvert JF^{-1}_v(u)\rvert du.
\]
As $F^{-1}_v$ is $C^2$, then $u\mapsto JF^{-1}_v(u)$ is also $C^2$ and by Taylor-Lagrange for all $u\in B(0,2\gamma^{-\beta})$ we have
\begin{align}
\label{TL1}
\lvert JF^{-1}(u)-1\rvert \leq C\gamma^{-\beta}.
\end{align}
Writing
\[
\left \lvert \int_{\Omega_{v,2\gamma^{-\beta}}} \lvert \phi(x)\rvert^2 dx - \int_{(U_v)_{0,2\gamma^{-\beta}}} \lvert \psi(u) \rvert^2 du\right \rvert \leq \int_{(U_v)_{0,2\gamma^{-\beta}}} \lvert \psi(u)\rvert^2 \left\lvert\lvert JF^{-1}_v\rvert -1\right\rvert du,
\]
finishes the proof of \eqref{est1}. 

We now estimate the $L^2$-norm of $\nabla \phi$. 
By definition of $\phi$ we have $\nabla \phi(x)= \nabla \psi(F_v(x))\nabla F_v(x)$. Then,
\[
\int_{\Omega_{v,2\gamma^{-\beta}}} \lvert \nabla \phi(x)\rvert^2 dx =\int_{(U_v)_{0,2\gamma^{-\beta}}} \lvert \nabla \psi(u) \nabla F_v(F^{-1}_v(u))\rvert^2 \lvert JF^{-1}_v(u)\rvert du.
\]
Using again Taylor-Lagrange, we know that for all $u\in B(0,2\gamma^{-\beta})$,
\begin{align}
\label{TL2}
\lvert \nabla F_v(F^{-1}_v(u))-I_2\rvert \leq C \gamma^{-\beta}.
\end{align}
We denote $I\coloneqq\Big \lvert \displaystyle \int_{\Omega_{v,2\gamma^{-\beta}}} \lvert \nabla \phi(x)\rvert^2 dx -\displaystyle\int_{(U_v)_{0,2\gamma^{-\beta}}} \lvert \nabla \psi(u)\rvert^2 du\Big  \rvert$. Then,
\begin{align*}
 I
\leq \Big \lvert \int_{(U_v)_{0,2\gamma^{-\beta}}} \Big ( \lvert \nabla \psi(u) \nabla F_v(F^{-1}_v(u))\rvert^2 -&\lvert \nabla \psi(u)\rvert^2 \Big ) \lvert JF^{-1}(u)\rvert du \Big \lvert \\
& +\left \lvert \int_{(U_v)_{0,2\gamma^{-\beta}}} \lvert \nabla \psi(u)\rvert^2 (\lvert JF_v^{-1}(u)\rvert-1) du \right \rvert.
\end{align*}
First we have using \eqref{TL2}, 
\begin{align*}
\left \lvert \lvert \nabla \psi(u) \nabla F_v(F^{-1}_v(u))\rvert^2-\lvert \nabla \psi(u)\rvert^2 \right \rvert \leq C \gamma^{-\beta}  \lvert \nabla \psi(u)\rvert^2, 
\end{align*}
and using \eqref{TL1} we also get 
\begin{align*}
\left \lvert \lvert \nabla \psi(u) \nabla F_v(F^{-1}_v(u))\rvert^2-\lvert \nabla \psi(u)\rvert^2 \right \rvert \lvert JF^{-1}(u)\rvert \leq  C\gamma^{-\beta} \lvert \nabla \psi(u)\rvert^2,
\end{align*}
which gives us the upper bound for the first term of $I$.
We can use again \eqref{TL1} for the second term and we get \eqref{est2}.

 We are now interested in the integral along the boundary. Recall that by assumption $\partial\Omega=\bigcup_{k=1}^M \overline{\Gamma_k}$. Without loss of generality, we suppose that two components $\Gamma_j$, $\Gamma_k$ intersect iff $k=j+1$ or $k=j-1$. Then, there exists $j\in\{1,...,M\}$ such that $v=\overline\Gamma_j\cap \overline \Gamma_{j+1}$
 and $\Gamma_{v,2\gamma^{-\beta}}$ is composed by two connected $C^4$ components
\begin{align*}
\Gamma^{j}_{v,2\gamma^{-\beta}}\coloneqq\Gamma_j\cap B(v,2\gamma^{-\beta}), \text{ and }
\Gamma^{j+1}_{v,2\gamma^{-\beta}}\coloneqq\Gamma_{j+1}\cap B(v,2\gamma^{-\beta}),
\end{align*}
such that $ \overline{\Gamma^{j}_{v,2\gamma^{-\beta}}}\cup \overline{\Gamma^{j+1}_{v,2\gamma^{-\beta}}}=\overline{\Gamma_{v,2\gamma^{-\beta}}}$. We can introduce 
\begin{align*}
\overline{\Sigma^j_{v,2\gamma^{-\beta}}}\coloneqq F_v(\overline{\Gamma^j_{v,2\gamma^{-\beta}}}), \text{ and }
\overline{\Sigma^{j+1}_{v,2\gamma^{-\beta}}}\coloneqq F_v(\overline{\Gamma^{j+1}_{v,2\gamma^{-\beta}}}),
\end{align*}
such that $\overline{\Sigma_{v,2\gamma^{-\beta}}}=\overline{\Sigma^j_{v,2\gamma^{-\beta}}}\cup \overline{\Sigma^{j+1}_{v,2\gamma^{-\beta}}}$. Notice that $\Sigma^l_{v,2\gamma^{-\beta}}$, $l=j, j+1$, is simply $\{t\in(0,2\gamma^{-\beta}\cos(\alpha_v)), (t,\pm \tan\alpha_v t)\}$. Thus,
\begin{align*}
\int_{\Gamma_{v,2\gamma^{-\beta}}} \lvert \phi(s)\rvert^2 ds =\sum_{l=j,j+1} \int_{\Gamma^l_{v,2\gamma^{-\beta}}} \lvert \phi(s)\rvert ^2 ds 
= \sum_{l=j,j+1} \int_0^{2\gamma^{-\beta}\cos(\alpha_v)} \lvert \psi(t,\pm\tan\alpha_vt)\rvert^2 \rho^l(t) dt,
\end{align*}
where 
\begin{align*}
\rho^l(t)\coloneqq\displaystyle\sqrt{\det \left((\displaystyle\frac{d}{dt}(F^{-1}_v(t,\pm \tan\alpha_v t)))^T(\displaystyle\frac{d}{dt}(F^{-1}_v(t,\pm \tan\alpha_v t)))\right)}=\left \lvert \nabla F^{-1}_v(t,\pm \tan\alpha_v t) T^l \right \rvert,
\end{align*}
with $T^l= \begin{pmatrix} 1 \\ \pm \tan \alpha_v \end{pmatrix}$. Finally we can write 
\begin{align*}
\int_{\Gamma_{v,2\gamma^{-\beta}}} \lvert \phi(s)\rvert^2 ds &=\sum_{l=j,j+1} \int_{\Sigma^l_{v,2\gamma^{-\beta}}} \lvert \psi(s)\rvert^2 \lvert \nabla F^{-1}_v (s) T^l \rvert \cos(\alpha_v) ds\\
& =\int_{\Sigma_{v,2\gamma^{-\beta}}} \lvert \psi(s)\rvert ^2 \lvert \nabla F^{-1}_v (s) T(s) \rvert \cos(\alpha_v) ds,
\end{align*}
where $T(s)=T^l$, as $s \in \Sigma^l_{v,2\gamma^{-\beta}}$, $l=j,j+1$. 
As $t\mapsto \nabla F^{-1}_v(t,\pm\tan\alpha_v t)$ is $C^1$ on $(0,2\gamma^{-\beta}\cos(\alpha_v))$, by Taylor-Lagrange, for all $t\in (0,2\gamma^{-\beta}\cos(\alpha_v))$,
\begin{align}
\label{TL3}
\lvert \nabla F^{-1}_v(t,\pm\tan\alpha_v t)-I_2\rvert \leq C\gamma^{-\beta}.
\end{align}
Then, $\lvert \rho^l(t)-\displaystyle \frac{1}{\cos\alpha_v}\vert \leq \displaystyle\frac{C\gamma^{-\beta}}{\cos\alpha_v}$, as $\lvert T^l\rvert=\cos^{-1}\alpha_v$. Finally,
\[
\left \lvert \int_{\Gamma_{v,2\gamma^{-\beta}}}\lvert \phi(s)\rvert^2 ds -\int_{\Sigma_{v,2\gamma^{-\beta}}} \lvert \psi(s)\rvert2 ds\right\rvert \leq \int_{\Sigma_{v,2\gamma^{-\beta}}} \lvert \psi(s)\rvert^2 \left \lvert \lvert \nabla F_v^{-1}(s)T(s)\rvert \cos(\alpha_v)-1\right \rvert ds,
\]
which concludes the proof of \eqref{est3}.
\end{proof}

We can now summarize some properties of the weak quasi-modes. The ideas are the same as the ones for polygons with straight edges and the following results are based on a decay property of eigenfunctions of the Robin Laplacian defined on infinite sectors recalled in Theorem~\ref{Agmon}.

\begin{proposition}\label{proppropqm}
Let $\epsilon \in(0,1)$ and $\beta\in(\displaystyle\frac{1}{2},1)$. For all $i,j\in\{1,...,\mathcal N_v\}$ and for $\gamma$ large enough we have 

\begin{align}
\label{prop1bis}
\lvert \langle \tilde\phi^{\gamma,v}_i,\tilde\phi^{\gamma,v}_j\rangle_{L^2(\Omega)}-\delta_{i,j}\rvert &\leq C \gamma^{-\beta}, \\
\label{prop2}
\lvert q^\gamma(\tilde \phi^{\gamma,v}_i,\tilde \phi^{\gamma,v}_j)- \delta_{i,j} \gamma^2 E_i(T_v)\rvert &\leq C\gamma^{2-\beta}.
\end{align}
\end{proposition}

\begin{proof}
Let us first estimate the $L^2$-norm of $\tilde \phi^{\gamma,v}_i$. We have immediately
\begin{align*}
\lVert \tilde \phi^{\gamma,v}_i\rVert^2_{L^2(\Omega)} =\int_{\Omega_{v,2\gamma^{-\beta}}} \lvert \chi^\gamma_v\rvert^2 \lvert \phi^{\gamma,v}_i\rvert^2 dx 
\leq \int_{(U_v)_{0,2\gamma^{-\beta}}} \lvert \psi_i^{\gamma,v}\rvert^2 \lvert JF^{-1}_v\rvert du.
\end{align*}
We conclude thanks to \eqref{TL1}.
For the lower bound we remark that
\begin{align*}
\lVert \tilde \phi^{\gamma,v}_i\rVert^2_{L^2(\Omega)} \geq \int_{\Omega_{v,\gamma^{-\beta}}} \lvert \phi^{\gamma,v}_i\rvert^2 dx 
=\int_{(U_v)_{0,\gamma^{-\beta}}} \lvert \psi^{\gamma,v}_i\rvert^2 \lvert JF_v^{-1} \rvert du.
\end{align*}
We can use again \eqref{TL1} to obtain
\[
\lVert \tilde \phi^{\gamma,v}_i\rVert^2_{L^2(\Omega)} \geq \int_{(U_v)_{0,\gamma^{-\beta}}} \lvert \psi^{\gamma,v}_i\rvert^2 du -C\gamma^{-\beta}.
\]
Writing 
\[
\displaystyle \int_{(U_v)_{0,\gamma^{-\beta}}} \lvert \psi^{\gamma,v}_i \rvert^2 du = \int_{U_v} \lvert \psi ^{\gamma,v}_i \rvert^2 du -\int_{U_v\backslash B(0,\gamma^{-\beta})} \lvert \psi^{\gamma,v}_i \rvert^2 du,
\]
and using the estimate of Theorem~\ref{Agmon} to get a lower bound for the second term permits us to conclude the proof of \eqref{prop1bis} when $i=j$.

Now, if $i\neq j$,
\[
\left \lvert \langle \tilde\phi^{\gamma,v}_i,\tilde\phi^{\gamma,v}_j\rangle -\int_{(U_v)_{0,2\gamma^{-\beta}}} \psi^{\gamma,v}_i\overline{\psi^{\gamma,v}_j}du\right\rvert =
\left \lvert \int_{(U_v)_{0,2\gamma^{-\beta}}}  \psi^{\gamma,v}_i \overline{\psi^{\gamma,v}_j} \left(( \chi^\gamma_v\circ F^{-1}_v(u))^2\lvert JF^{-1}_v(u)\rvert-1\right ) du \right \rvert.
\]
By Taylor-Lagrange, for all $u\in B(0,2\gamma^{-\beta})$,
\begin{align}
\label{TL4}
\left\lvert (\chi^\gamma_v\circ F^{-1}_v(u))^2 JF^{-1}_v(u) -1\right \rvert \leq C \gamma^{-\beta}.
\end{align}
Then we get, using the Cauchy-Schwarz inequality,
\begin{align*}
\left \lvert \langle \tilde\phi^{\gamma,v}_i,\tilde\phi^{\gamma,v}_j\rangle-\int_{(U_v)_{0,2\gamma^{-\beta}}}  \psi^{\gamma,v}_i \overline{\psi^{\gamma,v}_j} du \right \rvert 
\leq C \gamma^{-\beta}.
\end{align*}
We have now to estimate $\displaystyle \int_{(U_v)_{0,2\gamma^{-\beta}}}  \psi^{\gamma,v}_i \overline{\psi^{\gamma,v}_j} du $ more precisely. As $(\psi^{\gamma,v}_n)_{n\leq N_v}$ is orthonormal we can write 
\[ \int_{(U_v)_{0,2\gamma^{-\beta}}}  \psi^{\gamma,v}_i \overline{\psi^{\gamma,v}_j} du = -\int_{U_v\backslash B(0,2\gamma^{-\beta})}  \psi^{\gamma,v}_i \overline{\psi^{\gamma,v}_j}du.
\]
We then obtain, using again Cauchy-Schwarz and Theorem~\ref{Agmon},
\begin{align*}
\left \lvert \int_{(U_v)_{0,2\gamma^{-\beta}}}  \psi^{\gamma,v}_i \overline{\psi^{\gamma,v}_j} du\right \rvert 
\leq \int_{U_v \backslash B(0,\gamma^{-\beta})} \lvert \psi^{\gamma,v}_i \psi^{\gamma,v}_j\rvert du
\leq C e^{-2(1-\epsilon) \gamma^{1-\beta }\sqrt{-1-E^{\max}}},
\end{align*}
which ends the proof of \eqref{prop1bis}.

Let us focus of \eqref{prop2}. We first expand $q^\gamma(\tilde\phi^{\gamma,v}_i,\tilde \phi^{\gamma,v}_i )$:
\begin{multline*}
q^\gamma(\tilde\phi^{\gamma,v}_i,\tilde \phi^{\gamma,v}_i )=\int_{\Omega} \lvert \chi^\gamma_v\rvert^2 \lvert \nabla \phi^{\gamma,v}_i \rvert^2 dx -\gamma \int_{\partial \Omega} \lvert \chi^\gamma_v\rvert^2 \lvert \phi^{\gamma,v}_i\rvert^2 ds \\
+ \int_{\Omega} \lvert \nabla \chi^\gamma_v\rvert^2 \lvert \phi^{\gamma,v}_i \rvert^2 dx + 2\Re \int_{\Omega} \chi^\gamma_v \nabla \phi^{\gamma,v}_i.\overline{\nabla \chi^\gamma_v} \overline{\phi^{\gamma,v}_i}dx.
\end{multline*}
Notice that $\supp \nabla \chi^\gamma_v \subset A_v\coloneqq B(v,2\gamma^{-\beta})\backslash B(v,\gamma^{-\beta})$ and by definition of $\chi_v^\gamma$, $\lvert \nabla \chi^\gamma_v \rvert^2 \leq \gamma^{2\beta} \lVert \varphi' \rVert^2_\infty$. Then using \eqref{TL1}  we can write
\begin{align*}
\left \lvert \int_{\Omega} \lvert \nabla \chi^\gamma_v \rvert^2 \lvert \phi^{\gamma,v}_i\rvert^2 dx \right \rvert 
\leq \gamma^{2\beta} (1+C\gamma^{-\beta}) \lVert \varphi'\rVert^2_\infty \int_{U_v \backslash B(0,\gamma^{-\beta})} \lvert \psi^{\gamma,v}_i\rvert^2 du.
\end{align*}
We use Theorem~\ref{Agmon} to obtain
\begin{align}
\label{cross1}
\left \lvert \int_{\Omega} \lvert \nabla \chi^\gamma_v \rvert^2 \lvert \phi^{\gamma,v}_i\rvert^2 dx \right \rvert
\leq \gamma^{2\beta} (1+C \gamma^{-\beta}) \lVert\varphi'\rVert^2_\infty C e^{-2(1-\epsilon) \gamma^{1-\beta}\sqrt{-1-E^{\max}}}.
\end{align}
We can obtain the same kind of upper bound for the cross-term using Cauchy-Schwarz, 
\begin{align*}
\left \lvert \int_{\Omega} \chi^\gamma_v \nabla \phi^{\gamma,v}_i . \overline{\nabla \chi^\gamma_v}\overline {\phi^{\gamma,v}_i}dx \right \rvert 
&\leq \left(\int_{\Omega\cap A_v} \lvert \chi^{\gamma}_v\rvert^2 \lvert \nabla \phi^{\gamma,v}_i\rvert^2 dx \right)^{\frac{1}{2}} \left( \int_{\Omega\cap A_v} \lvert \nabla \chi^{\gamma}_v\rvert^2 \lvert \phi^{\gamma,v}_i\rvert^2 dx \right)^{\frac{1}{2}}, 
\end{align*}
the estimates \eqref{TL1}, \eqref{est2} and Theorem~\ref{Agmon},
\begin{align}
\label{cross2}
\left \lvert \int_{\Omega} \chi^\gamma_v \nabla \phi^{\gamma,v}_i . \overline{\nabla \chi^\gamma_v}\overline {\phi^{\gamma,v}_i}dx \right \rvert
\leq
\gamma^{\beta}\lVert \varphi'\Vert_\infty (1+C\gamma^{-\beta}) C e^{-2(1-\epsilon) \gamma^{1-\beta}\sqrt{-1-E^{\max}}}.
\end{align}
Combining \eqref{cross1} and \eqref{cross2} we get 
\begin{align}
\label{cross3}
\left \lvert \int_{\Omega} \lvert \nabla \chi^\gamma_v\rvert^2 \lvert \phi^{\gamma,v}_i \rvert^2 dx 
+ 2\Re \int_{\Omega} \chi^\gamma_v \nabla \phi^{\gamma,v}_i.\overline{\nabla \chi^\gamma_v} \overline{\phi^{\gamma,v}_i}dx \right \rvert 
\leq \gamma^{2\beta} C e^{-2(1-\epsilon)\gamma^{1-\beta} \sqrt{-1-E^{\max}}}.
\end{align}
Let us now focus on the main term. First, using \eqref{est2} we obtain 
\[
\int_{\Omega}\lvert \chi^\gamma_v\rvert^2 \lvert \nabla \phi^{\gamma,v}_i\rvert^2 dx \leq \int_{U_v} \lvert \nabla \psi^{\gamma,v}_i\rvert^2 du + C\gamma^{2-\beta} \lVert \nabla \psi^{1,v}_i\rVert^2_{L^2(U_v)},
\]
and
\begin{multline*}
\int_{\Omega}\lvert \chi^\gamma_v\rvert^2 \lvert \nabla \phi^{\gamma,v}_i\rvert^2 dx 
\geq \int_{U_v} \lvert \nabla \psi^{\gamma,v}_i\rvert^2 du 
-C\gamma^{2-\beta} \lVert \nabla\psi^{1,v}_i\rVert^2_{L^2(U_v)} \\
- (1-C\gamma^{-\beta})C e^{-2(1-\epsilon) \gamma^{1-\beta}\sqrt{-1-E^{\max}}}.
\end{multline*}
 For the boundary term, we use \eqref{est3} to obtain 
\[
\int_{\partial \Omega} \lvert \chi^\gamma_v\rvert^2 \lvert \phi^{\gamma,v}_i\rvert^2 ds \leq  \int_{\partial U_v} \lvert \psi ^{\gamma,v}_i \rvert^2 ds + C\gamma^{1-\beta} \int_{\partial U_v} \lvert \psi^{1,v}_i\rvert^2 ds,
\]
and,
\begin{align*}
\int_{\partial\Omega} \lvert \chi^\gamma_v\rvert^2 \lvert \phi^{\gamma,v}_i\rvert^2 ds 
\geq \int_{\partial U_v } \lvert\psi^{\gamma,v}_i\rvert^2 ds
 -C\gamma^{1-\beta} \int_{\partial U_v}\lvert \psi^{1,v}_i\rvert^2 ds 
-(1-C\gamma^{-\beta}) \int_{\partial U_v\backslash \Sigma_{v,\gamma^{-\beta}}} \lvert \psi^{\gamma,v}_i\rvert^2 ds.
\end{align*}
As  $U_v \backslash B(v,\gamma^{-\beta})$ is a Lipschitz domain, there exists a constant $K$ such that
\[
\lVert  \psi^{\gamma,v}_i\rVert^2_{L^2(\partial U_v\backslash \Sigma_{v,\gamma^{-\beta}} )} \leq \gamma^{\beta}K \lVert  \psi^{\gamma,v}_i \rVert^2_{H^1(U_v\backslash B(v,\gamma^{-\beta}))},
\] 
for all $v\in\mathcal V$ and $i\leq \mathcal N_v$.
Then we can use Theorem~\ref{Agmon} to get
\begin{multline*}
\int_{\partial\Omega} \lvert \chi^\gamma_v\rvert^2 \lvert \phi^{\gamma,v}_i\rvert^2 ds 
\geq \int_{\partial U_v } \lvert\psi^{\gamma,v}_i\rvert^2 ds 
-C\gamma^{1-\beta} \int_{\partial U_v}\lvert \psi^{1,v}_i\rvert^2 ds 
-C \gamma^{\beta} e^{-2(1-\epsilon)\gamma^{1-\beta}\sqrt{-1-E^{\max}}}.
\end{multline*}
This concludes the proof of \eqref{prop2} when $i=j$ as $t^{\gamma,\alpha_v}(\psi^{\gamma,v}_i,\psi^{\gamma,v}_i ) =\gamma^2 E_i(T_v)$. 

Let  $i\neq j$. We can write 
\[
q^\gamma(\tilde \phi^{\gamma,v}_i,\tilde \phi^{\gamma,v}_j) = \int_{\Omega} \lvert \chi^\gamma_v \rvert^2  \nabla \phi^{\gamma,v}_i \overline{\nabla \phi^{\gamma,v}_j} dx -\gamma \int_{\partial \Omega} \lvert \chi^{\gamma}_v \rvert^2 \phi^{\gamma,v}_i \overline{\phi^{\gamma,v}_j} ds + I(\gamma),
\]
where
\[
I(\gamma)\coloneqq \int_{\Omega} \left ( \chi^\gamma_v \nabla \phi^{\gamma,v}_i \overline{\nabla \chi^\gamma_v \phi^{\gamma,v}_j}dx +\int_{\Omega} \phi^{\gamma,v}_i \nabla \chi^{\gamma}_v \overline{\nabla \phi^{\gamma,v}_j \chi^\gamma_v}+ \int_{\Omega} \lvert \nabla \chi^\gamma_v\rvert^2 \phi^{\gamma,v}_i \overline{\phi^{\gamma,v}_j}\right )dx
\]
We first estimate $I(\gamma)$ using the same tools as before. We have
\begin{align*}
\left \lvert \int_{\Omega} \lvert \nabla \chi^\gamma_v\rvert^2 \phi^{\gamma,v}_i \overline{\phi^{\gamma,v}_j} dx  \right \rvert 
\leq \gamma^{2\beta} \lVert \varphi'\rVert^2_\infty (1+C\gamma^{-\beta}) C e^{-2(1-\epsilon) \gamma^{1-\beta} \sqrt{-1-E^{\max}}}.
\end{align*}
Notice that the other terms in the brackets are symmetric with respect to $i$ and $j$, it is then sufficient to estimate one of them
\begin{align*}
\left \lvert  \int_{\Omega} \chi^{\gamma}_v \nabla \phi^{\gamma,v}_i \overline{\nabla \chi^{\gamma}_v \phi^{\gamma,v}_j} dx \right \rvert 
\leq \gamma^{\beta} \lVert \varphi'\rVert_\infty \left( (1+C\gamma^{-\beta})   
(1+ C\gamma^{-\beta})\right)^{\frac{1}{2}} C e^{-2(1-\epsilon) 
\gamma^{1-\beta} \sqrt{-1-E^{\max}}}.
\end{align*}
Then,
\begin{align}
\label{crossbis}
\lvert I(\gamma)\rvert \leq \gamma^{2\beta} C e^{-2(1-\epsilon) 
\gamma^{1-\beta} \sqrt{-1-E^{\max}}}.
\end{align}
Let us now focus on the main term. By Taylor-Lagrange, for all $u\in B(0,2\gamma^{-\beta})$ and for all $t\in (0,2\gamma^{-\beta})$ we have 
\begin{align}
\label{TL5}
\left \lvert \right (\chi^\gamma_v\circ F^{-1}_v(u))^2 JF^{-1}_v(u) \nabla F_v(F^{-1}_v(u)) -I_2 \rvert &\leq C \gamma^{-\beta},\\
\label{TL6}
\left \lvert \right (\chi^\gamma_v\circ F^{-1}_v(t,\pm \tan\alpha_v t))^2 \nabla F^{-1}_v(t,\pm \tan\alpha_v t) -I_2\rvert &\leq C \gamma^{-\beta}.
\end{align}
Then by \eqref{TL5}, \eqref{TL6} and Theorem~\ref{Agmon} we get 
\[
\left\lvert \int_{\Omega} \lvert \chi^\gamma_v\rvert^2 \nabla \phi^{\gamma,v}_i \overline{\nabla \phi^{\gamma,v}_j} dx -\int_{(U_v)_{0,2\gamma^{-\beta}}} \nabla \psi^{\gamma,v}_i \overline{\nabla \psi^{\gamma,v}_j} dx  \right \rvert
\leq
C\gamma^{2-\beta} + C e^{-2(1-\epsilon)\gamma^{1-\beta} \sqrt{-1-E^{\max}}},
\]
and
\[
\left \lvert \int_{\partial \Omega} \lvert \chi^{\gamma}_v \rvert^2 \phi^{\gamma,v}_i \overline{\phi^{\gamma,v}_j} ds - \int_{\Sigma_{v,2\gamma^{-\beta}}} \psi^{\gamma,v}_i \overline{\psi^{\gamma,v}_j} ds \right \rvert
\leq
C\gamma^{2-\beta} + C e^{-2(1-\epsilon)\gamma^{1-\beta} \sqrt{-1-E^{\max}}}.
\]
As $i\neq j$, by the spectral theorem $t^{\gamma,\alpha_v}(\psi^{\gamma,v}_i,\psi^{\gamma,v}_j)=0$. Then,
\begin{multline*}
\left\lvert \int_{\Omega} \lvert \chi^\gamma_v \rvert^2  \nabla \phi^{\gamma,v}_i \overline{\nabla \phi^{\gamma,v}_j} dx -\gamma \int_{\partial \Omega} \lvert \chi^{\gamma}_v \rvert^2 \phi^{\gamma,v}_i \overline{\phi^{\gamma,v}_j} ds \right \rvert
\leq
C\gamma^{2-\beta} + C e^{-2(1-\epsilon)\gamma^{1-\beta} \sqrt{-1-E^{\max}}},
\end{multline*}
which concludes the proof using \eqref{crossbis}.
\end{proof}

\begin{lemma}\label{licurv}
For $\gamma$ large enough the family $(\tilde\phi^{\gamma,v}_n)_{(n,v)\in \cup_{l=1}^{K^\oplus} \mathcal S_l}$ is linearly independent.
\end{lemma}

\begin{proof}
Let us denote by $G$ the Gramian matrix associated with $(\tilde\phi^{\gamma,v}_n)_{(n,v)\in \cup_{l=1}^{K^\oplus} \mathcal S_l}$ which entries are $G_{i,j}=\langle \tilde\phi^{\gamma,v_i}_{n_i},\tilde \phi^{\gamma,v_j}_{n_j}\rangle$, where $(n_i,v_i)$, $(n_j,v_j) \in \bigcup_{l=1}^{K^\oplus}\mathcal S_l$. On one hand, the diagonal is simply composed by $G_{i,i}=1+O(\gamma^{-\beta})$, according to \eqref{prop1bis}. On the other hand, if $(n_i,v_i)\neq (n_j,v_j)$ then two cases are allowed : $v_i=v_j$ and $n_i\neq n_j$ or $v_i \neq v_j$. In the first case, we already know by \eqref{prop1bis} that $G_{i,j}=O(\gamma^{-\beta})$. In the second case, $\supp \chi^\gamma_{v_i}\cap \supp \chi^{\gamma}_{v_j}=\emptyset$ for $\gamma$ large enough which implies $G_{i,j}=0$. Then we can conclude that $\det(G)=1+O(\gamma^{-\beta})$ and in particular $\det(G)\neq 0$ for $\gamma$ large enough.
\end{proof}

\subsection{Cutting out the vertices}\label{cutvertices}

This section is a prelude to the study of the asymptotics of the eigenvalues of $Q^\gamma$ (Section~\ref{sectascurvpol}) and the Weyl asymptotics (Section~\ref{sectweylascurvpol}). We show how to separate the convex vertices from the rest of $\Omega$, which we will call \emph{regular part}, using a partition of unity and a Dirichlet bracketing.

We introduce the smooth function $\chi^\gamma_0$ defined by $\chi^\gamma_0 \coloneqq 1-\sum_{v\in\mathcal V} \chi^\gamma_v$, 
where the $\chi^\gamma_v$ are defined in \eqref{chigamma}. For all $v\in\mathcal V\cup\{0\}$, let 
 \[
 \tilde \chi^\gamma_v (x)\coloneqq \frac{\chi^\gamma_v(x)}{\sqrt{\sum_{v\in\mathcal V\cup\{0\}} (\chi_v^\gamma(x))^2}}, \quad x\in\Omega
 \]
 such that $\sum_{v\in\mathcal V\cup\{0\}} (\tilde \chi^\gamma_v)^2=1$ on $\Omega$.
 Then, for all $\phi\in H^1(\Omega)$, 

\begin{align}\label{IMScurv}
q^\gamma(\phi,\phi)= \sum_{v\in\mathcal V\cup \{0\}} q^\gamma(\phi \tilde\chi^\gamma_v,\phi\tilde \chi^\gamma_v) -\int_{\Omega} V(x) \lvert \phi \rvert^2 dx,
\end{align}
with $V(x)\coloneqq\displaystyle\sum_{v\in\mathcal V \cup\{0\}} \lvert \nabla \tilde \chi^\gamma_v(x) \rvert^2$. Notice that, by definition of $\tilde\chi^\gamma_v$, there exists $c>0$ such that
\begin{align}
\label{majV}
\lVert V \rVert_\infty \leq c \gamma^{2\beta}.
\end{align}
Let us define the \emph{regular part}, $\Omega_0\coloneqq \Omega \backslash \bigcup_{v\in\mathcal V}\overline{ B(v,\gamma^{-\beta})}$.  For all $v\in\mathcal V$ we introduce the sesquilinear forms
\[
 q^{\gamma,V}_{v,2\gamma^{-\beta}}(\phi,\phi) =\int_{\Omega_{v,2\gamma^{-\beta}}} (\lvert \nabla \phi\rvert^2 - V(x) \lvert \phi\rvert^2 )dx -\gamma\int_{\Gamma_{v,2\gamma^{-\beta}}}\lvert \phi\rvert^2 ds,
\]
with $D( q^{\gamma,V}_{v,2\gamma^{-\beta}})\coloneqq \{ \phi\in H^1(\Omega_{v,2\gamma^{-\beta}}), \phi(x)=0 \text{ for } x\in \partial \Omega_{v,2\gamma^{-\beta}}\backslash \Gamma_{v,2\gamma^{-\beta}}\}$, and 
\[
 q^{\gamma,V}_0(\phi,\phi)=\int_{\Omega_0} (\lvert\nabla \phi\rvert^2 -V(x) \lvert \phi\rvert^2 )dx -\gamma\int_{\Gamma_0} \lvert \phi\rvert^2 ds,
\]
where $\Gamma_0\coloneqq \partial \Omega_0\cap \partial \Omega$ and $D( q^{\gamma,V}_0)\coloneqq\{ \phi\in H^1(\Omega_0), \phi(x)=0 \text{ for } x\in \partial \Omega_0\backslash \Gamma_0\}$. 

\begin{lemma}\label{mincurvpol}
For all $n\in \mathbb N$ and for all $\gamma >0$, 
\begin{align*}
E_n(Q^\gamma) &\geq E_n\left( (\bigoplus_{v\in\mathcal V}  Q^{\gamma,V}_{v,2\gamma^{-\beta}} ) \bigoplus  Q^{\gamma,V}_0 \right ).
\end{align*}
\end{lemma}
\begin{proof}
Notice that $\phi \tilde\chi^\gamma_v \in D( q^{\gamma,V}_{v,2\gamma^{-\beta}})$ for all $v\in \mathcal V$ and $\phi \tilde\chi^\gamma_0 \in D( q^{\gamma,V}_0)$.
Thanks to \eqref{IMScurv}, for all $\phi \in H^1(\Omega)$,
\begin{align}
\label{mincurvpol2}
q^\gamma(\phi,\phi) \geq \sum_{v\in\mathcal V} q^{\gamma,V}_{v,2\gamma^{-\beta}}(\phi \tilde\chi^\gamma_v, \phi \tilde\chi^\gamma_v) +  q^{\gamma,V}_0(\phi \tilde\chi^\gamma_0,\tilde \phi \chi^\gamma_0).
\end{align}
Moreover, $\lVert \phi \rVert^2_{L^2(\Omega)}= \displaystyle \sum_{v\in\mathcal V} \lVert \phi \tilde \chi^\gamma_v\rVert^2_{L^2(\Omega_{v,2\gamma^{-\beta}})} +\lVert \phi \tilde\chi^\gamma_0\rVert^2_{L^2(\Omega_0)}$. Then, by \eqref{mincurvpol2} and the min-max principle we get, for all $n\in\mathbb N$,
\begin{align*}
E_n(Q^\gamma) &\geq 
\min_{\substack{G\subset D(q^\gamma) \\ \dim(G)=n}} 
\max_{\substack{\phi \in G\\ \phi \neq 0}}
\frac{\sum_{v\in\mathcal V}  q^{\gamma,V}_{v,2\gamma^{-\beta}}(\phi \tilde\chi^\gamma_v,\phi \tilde\chi^\gamma_v) +  q^{\gamma,V}_0(\phi \tilde\chi^\gamma_0,\phi \tilde\chi^\gamma_0)}{\sum_{v\in\mathcal V} \lVert \phi\tilde\chi^\gamma_v \rVert^2_{L^2(\Omega_{v,2\gamma^{-\beta}})}+\lVert \phi \tilde\chi^\gamma_0\rVert^2_{L^2(\Omega_0)}} \\
&\geq 
\min_{\substack{G\subset D\left((\bigoplus_{v\in\mathcal V}  q^{\gamma,V}_{v,2\gamma^{-\beta}})\bigoplus  q^{\gamma,V}_0 \right)\\ \dim(G)=n }}
\max_{\substack{\left((\phi_v)_{v\in\mathcal V}, \phi_0 \right) \in G \\ \left((\phi_v)_{v\in\mathcal V}, \phi_0 \right) \neq (0,...,0)  }}
\frac{\sum_{v\in\mathcal V}  q^{\gamma,V}_{v,2\gamma^{-\beta}} (\phi_v,\phi_v) + q^{\gamma,V}_0(\phi_0,\phi_0)}
{\sum_{v\in\mathcal V} \lVert \phi_v \rVert^2_{L^2(\Omega_{v,2\gamma^{-\beta}})}+ \lVert \phi_0\rVert^2_{L^2(\Omega_0)}} \\
&= E_n\left( (\bigoplus_{v\in\mathcal V} Q^{\gamma,V}_{v,2\gamma^{-\beta}} ) \bigoplus Q^{\gamma,V}_0 \right ),
\end{align*}
which concludes the proof.
\end{proof}
Let us now introduce the sesquilinear forms
\[
 q^\gamma_{v,\gamma^{-\beta}}(\phi,\phi) =\int_{\Omega_{v,\gamma^{-\beta}}} \lvert \nabla \phi\rvert^2 dx -\gamma\int_{\Gamma_{v,\gamma^{-\beta}}}\lvert \phi\rvert^2 ds,
\]
with $D(q^\gamma_{v,\gamma^{-\beta}})\coloneqq \{\phi\in H^1(\Omega_{v,\gamma^{-\beta}}), \phi(x)= 0 \text{ for } x\in \partial \Omega_{v,\gamma^{-\beta}}\backslash \Gamma_{v,\gamma^{-\beta}}\}$ and 
\[
q^\gamma_0(\phi,\phi)=\int_{\Omega_0} \lvert\nabla \phi\rvert^2 dx -\gamma\int_{\Gamma_0} \lvert \phi\rvert^2 ds,
\]
where $D(q^\gamma_0)\coloneqq D( q^{\gamma,V}_0)$.

\begin{lemma}\label{Dbracket}
For all $n\in\mathbb N$ and for all $\gamma >0$, 
\[
E_n(Q^\gamma) \leq E_n\left(\bigoplus_{v\in\mathcal V} Q^\gamma_{v,\gamma^{-\beta}}\bigoplus Q^\gamma_0 \right).
\]
\end{lemma}
\begin{proof}
It is a consequence of the min-max principle, noticing that if $\phi \in D\left ( (\bigoplus_{v\in\mathcal V} q^\gamma_{v,\gamma^{-\beta}}) \bigoplus q^\gamma_0\right)$, then $\phi \in D(q^\gamma)$. 
\end{proof}

Recall that as $\Omega$ is a curvilinear polygon, its boundary is composed by $M\geq 1$  connected arcs $(\Gamma_k)_{k=1}^M$ such that $\partial \Omega = \bigcup_{k=1}^M \overline{\Gamma_k}$. We denote by $l_k$ the lenght of $\Gamma_k$ and by $\kappa_k$ its curvature. The following lemma gives us some estimates concerning the \emph{regular part} of $\Omega$. The proof is given in the next section.

\begin{lemma}\label{thregpart}
For $\gamma$ large enough and for all $\beta\in(\frac{1}{2},1)$ one has 
\begin{align}
\label{lbvpregpart}
E_1(Q^{\gamma,V}_0)\geq -\gamma^2 -C\gamma^{2\beta}.
\end{align}
Moreover, for all $E\in(-1,0)$, $\lambda \in\mathbb R$ and for all $\beta\in(1/2,1)$, one has as $\gamma >0$ is large enough,
\begin{align}
\label{upNregpart}
\mathcal N(Q^{\gamma,V}_0,E\gamma^2) &\leq  \gamma \frac{\lvert \partial \Omega\rvert\sqrt{E+1}}{\pi}+C\gamma^{1-\beta}, \\
\label{lbNregpart}
\mathcal N(Q^{\gamma}_0,E\gamma^2) &\geq \gamma \frac{\lvert \partial \Omega\rvert\sqrt{E+1}}{\pi}-C\gamma ^{1-\beta},
\end{align}
and,
\begin{align}
\label{upNregpartbis}
\mathcal N ( Q^{\gamma,V}_0, -\gamma^2+\lambda\gamma) &\leq \frac{\sqrt{\gamma}}{\pi} \sum_{k=1}^M \int_0^{l_k}\sqrt{(\kappa_k(s)+\lambda)_+} ds+C \gamma^{\beta-1/2}, \\
\label{lbNregpartbis}
\mathcal N(Q^{\gamma}_{0}, -\gamma^2+\lambda\gamma) &\geq \frac{\sqrt{\gamma}}{\pi} \sum_{k=1}^M \int_0^{l_k} \sqrt{(\kappa_k(s)+\lambda)_+} ds-C.
\end{align}
\end{lemma}

\subsection{Proof of Lemma~\ref{thregpart}}

This section is devoted to the proof of Lemma~\ref{thregpart}.  
Recall that $\partial \Omega=\bigcup_{k=1}^M \overline{\Gamma_k}$, and $l_k$ denotes the lenght of $\Gamma_k$.
We consider $\gamma_k$ the parametrization by the arc length of $\overline{\Gamma_k}$, namely :
 $\gamma_k \in C^4([0,l_k], \mathbb R^2)$ is injective, 
$\gamma_k(s)=(\gamma_{k,1}(s),\gamma_{k,2}(s)) \in \overline{\Gamma_k}$ and
 $\lvert \gamma_k'(s)\rvert =1$ for all $s\in[0,l_k]$. 
We denote by $\nu_k(s)$ the unit outward normal of $\Gamma_k$ at the point $s$ and suppose that the orientation of $\gamma_k$ is chosen such that 
$\nu_k(s)= (\gamma'_{k,2}(s),- \gamma'_{k,1}(s))$ for each $k$. 
Let $\kappa_k$ be the signed curvature of $\Gamma_k$:
\[
\kappa_k(s)=\gamma'_{k,1}(s)\gamma''_{k,2}(s)-\gamma'_{k,2}(s)\gamma''_{k,1}(s),
\]
and $\kappa_{\max}\coloneqq\displaystyle\max_{k=1,...,M} \left(\displaystyle \max _{s\in[0,l_k]}\kappa_k(s)\right)$.
We introduce the map
\[
\varphi_k : (0,l_k)\times \mathbb R \to \mathbb R^2,\quad \varphi_k(s,t)=\gamma_k(s) - t \nu_k(s).
\]
There exists $a_k>0$ such that, for all $a<a_k$, $\varphi_k$
is a diffeomorphism between $\square^k_a\coloneqq (0,l_k)\times (0,a)$ and $\Omega^k_a\coloneqq \varphi_k(\square^k_a)$. We define
\begin{align}
\label{truncstrip}
\tilde \Omega^k_a\coloneqq \Omega^k_a\cap \Omega_0.
\end{align}
Recall that the parameter $\beta \in (1/2,1)$ was introduced in \eqref{chigamma}.
For $\epsilon \in (0, 1-\beta)$, let $a_\gamma\coloneqq \gamma^{-1+\epsilon}$. Then, for $\gamma$ large enough and for $k\neq k'$ we have
\begin{align}
\label{empty}
\tilde \Omega^k_{a_\gamma}\cap \tilde\Omega^{k'}_{a_\gamma}=\emptyset.
\end{align}
In the following, $\gamma>0$ is large enough such that \eqref{empty} is satisfied and $a_\gamma< \displaystyle\min_{k=1,...,M}a_k$.

\subsubsection{Proof of \eqref{lbvpregpart}, \eqref{upNregpart} and \eqref{upNregpartbis} }\label{Neumann}

Notice that in order to prove \eqref{lbvpregpart} we cannot use the same trick used in the proof of Proposition~\ref{asvp}, namely extending the domain in a smooth way to apply  \cite[Theorem 1]{Pan}. Indeed, now $\Omega_0$ depends on $\gamma$. To overcome this problem, we adapt the sketch of the proof of \cite{Pan}. The difficulty remains in the fact that we have to control the potential $V$.


Let $k\in\{1,...,M\}$ be fixed. 
We define 
\[
 q^{\gamma,N,V}_k(\phi,\phi)=\int_{ \Omega^k_{a_\gamma}} \left(\lvert \nabla \phi \rvert^2 -V(x) \lvert \phi\rvert^2 \right) dx -\gamma \int_{\Gamma_k} \lvert \phi\rvert^2 ds,
\]
 where $D(q^{\gamma,N,V}_k)\coloneqq \{\phi \in H^1(\Omega^k_{a_\gamma}), \phi(x)=0 \text{ for } x \in \Gamma^k_l\cup \Gamma^k_r\}$, $\Gamma^k _l\coloneqq\{ \varphi(0,t), t\in(0,a_\gamma)\}$ and $\Gamma^k_r\coloneqq\{\varphi_k(l_k,t), t\in (0,a_\gamma)\}$.

\begin{proposition}\label{interpb}
For $\gamma$ large enough we have
\begin{align}
\label{interpb1}
E_1(Q^{\gamma,N,V}_k)\geq -\gamma^2-C\gamma^{2\beta}.
\end{align}
Moreover, for all $E\in (-1,0)$,
\begin{align}
\label{interpb2}
\mathcal N(Q_k^{\gamma,N,V}, E\gamma^2) \leq \gamma \frac{l_k \sqrt{E+1}}{\pi}+ C\gamma^{1-\beta},
\end{align}
and for all $\lambda\in \mathbb R$,
\begin{align}
\label{interpb3}
\mathcal N(Q_k^{\gamma,N,V}, -\gamma^2+\lambda\gamma)\leq \frac{\sqrt{\gamma}}{\pi} \int_0^{l_k} \sqrt{(\kappa_k(s)+\lambda)_+} ds + C \gamma^{\beta-1/2}.
\end{align}
\end{proposition}

\begin{proof}
As the study is the same for all $k\in\{1,...,M\}$ we omit the indices $k$ in the proof.
The idea is to perform a change of variables thanks to the diffeomorphism $\varphi$ in order to work with 
$\square_{a_\gamma}$ 
which will allows us to use separation of variables. 
But, in order to avoid the weight in the integrals due to the Jacobian of the change of variables, we first introduce a unitary transform. Define $U_{a_\gamma} : L^2(\Omega_{a_\gamma})\to L^2(\square_{a_\gamma})$, 
$U_{a_\gamma}(\phi)(s,t)\coloneqq \sqrt{1-t\kappa(s)} \phi\circ\varphi(s,t)$. 
Then, $U_{a_\gamma}\left(D(q^{\gamma,N,V})\right)=\{\phi \in H^1(\square_{a_\gamma}), \phi(0,t)=\phi(l,t)=0\}$.
 It is easy to prove (see \cite{Pan} or \cite{EMP} for a detailed computation) that, after using integrations by parts,
$q^{\gamma,N,V}(\phi,\phi)=p^{\gamma,N,\tilde V}(U_{a_\gamma}\phi,U_{a_\gamma}\phi)$ with 
$D(p^{\gamma,N,\tilde V})\coloneqq U_{a_\gamma}\left(D(q^{\gamma,N,V})\right)$, 
and $p^{\gamma,N,\tilde V}$ is given by the following expression,
\begin{align*}
p^{\gamma,N,\tilde V}(\phi,\phi) &=
 \int_{\square_{a_\gamma}} 
\left( \frac{1}{(1-t\kappa(s))^2} \lvert \partial_s \phi\rvert^2 + \lvert\partial_t \phi \rvert^2 -\tilde V(s,t) \lvert \phi \rvert^2 -P(s,t) \lvert \phi\rvert^2 \right) ds dt \\
&+\frac{1}{2} \int_0^{l} \frac{\kappa(s)}{1-a_\gamma\kappa(s)} \lvert \phi (s,a_\gamma)\rvert^2 ds
- \int_0^{l} \left(\frac{\kappa(s)}{2}+\gamma\right) \lvert \phi(s,0)\rvert^2 ds ,
\end{align*}
where $\tilde V(s,t)=V\circ\varphi(s,t)$ and 
\begin{align}
\label{potP}
P(s,t)= \frac{\kappa^2(s)}{4(1-t\kappa(s))^2} + \frac{t \kappa''(s)}{2(1-t\kappa(s))^3} + \frac{5 t^2 (\kappa'(s))^2}{4(1-t\kappa(s))^4}.
\end{align}
As $U_{a_\gamma}$ is a unitary map, we immediately get by the min-max principle, for all $n\in\mathbb N$,
\begin{align}
\label{ineg3.1}
E_n(Q^{\gamma,N,V})=E_n(P^{\gamma,N,\tilde V}).
\end{align}
Before going further, let us make a remark on the potential $\tilde V$. 
Recall that $V(x)=\sum_{v\in\mathcal V\cup\{0\}} \lvert \nabla \tilde\chi^\gamma_v(x)\rvert^2$. 
Then, $\supp V\subset \bigcup_{v\in\mathcal V} B(v,2\gamma^{-\beta})\backslash \overline{B(v,\gamma^{-\beta})}$.
 This implies that there exists a constant $b>0$ such that 
\[
\supp \tilde V\subset \left((0,b\gamma^{-\beta})\times (0,a_\gamma)\right) \bigcup \left((l-b\gamma^{-\beta},l)\times (0,a_\gamma)\right).
\]
We denote $s_\gamma\coloneqq b\gamma^{-\beta}$ and introduce 
\[
\xi(s)\coloneqq 
\begin{cases}
1,&\text{ if } s\in (0,s_\gamma) \text{ or } s\in (l-s_\gamma,l),\\
0,&\text{ otherwise.}
\end{cases}
\]
By \eqref{majV} we can write, for all $(s,t)\in\square_{a_\gamma}$,
\[
\tilde V(s,t)\leq c\gamma^{2\beta} \xi(s).
\]
We now give some estimates which will simplify the study.
As $\kappa\in C^2([0,l],\mathbb R)$, there exist $\mathcal K>0$ and $C>0$ such that, $\lvert \displaystyle \frac{\kappa(s)}{1-t\kappa(s)}\rvert \leq 2\mathcal K$  for all $(s,t)\in \square_{a_\gamma}$, and 
\begin{align}
\label{estcov}
1-a_\gamma C < \frac{1}{(1-t\kappa(s))^2}< 1+a_\gamma C, \quad \text{and}\quad
\lvert P(s,t)\rvert \leq C.
\end{align}
Thus we can write for all $\phi \in D(p^{\gamma,N,\tilde V})$, 
$p^{\gamma,N,\tilde V}(\phi,\phi) \geq h^{\gamma,N,\xi_\gamma}(\phi,\phi)$, where
\begin{align*}
h^{\gamma,N,\xi_\gamma}(\phi,\phi)&=\int_{\square_{a_\gamma}} \left( (1-a_\gamma C) \lvert \partial_s \phi\rvert^2 + \lvert \partial_t\phi\rvert^2 -\xi_\gamma(s) \lvert \phi \rvert^2 - C \lvert \phi \rvert ^2 \right)dsdt \\
&- \mathcal K \int_0^{l} \lvert \phi(s,a_\gamma)\rvert^2 ds 
-\int_0^{l} (\frac{\kappa(s)}{2}+\gamma) \lvert \phi (s,0) \rvert^2 ds,
\end{align*}
where $D(h^{\gamma,N,\xi_\gamma})\coloneqq H^1(\square_{a_\gamma})$ and $\xi_\gamma(s)\coloneqq c\gamma^{2\beta} \xi(s)$.
We now can conclude by the min-max principle that, for all $n\in\mathbb N$,
\begin{align}
\label{ineg5}
E_n(P^{\gamma,N,\tilde V})\geq E_n(H^{\gamma,N,\xi_\gamma}).
\end{align}
In order to control the potential $\xi_\gamma$, we have to introduce some new sesquilinear forms.
We define
\begin{align*}
h^{\gamma,N,1}(\phi,\phi)=&\int_{0}^{s_{\gamma}}\int_0^{a_\gamma} \left( (1-a_\gamma C) \lvert \partial_s \phi\rvert^2 + \lvert \partial_t\phi\rvert^2 -c\gamma^{2\beta} \lvert \phi \rvert^2 -C \lvert \phi \rvert ^2 \right)dsdt \\
&- \mathcal K \int_{0}^{s_\gamma} \lvert \phi(s,a_\gamma)\rvert^2 ds
-\int_{0}^{s_\gamma} (\frac{\kappa(s)}{2}+\gamma) \lvert \phi (s,0) \rvert^2 ds, \quad \phi\in H^1\left((0,s_\gamma)\times (0,a_\gamma)\right),
\end{align*}

\begin{align*}
h^{\gamma,N,2}(\phi,\phi)&=\int_{s_{\gamma}}^{l-s_\gamma}\int_{0}^{a_\gamma} \left( (1-a_\gamma C) \lvert \partial_s \phi\rvert^2 + \lvert \partial_t\phi\rvert^2  -C\lvert \phi \rvert ^2 \right)dsdt \\
&- \mathcal K \int_{s_\gamma}^{l-s_\gamma} \lvert \phi(s,a_\gamma)\rvert^2 ds 
-\int_{s_\gamma}^{l-s_\gamma} (\frac{\kappa(s)}{2}+\gamma) \lvert \phi (s,0) \rvert^2 ds,\quad \phi \in H^1\left((s_\gamma,l-s_\gamma)\times (0,a_\gamma)\right),
\end{align*}
 and 
\begin{align*}
h^{\gamma,N,3}(\phi,\phi)=&\int_{l-s_\gamma}^{l}\int_0^{a_\gamma} \left( (1-a_\gamma C) \lvert \partial_s \phi\rvert^2 + \lvert \partial_t\phi\rvert^2 -c\gamma^{2\beta} \lvert \phi \rvert^2 -C \lvert \phi \rvert ^2 \right)dsdt \\
&- \mathcal  K \int_{l-s_\gamma}^{l} \lvert \phi(s,a_\gamma)\rvert^2 ds
-\int_{l-s_\gamma}^{l} (\frac{\kappa(s)}{2}+\gamma) \lvert \phi (s,0) \rvert^2 ds, \quad \phi \in H^1\left((l-s_\gamma,l)\times (0,a_\gamma)\right).
\end{align*}

Using the min-max principle we obtain the following inequality for all $n\in\mathbb N$,
\begin{align}
\label{ineg7}
E_n(H^{\gamma,N,\xi_\gamma}) \geq E_n\left(\bigoplus_{i=1}^3 H^{\gamma,N,i}\right).
\end{align}

Let us introduce, for simplicity, the more general sesquilinear form 
\begin{align*}
h^{\gamma,N}(\phi,\phi)=&\int_{0}^{L}\int_0^{a_\gamma} \left( (1-a_\gamma C) \lvert \partial_s \phi\rvert^2 + \lvert \partial_t\phi\rvert^2 -c_\gamma\lvert \phi \rvert^2 \right)dsdt \\
&- \mathcal K \int_{0}^{L} \lvert \phi(s,a_\gamma)\rvert^2 ds 
-\int_{0}^{L} (\frac{g(s)}{2}+\gamma) \lvert \phi (s,0) \rvert^2 ds, \quad \phi \in H^1\left((0,L)\times (0,a_\gamma)\right),
\end{align*}
where $L>0$, $c_\gamma>0$ depends on $\gamma$ and will play the role of the potentials $c\gamma^{2\beta}+C$ or $C$ and $g\in C^2([0,L],\mathbb R)$.
We first prove some results on $H^{\gamma,N}$, namely estimates on the first eigenvalue and the counting function, and then apply them to the $H^{\gamma,N,i}$.
For any $K\in\mathbb N$, we denote
\[
\delta\coloneqq \frac{L}{K},\quad I_j\coloneqq(\delta (j-1), \delta j), \quad j\in\{1,...,K\},
\]
and
\[
g^+_j\coloneqq \sup_{s \in I_j} g(s).
\]
We begin defining the sesquilinear forms associated with the partition of $(0,L)$. For $j\in\{1,...,K\}$, let us consider
\begin{align*}
t^{\gamma,N}_j(\phi,\phi) &=\int_{I_j} \int_0^{a_\gamma} \left ((1-a_\gamma C) \lvert \partial_s \phi\rvert^2 +\lvert \partial_t \phi\rvert^2 -c_\gamma\lvert \phi \rvert ^2 \right)dsdt \\
&-\mathcal K\int_{I_j} \lvert \phi(s,a_\gamma)\rvert^2 ds 
-\int_{I_j} (\frac{g^+_j}{2}+\gamma) \lvert \phi(s,0)\rvert^2 ds, \quad \phi\in H^1(I_j\times(0,a_\gamma)).
\end{align*}
Clearly we have 
$D(h^{\gamma,N})\subset \bigoplus_{j=1}^K D(t^{\gamma,N}_j)$. 
Then, by the min-max principle we get for all $n\in\mathbb N$,
\begin{align}
\label{appTS1}
E_n(H^{\gamma,N}) \geq E_n\left(\bigoplus_{j=1}^K T^{\gamma,N}_j\right).
\end{align}
Let us fix $j\in\{1,...,K\}$. 
By separation of variables, it is easy to see that 
$E_n(T^{\gamma,N}_j)=E_n(\mathcal L^N \otimes 1 + 1\otimes \mathscr R_j)$. 
Here, the operator $\mathcal L^N$ acts on $L^2(0,\delta)$ as 
\[
\mathcal L^N f= -(1-a_\gamma C)f'' -c_\gamma f, \quad D(\mathcal L^N)\coloneqq\{f\in H^2(0,\delta), -f'(0)=f'(\delta)=0\}.
\]
The operator $\mathscr R_j\coloneqq \mathscr R_{\frac{g^+_j}{2}+\gamma, \mathcal K, a_\gamma}$ is defined in Section~\ref{auxop}. It acts on $L^2(0,a_\gamma)$ as $f\mapsto -f''$ with 
\begin{align*}
D(\mathscr R_j)&\coloneqq \{f\in H^2(0,a_\gamma), -f'(0)-(\frac{g^+_j}{2}+\gamma)f(0)=f'(a_\gamma)-\mathcal K f(a_\gamma)=0\}.
\end{align*}
There exists $\gamma_{1}>0$ 
such that for all $\gamma > \gamma_{1}$ 
we have $(\displaystyle \frac{g^+_j}{2}+\gamma) a_\gamma >1$ and $(\displaystyle \frac{g^+_j}{2}+\gamma) > 2\mathcal K$. 
Then, we know by Proposition~\ref{LemmaPan} that $E_1(\mathscr R_j)$ is the unique negative eigenvalue of $\mathscr R_j$ and we also have the following estimate, for all $\gamma >\gamma_{1}$,
\begin{align}
\label{estpan1}
E_1(\mathscr R_j)>-(\frac{g^+_j}{2}+\gamma)^2 -123(\frac{g^+_j}{2}+\gamma)^2 e^{-2(\frac{g^+_j}{2}+\gamma) a_\gamma} .
\end{align}
As $\inf\Spec(\mathcal L^N)=-c_\gamma$, we get
\begin{align}
\label{appTS3}
E_1(T^{\gamma, N}_j)=E_1(\mathscr R_j)-c_\gamma.
\end{align}
Using \eqref{appTS3} and \eqref{estpan1}, there exists $\gamma_{2} >\gamma_{1}$ such that for all $\gamma >\gamma_{2}$,
\[
E_1(T^{\gamma,N}_j) \geq -(\frac{g^+_j}{2} +\gamma)^2 -c_\gamma -C.
\]
For all $j\in \{1,...,K\}$, $g^+_j\leq g_{\max}\coloneqq \max_{s\in [0,L]}g(s) $ and , by \eqref{appTS1} we can conclude that  for all $\gamma > \gamma_{2}$,
\[
E_1(H^{\gamma,N})\geq -\gamma^2 -\gamma g_{\max} -c_\gamma - C.
\]
Notice that it is easy to apply the previous result to the operators 
 $H^{\gamma,N,i}$ by making a translation and considering, for $i=2$, $g(s)\coloneqq \kappa(s+s_\gamma)$ and for $i=3$, $g(s)\coloneqq \kappa(s+(l-s_\gamma))$.
Thus, for $\gamma>\gamma_2$ we have
\[
E_1(H^{\gamma,N,i})\geq -\gamma^2-\gamma\kappa_{\max} -c\gamma^{2\beta} -C ,\quad i=1,3,
\] 
and
\[
E_1(H^{\gamma,N,2})\geq -\gamma^2-\gamma\kappa_{\max} -C .
\]
There exists $\gamma_{3}>\gamma_{2}$ such that for all $\gamma>\gamma_{3}$ we have, thanks to \eqref{ineg7},
\[
E_1(H^{\gamma,N,\xi_\gamma}) \geq -\gamma^2-C\gamma^{2\beta},
\]
as $\beta\in(1/2,1)$. Finally, this concludes the proof of \eqref{interpb1} thanks to \eqref{ineg3.1} and \eqref{ineg5}.

We now focus on the eigenvalue counting function. 
Let $E\in(-1,0)$ be fixed. Thanks to the fact that $E_1(\mathscr R_j)$ is the unique negative eigenvalue of the operator $\mathscr R_j$ as $\gamma >\gamma_1$ and using estimate \eqref{estpan1} one can write
\[
\mathcal N(T^{\gamma,N}_j, E\gamma^2) \leq \frac{\delta}{\pi\sqrt{1-a_\gamma C}} \sqrt{(E+1)\gamma^2+g^+_j\gamma+c_\gamma+C}+1.
\]
Thus, summing on $k\in\{1,...,K\}$ and using \eqref{appTS1} we obtain for $\gamma$ large enough,
\[
\mathcal N(H^{\gamma,N},E\gamma^2)\leq \gamma\frac{L\sqrt{E+1}}{\pi \sqrt{1-a_\gamma C}} + L C c_\gamma \gamma^{-1} +K.
\]
Recall that $a_\gamma\coloneqq \gamma^{-1+\epsilon}$ with $\epsilon\in(0,1-\beta)$. We can write $(1-a_\gamma C)^{-1/2}= 1+ \displaystyle\frac{1}{2} C \gamma^{-1+\epsilon} +O(\gamma^{-2+2\epsilon})$ as $\gamma\to +\infty$. 
Then,
\[
\mathcal N(H^{\gamma,N},E\gamma^2) \leq \gamma \frac{L \sqrt{E+1}}{\pi} + LC\gamma^{\epsilon} +L C c_\gamma \gamma^{-1} +K.
\]
We can now apply this previous result to the operators $H^{\gamma,N,i}$ with $c_\gamma=c\gamma^{2\beta}+C$ and $L=b\gamma^{-\beta}$ for $i=1,3$ and $c_\gamma=C$ and $L=l-2b\gamma^{-\beta}$ for $i=2$. We finally obtain, choosing $K\in[\gamma^{\epsilon},2\gamma^{\epsilon}]\cap\mathbb N$,
\[
\mathcal N(H^\gamma,N,\xi_\gamma) \leq \gamma \frac{l\sqrt{E+1}}{\pi}+ O(\gamma^\epsilon),\quad \gamma\to+\infty,
\]
with $\epsilon <1-\beta$.
This finishes the proof of \eqref{interpb2} thanks to 
\eqref{ineg3.1} and \eqref{ineg5}.

Let us prove \eqref{interpb3}. Let $\lambda\in\mathbb R$ be fixed. There exists $\gamma_4>\gamma_1$ such that for all $\gamma >\gamma_4$ we have $-\gamma^2+\lambda\gamma<0$. We can write, using again \eqref{estpan1},
\[
\mathcal N(T^{\gamma,N}_j,-\gamma^2+\lambda \gamma)\leq \frac{\delta}{\pi\sqrt{1-a_\gamma C}} \sqrt{\gamma(g^+_j+\lambda)_+} + C\delta c_\gamma \gamma^{-1/2} +1. 
\]
We can sum the inequalities on $j\in\{1,...,K\}$ and apply it to the operators $H^{\gamma,N,i}$. We obtain, for $i=1,3$,
\[
\mathcal N(H^{\gamma,N,i},-\gamma^2+\lambda\gamma) \leq  C\gamma^{1/2-\beta} K+ C \gamma^{\beta-1/2} +K,
\]
and 
\[
\mathcal N(H^{\gamma,N,2},-\gamma^2+\lambda\gamma) \leq \frac{\sqrt{\gamma}}{\pi\sqrt{1-a_\gamma C}} \frac{l-2b\gamma^{-\beta}}{K} \sum_{j=1}^K \sqrt{(g^+_j+\lambda)_+} + C\gamma^{-1/2} +K,
\]
where $g(s)= \kappa(s+s_\gamma)$.
Notice that $(0,L)\ni s \mapsto \sqrt{(g+\lambda)_+}$ is Lipschitz, thus we can use the convergence of Riemann sum to have
\[
\int_{0}^L \sqrt{(g+\lambda)_+} ds = \frac{L}{K} \sum_{j=1}^K \sqrt{(g^+_j+\lambda)_+}+ O(\frac{1}{K}),\quad K\to +\infty.
\]
Let us choose $K\in [\gamma^{\beta-1/2},2\gamma^{\beta-1/2}]\cap \mathbb N$. Then for $\gamma$ large enough,
\begin{align*}
\mathcal N(H^{\gamma,N,2},-\gamma^2+\lambda \gamma) \leq \frac{\sqrt{\gamma}}{\pi} \int_{s_\gamma}^{l-s_\gamma} \sqrt{(\kappa(s)+\lambda)_+} ds + C \gamma^{\beta-1/2}.
\end{align*}
In addition we have 
\[
\int_0^l \sqrt{(\kappa(s)+\lambda)_+}ds =  \int_{s_\gamma}^{l-s_\gamma} \sqrt{(\kappa(s)+\lambda)_+}ds +O(\gamma^{-\beta}), \quad \gamma\to +\infty,
\]
as $\sqrt{(\kappa(s)+\lambda)_+} \leq \sqrt{(\kappa_{\max}+\lambda)_+}$ for all $s\in(0,l)$. 
Finally, 
\begin{align}
\label{ineg10}
\mathcal N(H^{\gamma,N,\xi_\gamma},-\gamma^2+\lambda \gamma)\leq  
\frac{\sqrt{\gamma}}{\pi} \int_0^l \sqrt{(\kappa(s)+\lambda)_+}ds+O(\gamma^{\beta-1/2}),\quad \gamma\to +\infty.
\end{align}
We conclude the proof of \eqref{interpb3} thanks to \eqref{ineg3.1} and \eqref{ineg5}.
\end{proof}

\begin{proof}[Proof of \eqref{lbvpregpart}, \eqref{upNregpart} and \eqref{upNregpartbis}]

We introduce $\tilde \Omega_{0}\coloneqq \Omega_0 \backslash \bigcup_{k=1}^M \overline{\tilde\Omega^k_{a_\gamma}}$ and the closed sesquilinear forms 
\[
\tilde q^{\gamma,N,V}_k(\phi,\phi)=\int_{\tilde \Omega^k_{a_\gamma}} \left(\lvert \nabla \phi \rvert^2 -V(x) \lvert \phi\rvert^2 \right) dx -\gamma \int_{\Gamma_k\cap \tilde \Omega^k_{a_\gamma}} \lvert \phi\rvert^2 ds,
\]
with $D(\tilde q^{\gamma,N,V}_k)\coloneqq \{\phi \in H^1(\tilde\Omega^k_{a_\gamma}), \phi(x)=0 \text{ for } x \in (\partial  \tilde\Omega^k_{a_\gamma} \cap \partial \Omega_0)\backslash \Gamma_0\}$
and
\[
\tilde q_{0}^{N,V}(\phi,\phi)= \int_{\tilde \Omega_{0}} \left(\lvert \nabla \phi\rvert^2-V(x) \lvert \phi\rvert^2  \right)dx,
\]
with $D(\tilde q_{0}^{N,V})\coloneqq \{ \phi\in H^1(\tilde \Omega_{0}),  \phi(x)=0 \text{ for } x\in \partial \tilde \Omega_{0}\cap\partial\Omega_0\}$.
Noticing that 
$D(q^{\gamma,V}_0)\subset \bigoplus_{k=1}^M D(\tilde q^{\gamma,N,V}_k) \bigoplus D(\tilde q^{N,V}_{0})$ and thanks to \eqref{empty}, 
we can use the min-max principle and immediately obtain, 
for all $n\in\mathbb N$,
\begin{align}
\label{ineg1.1}
E_n(Q^{\gamma,V}_0) &\geq E_n\left((\bigoplus_{k=1}^M \tilde Q^{\gamma,N,V}_k )\oplus \tilde Q^{N,V}_{0}\right).
\end{align}
Notice that, by \eqref{majV}, 
$\mathcal N(\tilde Q^{N,V}_{0}, E\gamma^2+\lambda\gamma)\leq \mathcal N(\tilde Q^{N}_{0}, E\gamma^2 +\lambda \gamma + c \gamma^{2\beta})$, 
where $\tilde Q^{N}_{0}$
 is the unique self-adjoint operator associated with the sesquilinear form 
\[
\tilde q^N_{0}(\phi,\phi)=\int_{\tilde \Omega_{0}} \lvert \nabla \phi\rvert^2 dx, \quad 
\phi \in D(\tilde q^N_{0})\coloneqq D(\tilde q^{N,V}_{0}).
\]
The operator $\tilde Q^N_{0}$ is positive.
As $\beta<1$, there exists $\gamma_0>0$ such that, for all $\gamma >\gamma_0$ we have $\tilde E\gamma^2+\lambda\gamma+c\gamma^{2\beta}<0$, with $\tilde E\in[-1,0)$ and $\lambda\in\mathbb R$. Then, for all $\gamma >\gamma_0$, 
$\mathcal N(\tilde Q^{N}_{0}, \tilde E\gamma^2 +\lambda \gamma + c \gamma^{2\beta})=0$ and by \eqref{ineg1.1},
\begin{align}
\label{ineg2.1}
\mathcal N(Q^{\gamma,V}_0, \tilde E\gamma ^2 +\lambda \gamma) &\leq \sum_{k=1}^M \mathcal N (\tilde Q^{\gamma,N,V}_k, \tilde E\gamma^2+\lambda\gamma ).
\end{align}
 As $\tilde \Omega^k _{a_\gamma}\subset \Omega^k_{a_\gamma}$, extending $\phi \in D(\tilde q^{\gamma,N,V}_k)$ by zero we obtain, by the min-max principle and for all $n\in\mathbb N$,
\begin{align}
\label{ineg2.3}
E_n(\tilde Q^{\gamma,N,V}_k) \geq E_n(Q^{\gamma,N,V}_k).
\end{align}
We are now able to conclude.
On one hand, noticing that $E_1(\tilde Q^{N,V}_0)\geq -c\gamma^{2\beta}$ and by \eqref{interpb1} we have for $\gamma$ large enough,
\[
E_1\left(\bigoplus_{k=1}^M \tilde Q^{\gamma,N,V}_k\bigoplus\tilde Q^{N,V}_0\right)\geq -\gamma^2-C\gamma^{2\beta},
\]
which finishes the proof of \eqref{lbvpregpart} thanks to \eqref{ineg1.1}.
On the other hand, gathering \eqref{interpb2}, \eqref{interpb3}, \eqref{ineg2.1} and \eqref{ineg2.3} finishes the proof of \eqref{upNregpart} and \eqref{upNregpartbis}.
\end{proof}

\subsubsection{Proof of \eqref{lbNregpart} and \eqref{lbNregpartbis}}\label{Dirichlet}

We still follow the ideas of the proof of \cite{Pan}, but this proof is easier than the previous one as there is no potential $V$ in the sesquilinear form to control.

Let us introduce the new sesquilinear forms
\[
q^{\gamma,D}_k(\phi,\phi)=\int_{\tilde \Omega^k_{a_\gamma}} \lvert \nabla \phi \rvert^2 dx -\gamma \int_{\Gamma_k\cap \tilde \Omega^k_{a_\gamma}} \lvert \phi \rvert^2 ds,
\]
where $D(q^{\gamma,D}_k)\coloneqq\{ \phi \in H^1(\tilde \Omega^k_{a_\gamma}), \phi(x)=0 \text{ for } x \in \partial \tilde \Omega^k_{a_\gamma} \backslash \Gamma_k\}$, $\tilde \Omega^k_{a_\gamma}$ being defined in \eqref{truncstrip} and
\[
\tilde q^D_{0}(\phi,\phi)=\int_{\tilde \Omega_{0}} \lvert \nabla \phi\rvert^2 dx,\quad \phi \in H^1_0(\tilde \Omega_0),
\] 
with $\tilde \Omega_0\coloneqq \Omega_0\backslash \bigcup_{k=1}^M \overline {\tilde \Omega^k_{a_\gamma}}$.

Even if the 
 strategy of the proofs will be same as in Section~\ref{Neumann}, we have to work with $\tilde \Omega^k_{a_\gamma}$ instead of $\Omega^k_{a_\gamma}$, as the trick we used previously does not apply here. 

\begin{Remarques}
Let $k$ be such that $\Gamma_k$ links two convex corners. Then, by definition of $\tilde \Omega^k_{a_\gamma}$, there exists $b>0$  such that 
$
\tilde \Omega^k_{a_\gamma}=\varphi_k\left(\tilde \square^k_{a_\gamma}\right),
$
with 
\[
\tilde \square^k_{a_\gamma}\coloneqq (b\gamma^{-\beta},l_k-b \gamma^{-\beta}) \times (0,a_\gamma).
\]
 In the following we denote $s_\gamma\coloneqq b\gamma^{-\beta}$.
Notice that it is sufficient to study the case where $\Gamma_k$ links two convex corners. Indeed, in the two other cases (namely $\Gamma_k$ links one convex corner and one non-convex corner or two non-convex corners) we have $\tilde \square ^k_{a_\gamma}\coloneqq(s_\gamma,l)\times (0,a_\gamma)$ or $\tilde \square ^k_{a_\gamma} \coloneqq (0,l)\times (0,a_\gamma)$ and the study is then the same.

\end{Remarques}

\begin{proposition}\label{minND}
For all $k\in\{1,...,M\}$, for all $E\in(-1,0)$ and $\lambda \in\mathbb R$, one has for $\gamma>0$ large enough,
\begin{align}
\label{interpbD}
\mathcal N(Q^{\gamma,D}_k, E\gamma^2) \geq \gamma \frac{l_k \sqrt{E+1}}{\pi} - C \gamma^{1-\beta},
\end{align}
and 
\begin{align}
\label{interpbD2}
\mathcal N(Q^{\gamma,D}_k, -\gamma^2+\lambda\gamma) \geq \frac{\sqrt{\gamma}}{\pi} \int_{s_\gamma}^{l_k-s_\gamma} \sqrt{\kappa_k(s)+\lambda)_+}ds- C.
\end{align}
\end{proposition}

\begin{proof}
In the following we omit the indices $k$. Let $E\in(-1,0)$ and $\lambda\in\mathbb R$.

We want to perform a change of variables in order to work with $\tilde \square_{a_\gamma}$. 
As in the previous section, we first introduce a unitary transform. 
Define $\tilde U_{a_\gamma} : L^2(\tilde \Omega_{a_\gamma})\to L^2(\tilde \square_{a_\gamma})$, 
$\tilde U_{a_\gamma}\phi(s,t)\coloneqq \sqrt{1-\kappa(s)}\phi\circ\varphi(s,t)$. 
Then, $\tilde U_{a_\gamma} \left(D(q^{\gamma,D}) \right)=\{\phi \in H^1(\tilde \square_{a_\gamma}), \phi(s_\gamma, t)=\phi(l-s_\gamma,t)=\phi(s,a_\gamma)=0\}$. 
It is easy to prove that, after a using integration by parts, $q^{\gamma,D}(\phi,\phi)=p^{\gamma,D}(\tilde U_{a_\gamma}\phi,\tilde U_{a_\gamma})$ with $D(q^{\gamma,D})\coloneqq \tilde U_{a_\gamma} \left(D(q^{\gamma,D}) \right)$ and $p^{\gamma,D}$ is given by the following expression :
\begin{align*}
p^{\gamma,D}(\phi,\phi)=&\int_{\tilde \square_{a_\gamma}} ( \frac{1}{(1-t\kappa(s))^2} \lvert \partial_s \phi \rvert^2 +\lvert \partial_t \phi \rvert^2 -P(s,t)\lvert \phi\rvert^2 )dsdt \\
&-\int_{s_\gamma}^{l-s_\gamma} (\frac{\kappa(s)}{2}+\gamma) \lvert \phi(s,0)\rvert^2ds, 
\end{align*}
where the potential $P$ is given by the expression in \eqref{potP}, namely
\[
P(s,t)= \frac{\kappa^2(s)}{4(1-t\kappa(s))^2} + \frac{t \kappa''(s)}{2(1-t\kappa(s))^3} + \frac{5 t^2 (\kappa'(s))^2}{4(1-t\kappa(s))^4}.
\]
As $\tilde U_{a_\gamma}$ is a unitary map we have, for all $n\in\mathbb N$,
\begin{align}
\label{inequ2}
E_n(Q^{\gamma,D})=E_n(P^{\gamma,D}).
\end{align}
We use the estimates mentioned in the proof of Proposition~\ref{interpb} to simplify the study. Recall that, as $\kappa \in C^2([0,l],\mathbb R^2)$ there exists $C>0$ such that for all $(s,t)\in \tilde \square_{a_\gamma}$
\[
1-a_\gamma C<\frac{1}{(1-t\kappa(s))^2} <1+ a_\gamma C, \quad \lvert P(s,t)\rvert \leq C.
\]
Thus, we can write for all $\phi\in D(p^{\gamma,D})$, $p^{\gamma,D}(\phi,\phi)\leq h^{\gamma,D}(\phi,\phi)$, where 
\begin{align*}
h^{\gamma,D}(\phi,\phi)&=\int_{\tilde\square_{a_\gamma}} ((1+a_\gamma C) \lvert \partial_s\phi\rvert^2+\lvert \partial_t\phi\rvert^2 +C \lvert \phi \rvert^2 ) ds dt \\
&- \int_{s_\gamma}^{l-s_\gamma} (\frac{\kappa(s)}{2}+\gamma) \lvert \phi(s,0)\rvert^2 ds, \quad \phi \in D(h^{\gamma,D})\coloneqq D(p^{\gamma,D}).
\end{align*}
 We obtain, by the min-max principle,
\begin{align}
\label{inequ4}
E_n(P^{\gamma,D})\leq E_n(H^{\gamma,D}).
\end{align}
Let us introduce $L_\gamma\coloneqq l-2s_\gamma$ 
and $\tilde\kappa(s)\coloneqq \kappa(s+s_\gamma)$, for $s\in(0,L_\gamma)$. 
Then, $H^{\gamma,D}$ is unitarily equivalent to the operator $\tilde H^{\gamma,D}$ 
acting on $L^2((0,L_\gamma)\times(0,a_\gamma))$ 
and defined as the unique self-adjoint operator associated with the sesquilinear form 
\begin{align*}
\tilde h^{\gamma,D}(\phi,\phi)&=\int_{0}^{L_\gamma}\int_0^{a_\gamma} ((1+a_\gamma C) \lvert \partial_s\phi\rvert^2+\lvert \partial_t\phi\rvert^2 +C \lvert \phi \rvert^2 ) ds dt \\
&- \int_{0}^{L_\gamma} (\frac{\tilde \kappa(s)}{2}+\gamma) \lvert \phi(s,0)\rvert^2 ds,
\end{align*}
where $D(\tilde h^{\gamma,D})\coloneqq \{\phi \in H^1((0,L_\gamma)\times (0,a^\gamma)), \phi(0,t)=\phi(L_\gamma,t)=\phi(s,a_\gamma)=0\}$. 
By the min-max principle we obtain the equality
 \begin{align}
\label{inequ5}
E_n(H^{\gamma,D})=E_n(\tilde H^{\gamma,D}).
\end{align}
Let us now introduce a partition of $(0,L_\gamma)$. For any $K\in\mathbb N$, we denote
\[
\delta\coloneqq \frac{L_\gamma}{K}, \quad I_j\coloneqq(\delta(j-1),\delta j), \quad j\in\{1,...,K\},
\]
and,
\[
\tilde \kappa^-_j\coloneqq\inf_{s\in I_j} \tilde \kappa(s).
\]
We define the new sesquilinear forms adapted to this partition,
\begin{align*}
t_j^{\gamma,D}(\phi,\phi)&=\int_{I_j}\int_{0}^{a_\gamma}((1+a_\gamma C) \lvert \partial_s\phi\rvert^2 +\lvert \partial_t\phi\rvert^2 +C\lvert \phi \rvert^2)  ds dt \\
&-\int_{I_j}(\frac{\tilde \kappa^-_j}{2} +\gamma) \lvert \phi(s,0)\rvert^2 ds,
\end{align*}
where $D(t_j^{\gamma,D})\coloneqq \{\phi \in H^1(I_j \times (0,a_\gamma)), \phi((j-1)\delta,t)=\phi(\delta j,t)=\phi(s,a_\gamma)=0\}$. 
Clearly we have 
$\bigoplus_{j=1}^K D(t_j^{\gamma,D}) \subset D(\tilde h^{\gamma,D})$,
and by the min-max principle we get 
\begin{align}
\label{inequ6}
E_n(\tilde H^{\gamma,D})\leq E_n(\oplus_{j=1}^K T_j^{\gamma,D}).
\end{align}
Let us fix $j\in\{1,...,K\}$. 
It is easy to see that, by separation of variables, 
$E_n(T^{\gamma,D}_j)=E_n(\mathcal L^D\otimes 1 + 1\otimes \mathscr D_j)$.
 Here, the operator $\mathcal L^D$ acts on $L^2(0,\delta)$ as
\[
\mathcal L^D f = -(1+a_\gamma C) f''+ C f, \quad D(\mathcal L^D)\coloneqq H^2(0,\delta)\cap H^1_0(0,\delta).
\]
The operator $\mathscr D_j\coloneqq \mathscr D_{\frac{\tilde \kappa^-_j}{2}+\gamma,a_\gamma}$ defined in Section~\ref{auxop} acts on $L^2(0,a_\gamma)$ as $f\mapsto -f''$ with
\[
D(\mathscr D_j)=\{f\in H^2(0,a_\gamma), -f'(0)-(\frac{\tilde \kappa^-_j}{2}+\gamma)f(0)=f(a_\gamma)=0\}.
\]
There exists $\gamma_1>0$ such that, for all $\gamma >\gamma_1$ we have $(\displaystyle\frac{\tilde \kappa^-_j}{2}+\gamma)a_{\gamma} >1$. Then, using Proposition~\ref{LemmaHP} we know that $E_1(\mathscr D_j)$ is the unique negative eigenvalue of $\mathscr D_j$ and we have the following estimate, for all $\gamma>\gamma_1$,
\begin{align}
\label{estpan}
 E_1(\mathscr D_j)\leq -(\frac{\tilde \kappa^-_j}{2}+\gamma)^2+4(\frac{\tilde \kappa^-_j}{2}+\gamma)^2 e^{-2(\frac{\tilde \kappa^-_j}{2}+\gamma) a_\gamma}.
\end{align}
Let $E\in(-1,0)$ be fixed. For all $\gamma>\gamma_1$ one can write, using estimate \eqref{estpan},
\[
\mathcal N(T^{\gamma,D}_j,E\gamma^2)\geq \gamma \frac{\delta\sqrt{E+1}}{\pi \sqrt{1+a_\gamma C}}-C.
\]
 We immediately have summing on $j\in\{1,...,K\}$,
\[
\mathcal N(H^{\gamma,D}, E\gamma^2) \geq \gamma \frac{L_\gamma\sqrt{E+1}}{\pi\sqrt{1+a_\gamma C}}  -C.
\]
Recall that $a_\gamma\coloneqq \gamma^{-1+\epsilon}$, and then $(1+a_\gamma C)^{-1/2}= 1-\displaystyle\frac{1}{2} C \gamma^{-1+\epsilon} +O(\gamma^{-2+2\epsilon})$, $\gamma \to +\infty$. Moreover, $L_\gamma\coloneqq l-2b\gamma^{-\beta}$. Thus we have
\[
\mathcal N(H^{\gamma,D},E\gamma^2)\geq \gamma \frac{l \sqrt{E+1}}{\pi} -C \gamma^{\epsilon},
\]
with $\epsilon\in (0,1-\beta)$. 
This concludes the proof of \eqref{interpbD} thanks to \eqref{inequ2} and \eqref{inequ4}.

Let $\lambda\in\mathbb R$. We have 
\[
\mathcal N(H^{\gamma,D}, -\gamma^2+\lambda \gamma) \geq \sqrt{\gamma} \frac{1}{\pi \sqrt{1-a_\gamma C}} \frac{L_\gamma}{K}  \sum_{j=1}^K \sqrt{(\tilde \kappa^-_j +\lambda)_+} -C.
\]
Again, we can use the convergence of Riemman sum of the Lipschitz function $s\mapsto \sqrt{(\tilde \kappa(s)+\lambda)_+}$ to write
\[
\int_{0}^{L_\gamma} \sqrt{\tilde \kappa(s)+ \lambda)_+} ds = \frac{L_\gamma}{K} \sum_{j=1}^K  \sqrt{(\tilde \kappa^-_j+\lambda)_+} +O(\frac{1}{K}),\quad K\to +\infty.
\]
This concludes the proof of \eqref{interpbD2} taking $K\in[\gamma,2\gamma]\cap \mathbb N$.
\end{proof}

\begin{proof}[Proof of \eqref{lbNregpart} and \eqref{lbNregpartbis}]
Noticing that if $\phi \in \bigoplus_{k=1}^M D(q^{\gamma,D}_k)\bigoplus D(\tilde q^D_{0})$ then $\phi \in D(q^\gamma_0)$, we obtain by the min-max principle for all $n\in\mathbb N$,
\begin{align}
\label{inequ1}
E_n(Q^\gamma_0)\leq E_n(\oplus_{k=1}^M Q_k^{\gamma,D}\oplus \tilde Q^D_0).
\end{align}
As $s_\gamma=b\gamma^{-\beta}$ we have,
\[
\int_{s_\gamma}^{l_k-s_\gamma} \sqrt{(\kappa_k(s)+\lambda)_+} ds = \int_{0}^{l_k} \sqrt{(\kappa_k(s)+\lambda)_+} ds +O(\gamma^{-\beta}),\quad  \gamma \to +\infty.
\]
Combining it with Proposition~\ref{minND} and \eqref{inequ1} finishes the proof.
\end{proof}

\subsection{Asymptotic behavior of the first eigenvalues on curvilinear polygons}\label{sectascurvpol}

In this section, we prove Theorem~\ref{Intro_Thasvp} for general curvilinear polygons.
\begin{theorem}\label{asexpcurv}
Let $\Omega\subset\mathbb R^2$ be a curvilinear polygon. There exists $C>0$ such that, 
for all $n\in\{1,..., \mathcal N^\oplus\}$ and for $\gamma$ large enough we have,
\[
\lvert E_n(Q^\gamma)-\gamma^2E_n(T^\oplus)\rvert \leq C\gamma^{\frac{4}{3}}.
\]
\end{theorem}

The proof of Theorem~\ref{asexpcurv} follows exactly the same steps as the one for polygons with straight edges. We need the following intermediary result.

\begin{proposition}\label{asvpcurv}
For all $l\in \{0,...,K^\oplus\}$ and for $\gamma$ large enough we have,
\begin{align}
\label{asvpcurvun}
E_{m_1+...+m_l}(Q^\gamma) &\leq \gamma^2 \lambda_l + C \gamma^{4/3}, \\
\label{asvpcurvdeux}
E_{m_0+...+m_l+1}(Q^\gamma) &\geq \gamma^2 \lambda_{l+1} -C \gamma^{4/3},
\end{align}
with the convention $m_0=0$.
\end{proposition}

\begin{proof}[Proof of Theorem~\ref{asexpcurv}]
For each $n\in\{1,..., \mathcal N^\oplus\}$, there exists $l\in\{0,...,K^\oplus-1\}$ such that $m_0+...+m_l+1\leq n\leq m_0+...+m_{l+1}$ and $\lambda_{l+1} = E_n(T^\oplus)$. We get the result by Proposition~\ref{asvpcurv} and the fact that the eigenvalues are ordered in the increasing way.
\end{proof}

\begin{proof}[Proof of Proposition~\ref{asvpcurv}]
We begin with the proof of \eqref{asvpcurvun}.
We introduce 
$d\coloneqq \sum_{j=1}^l m_j$, 
$\mathcal F^\gamma \coloneqq \Span \{ \tilde \phi^{\gamma,v}_n, (n,v)\in \bigcup_{j=1}^l \mathcal S_j\}$, 
and for simplicity we denote by $(\tilde \phi_1,..., \tilde \phi_d)$ the elements of 
$\{\tilde \phi^{\gamma,v}_n, (n,v)\in\bigcup_{j=1}^l \mathcal S_l \}$. 
By Lemma~\ref{licurv}, $\dim(\mathcal F^\gamma)=d$ for $\gamma$ large enough. Then, by the min-max principle, for $\gamma$ large enough we have
\begin{align}\label{asvpcurv1}
E_d(Q^\gamma)\leq \sup_{\substack{\psi \in \mathcal F^\gamma \\ \psi \neq 0}}\frac{q^{\gamma}(\psi,\psi)}{\lVert \psi\rVert^2} 
=\sup_{\substack{(c_1,...,c_d)\in \mathbb C^d\\ (c_1,...,c_d)\neq(0,...,0)}} \frac{q^{\gamma}(\sum_{j=1}^d c_j \tilde \phi_j,\sum_{j=1}^d c_j \tilde \phi_j)}{\lVert \sum_{j=1}^d c_j \tilde \phi_j\rVert^2}.
\end{align}
Let us first expand the numerator :
\[
q^\gamma(\sum_{j=1}^d c_j \tilde \phi_j,\sum_{j=1}^d c_j \tilde \phi_j)=\sum_{j=1}^d \lvert c_j\rvert^2 q^\gamma(\tilde \phi_j,\tilde \phi_j) +2\Re \sum_{j<k} c_j \overline{c_k} q^{\gamma}(\tilde \phi_j,\tilde \phi_k).
\]
We use \eqref{prop2} to obtain 
\begin{align}\label{asvpcurv3}
q^\gamma(\sum_{j=1}^d c_j \tilde \phi_j,\sum_{j=1}^d c_j \tilde \phi_j) \leq \left( \gamma^2 \lambda_l + C\gamma^{2-\beta}\right) \sum_{j=1}^d \lvert  c_j\rvert^2.
\end{align}
Then, the denominator expands as 
\[
\lVert \sum_{j=1}^d c_j \tilde \phi_j\rVert^2=\sum_{j=1}^d \lvert c_j\rvert^2 \lVert \tilde \phi_j\rVert^2 +2\Re \sum_{j<k} c_j \overline{c_k} \langle \tilde \phi_j,\tilde \phi_k\rangle,
\]
and by \eqref{prop1bis} we get
\begin{align}
\label{asvpcurv4}
\left \lvert \lVert \sum_{j=1}^d c_j \tilde \phi_j\rVert ^2 -\sum_{j=1}^d \lvert c_j\rvert^2 \right \rvert \leq C \gamma^{-\beta} \left ( \sum_{j=1}^d \lvert c_j \rvert^2 \right).
\end{align}
Combining \eqref{asvpcurv3} and \eqref{asvpcurv4}, we get for $\gamma$ large enough,
\[
\frac{q^{\gamma}(\sum_{j=1}^d c_j \tilde \phi_j,\sum_{j=1}^d c_j \tilde \phi_j)}{\lVert \sum_{j=1}^d c_j \tilde \phi_j\rVert^2} \leq \gamma^2\lambda_l +C \gamma^{2-\beta},
\]
which concludes the proof on the upper bound thanks to \eqref{asvpcurv1} and taking $\beta=\frac{2}{3}$.

Let us now focus on the lower bound \eqref{asvpcurvdeux}. In the following $d\coloneqq\sum_{j=0}^l m_l$. Thanks to Lemma~\ref{mincurvpol} we can write 
\begin{align}
\label{min1}
E_{d+1}(Q^\gamma)\geq E_{d+1}\left(\bigoplus_{v\in\mathcal V} Q^{\gamma,V}_{v,2\gamma^{-\beta}}\bigoplus Q^{\gamma,V}_0 \right).
\end{align}
Moreover, by \eqref{majV}, we have the lower bound, for all $n\in\mathbb N$,
\begin{align}
\label{lbvpcurv}
E_{n}(Q^{\gamma,V}_{v,2\gamma^{-\beta}}) \geq E_n(Q^\gamma_{v,2\gamma^{-\beta}})-c\gamma^{2\beta},
\end{align}
where $Q^\gamma_{v,2\gamma^{-\beta}}$ acts on $L^2(\Omega_{v,2\gamma^{-\beta}})$ and is defined as the unique self-adjoint operator associated with
\[
q^\gamma_{v,2\gamma^{-\beta}}(\phi,\phi)=\int_{\Omega_{v,2\gamma^{-\beta}}} \lvert \nabla \phi \rvert^2 dx -\gamma \int_{\Gamma_{v,2\gamma^{-\beta}}} \lvert \phi \rvert ^2 ds,\quad \phi \in D(q^\gamma_{v,2\gamma^{-\beta}})\coloneqq D(q^{\gamma,V}_{v,2\gamma^{-\beta}}).
\]
Let us fix $v\in\mathcal V$. 
In order to study $Q^\gamma_{v,2\gamma^{-\beta}}$ we perform a change of variables.
For $\phi \in D(q^\gamma_{v,2\gamma^{-\beta}})$, 
we introduce $\psi(u)\coloneqq\phi\circ F^{-1}_v(u)$ 
for all $u\in U_v\cap B(0,2\gamma^{-\beta})$. 
By Taylor-Lagrange, for all $u\in B(0,2\gamma^{-\beta})$ we have
\begin{align}
\label{TL7}
\lvert (\nabla F^{-1}_v(u))^{-1} -I_2\rvert \leq C \gamma^{-\beta}.
\end{align}
Thanks to the estimates \eqref{est2}, \eqref{est3} and \eqref{TL7} we get, for all $\phi\in D(q^\gamma_{v,2\gamma^{-\beta}})$,
\begin{align}
\label{TS1}
q^\gamma_{v,2\gamma^{-\beta}}(\phi,\phi) \geq (1-C\gamma^{-\beta}) \int_{(U_v)_{0,2\gamma^{-\beta}}} \lvert \nabla \psi\rvert^2 du -\gamma(1+C\gamma^{-\beta}) \int_{\Sigma_{v,2\gamma^{-\beta}}} \lvert \psi \rvert ds.
\end{align} 
We now introduce the sesquilinear form
\[
t^{f_\beta(\gamma)}_v(\psi,\psi) =\int_{(U_v)_{0,2\gamma^{-\beta}}} \lvert \nabla \psi \rvert^2 du -f_\beta(\gamma) \int_{\Sigma_{v,2\gamma^{-\beta}}} \lvert \psi \rvert^2 ds,
\]
with $D(t^{f_\beta(\gamma)}_v)\coloneqq \{\psi \in H^1\left((U_v)_{0,2\gamma^{-\beta}}\right), \psi(u)=0 \text{ for } u\in \partial(U_v)_{0,2\gamma^{-\beta}}\backslash \Sigma_{v,2\gamma^{-\beta}}\}$, 
and 
\[
f_\beta(\gamma)\coloneqq\gamma +C\gamma^{1-\beta}.
\] 

\begin{lemma}\label{min-maxTS}
For any $n\in\mathbb N$ and for $\gamma$ large enough,
\[
E_n(Q^\gamma_{v,2\gamma^{-\beta}}) \geq (1-C\gamma^{-\beta}) E_n(T^{f_\beta(\gamma)}_{v,2\gamma^{-\beta}}).
\]
\end{lemma}

\begin{proof}
First, we have to notice that if 
$\phi \in D(q^\gamma_{v,2\gamma^{-\beta}})$,
 then $\psi\in D(t^{f_\beta(\gamma )}_{v,2\gamma^{-\beta}})$. 
We  can use \eqref{TS1} and  the min-max principle to obtain,
\begin{align*}
E_n(Q^\gamma_{v,2\gamma^{-\beta}}) 
\geq(1-C\gamma^{-\beta}) \inf_{\substack{\mathcal G \subset D(q^\gamma_{v,2\gamma^{-\beta}})\\ \dim(\mathcal G)=n }}
\sup_{\substack{\phi\in \mathcal G\\\phi \neq 0}} \frac{t^{f_\beta(\gamma)}_{v,2\gamma^{-\beta}}(\psi,\psi)}{\lVert \phi \rVert^2_{L^2(\Omega_{v,2\gamma^{-\beta}})}}.
\end{align*}
By \eqref{est1} we have
\[
\frac{\displaystyle\int_{(U_v)_{0,2\gamma^{-\beta}}}\lvert \nabla \psi\rvert^2 dx}{\lVert \phi \rVert^2_{L^2(\Omega_{v,2\gamma^{-\beta}})}} \geq \frac{\displaystyle\int_{(U_v)_{0,2\gamma^{-\beta}}}\lvert \nabla \psi\rvert^2 dx}{(1+C\gamma^{-\beta})\lVert \psi\rVert^2_{L^2\left((U_v)_{0,2\gamma^{-\beta}}\right)}},
\]
and
\[
-f_\beta(\gamma)\frac{\displaystyle \int_{\Sigma_{v,2\gamma^{-\beta}}}\lvert \psi\rvert^2 ds}{\lVert \phi\rVert^2_{L^2(\Omega_{v,2\gamma^{-\beta}})}} \geq 
-f_\beta(\gamma) \frac{\displaystyle\int_{\Sigma_{v,2\gamma^{-\beta}}}\lvert \psi \rvert^2ds}{(1-C\gamma^{-\beta})\lVert \psi \rVert^2_{L^2\left((U_v)_{0,2\gamma^{-\beta}}\right)}}.
\]
 Thus, we first obtain 
\[
E_n(Q^\gamma_{v,2\gamma^{-\beta}}) \geq (1-C\gamma^{-\beta}) 
\inf_{\substack{\mathcal G \subset D(q^\gamma_{v,2\gamma^{-\beta}})\\ \dim(\mathcal G)=n }}
\sup_{\substack{\phi\in \mathcal G\\\phi \neq 0}} \frac{t^{f_\beta(\gamma)}_{v,2\gamma^{-\beta}}(\psi,\psi)}{\lVert \psi \rVert^2_{L\left((U_v)_{0,2\gamma^{-\beta}}\right)}}.
\]
If we denote $\mathcal J\coloneqq\{\psi= \phi\circ F_v^{-1},\phi\in\mathcal G\}$ and if
 $(\phi_1,...,\phi_n)$ is an orthonormal basis of $\mathcal G$, then using again \eqref{TL1} we obtain
\[
\lvert \langle \phi_i,\phi_j\rangle_{L^2(\Omega_{v,2\gamma^{-\beta}})} -\langle \psi_i,\psi_j\rangle_{L^2\left((U_v)_{0,2\gamma^{-\beta}}\right)} \rvert \leq C\gamma^{-\beta} \lvert \langle \phi_i,\phi_j\rangle_{L^2(\Omega_{v,2\gamma^{-\beta}})} \rvert,
\]
where $\psi_k=\phi_k\circ F^{-1}_v$. 
Then, $(\psi_k)_{k=1}^n$ is linearly independent if $\gamma$ is large enough which implies that $\dim(\mathcal J)=n$ for large $\gamma$. 
Finally
\[
E_n(T^{f_\beta(\gamma)}_{v,2\gamma^{-\beta}}) \leq
\inf_{\substack{\mathcal G \subset D(q^\gamma_{v,2\gamma^{-\beta}})\\ \dim(\mathcal G)=n }}
\sup_{\substack{\phi\in \mathcal G\\\phi \neq 0}} \frac{t^{f_\beta(\gamma)}_{v,2\gamma^{-\beta}}(\psi,\psi)}{\lVert \psi \rVert^2_{L^2\left((U_v)_{0,2\gamma^{-\beta}}\right)}},
\]
which concludes the proof.
\end{proof}
Extending $\psi \in D(t^{f_\beta(\gamma)}_{v,2\gamma^{-\beta}})$ by $0$, we immediately have for all $n\leq \mathcal N^\oplus$ and for $\gamma$ large enough, thanks to Lemma~\ref{min-maxTS} and the min-max principle, 
\[
E_n(Q^\gamma_{v,2\gamma^{-\beta}}) \geq (1-C\gamma^{-\beta})E_n(T^{f_\beta(\gamma)}_v).
\]
In particular, 
\begin{align*}
E_{d+1}(\bigoplus_{v\in\mathcal V} Q^\gamma_{v,2\gamma^{-\beta}})
 &\geq (1-C\gamma^{-\beta}) (f_\beta(\gamma))^2 E_{d+1} (T^{\oplus}) \\
&= (1-C\gamma^{-\beta}) (f_\beta(\gamma))^2 \lambda_{l+1}.
\end{align*}
Notice that $(1-C\gamma^{-\beta}) (f_\beta(\gamma))^2=\gamma^2 + O(\gamma^{2-\beta})$, as $\gamma \to +\infty$. Then, for $\gamma$ large enough,
\begin{align}
\label{asvpcurv5}
E_{d+1}(\bigoplus_{v\in\mathcal V} Q^\gamma_{v,2\gamma^{-\beta}}) \geq \gamma^2 \lambda_{l+1}-C \gamma^{2-\beta}.
\end{align}
To finish the proof, in view of \eqref{min1}, we need a lower bound of the first eigenvalue of $Q^{\gamma,V}_0$. 
By the inequality \eqref{lbvpregpart} of Lemma~\ref{thregpart} we know that $E_1(Q^{\gamma,V}_0)\geq -\gamma^2-C\gamma^{2\beta}$ for $\gamma$ large enough. As $\lambda_{l+1}<-1$,  we finally obtain
\[
E_{d+1}(Q^\gamma)\geq \gamma^2\lambda_{l+1}-C\gamma^{2-\beta}-c\gamma^{2\beta}.
\]
Taking $\beta=2/3$ gives us the result.
\end{proof}

\subsection{Weyl asymptotics for Robin Laplacian on curvilinear polygons}\label{sectweylascurvpol}

In this section we prove Theorem~\ref{Intro_asweyl}.
The choice of the thresholds $E\gamma^2$ for $E\in (-1,0)$ and $-\gamma^2 +\lambda\gamma$ for $\lambda\in\mathbb R$ is lead by the study of domains with smooth boundary \cite{HKR}. 

\begin{theorem}\label{weylcurvpol}
For all $E\in(-1,0)$ and $\beta\in(1/2,1)$,
\begin{align}
\label{weylcurvpol1}
\mathcal N(Q^\gamma, E\gamma^2)= \gamma \frac{\lvert \partial \Omega\rvert \sqrt{E+1}}{\pi} + O(\gamma ^{1-\beta}), \quad \text{as}\quad \gamma \to +\infty.
\end{align}
For all $\lambda \in\mathbb R$,
\begin{align}
\label{weylcurvpol2}
\mathcal N(Q^\gamma,-\gamma^2+\lambda \gamma) = \frac{\sqrt{\gamma}}{\pi} \sum_{k=1}^M \int_{0}^{l_k} \sqrt{(\kappa_k(s)+\lambda)_+} ds + O(\gamma ^{1/4}), \quad \text{as}\quad \gamma\to +\infty.
\end{align}
\end{theorem}

\begin{proof}
Let $\tilde E\in[-1,0)$ and $\lambda \in\mathbb R$. Gathering the results of Lemma~\ref{mincurvpol} and Lemma~\ref{Dbracket} we can write 
\begin{align*}
\mathcal N(Q^\gamma_0,\tilde E\gamma^2+\lambda \gamma ) \leq \mathcal N(Q^\gamma,\tilde E\gamma^2&+\lambda\gamma)\\
&\leq  \mathcal N(Q^{\gamma,V}_0,\tilde E\gamma^2+\lambda\gamma) +\sum_{v\in\mathcal V} \mathcal N(Q^{\gamma,V}_{v,2\gamma^{-\beta}},\tilde E\gamma^2+\lambda\gamma).
\end{align*}
 By \eqref{lbvpcurv} we know that
$
\mathcal N(Q^{\gamma,V}_{v,2\gamma^{-\beta}},\tilde E\gamma^2+\lambda\gamma) \leq \mathcal N(Q^{\gamma}_{v,2\gamma^{-\beta}},\tilde E\gamma^2+\lambda\gamma+c\gamma^{2\beta})
$. Moreover, Lemma~\ref{thregpart} gives us estimates on the eigenvalue counting functions of $Q^\gamma_0$ and $Q^{\gamma,V}_0$. Hence, in order to conclude we need to prove that the truncated sectors do not contribute to the Weyl law at the leading order.

\begin{proposition}\label{TSprop}
For all $E\in(-1,0)$, $\beta\in(1/2,1)$  and $C>0$ we have for large $\gamma$,
\begin{align*}
\mathcal N(Q^\gamma_{v,2\gamma^{-\beta}},E\gamma^2+C\gamma^{2\beta})= O(\gamma^{1-\beta}).
\end{align*}
\end{proposition}

The asymptotics \eqref{weylcurvpol1} and \eqref{weylcurvpol2} follows immediately from Lemma~\ref{thregpart} and Proposition~\ref{TSprop}, taking $\beta=3/4$ for the second one.
\end{proof}

\begin{proof}[Proof of Proposition~\ref{TSprop}]
By Lemma~\ref{min-maxTS} we can write for all $n\in\mathbb N$,
\[
E_n(Q^\gamma_{v,2\gamma^{-\beta}}) \geq (1-C\gamma^{-\beta}) E_n(T^{f_\beta(\gamma)}_{v,2\gamma^{-\beta}}),
\]
where $f_\beta(\gamma)\coloneqq \gamma + C\gamma^{1-\beta}$. Then, 
\begin{align}
\label{TSprop1}
\mathcal N(Q^\gamma_{v,2\gamma^{-\beta}}, E\gamma^2 +C \gamma ^{2\beta}) \leq \mathcal N(T^{f_\beta(\gamma)}_{v,2\gamma^{-\beta}},  E \gamma ^2+C\gamma^m),
\end{align}
where $m\coloneqq \max(2-\beta,2\beta)$.
We are now lead to study the eigenvalue counting function of $T^{f_\beta(\gamma)}_{v,2\gamma^{-\beta}}$.

We introduce $U_v^+\coloneqq U_v\cap\left(\mathbb R_+\times\mathbb R_+\right)$ and
 $\Sigma^+_{v,2\gamma^{-\beta}}\coloneqq \Sigma_{v,2\gamma^{-\beta}} \cap \partial U_v^+$.
Due to the symmetry of the domain $(U_v)_{0,2\gamma^{-\beta}}$ 
with respect to the $x_1-$axis, it is easy to see that
\begin{align}
\label{TS7}
\mathcal N(T^{f_\beta( \gamma)}_{v,2\gamma^{-\beta}}, E\gamma^2+C\gamma^m) \leq 2 \mathcal N(T^{f_\beta( \gamma),+}_{v,2\gamma^{-\beta}}, E\gamma^2+C\gamma^m),
\end{align}
where $T^{f_\beta(\gamma),+}_{v,2\gamma^{-\beta}}$ is the unique self-adjoint operator associated with
\[
t^{f_\beta(\gamma),+}_{v,2\gamma^{-\beta}}(\psi,\psi)=\int_{(U_v^+)_{0,2\gamma^{-\beta}}} \lvert \nabla \psi \rvert ^2 du -f_\beta(\gamma) \int_{\Sigma^+_{v,2\gamma^{-\beta}}} \lvert \psi \rvert^2 ds,
\]
with $D(t^{f_\beta(\gamma),+}_{v,2\gamma^{-\beta}})\coloneqq \left\{ \psi \in H^1\left((U_v^+)_{0,2\gamma^{-\beta}} \right), \psi (u)= 0 \text{ for } u\in \partial(U_v^+)_{0,2\gamma^{-\beta}}\cap \partial B(0,2\gamma^{-\beta})\right \}$. 
We now introduce a partition of $(U_v^+)_{0,2\gamma^{-\beta}}$.
Let $A_\gamma\coloneqq(\gamma^{-1},0)$, 
$H_{A_\gamma}$ be the orthogal projection of $A_\gamma$ on $\Sigma^+_{v,2\gamma^{-\beta}}$ and 
$V$ be the infinite sector obtained by translation of vector $(\gamma^{-\beta}, 0)$ of $U_v$. We denote $V^+\coloneqq V\cap \left ( \mathbb R_+ \times \mathbb R_+\right)$. Notice that $H_{A_\gamma}$ is well defined as  $\mathcal V$ is the set of \emph{convex} vertices of $\Omega$ : then $\alpha_v \in (0,\pi/2)$.
We introduce the two new domains: 
\begin{itemize}
\item $D_\gamma^1$ the triangle defined by its vertices $(0,0)$, $A_\gamma$ and $H_{A_\gamma}$;
\item $D^2 \coloneqq (U_v^+)_{0,2\gamma^{-\beta}} \backslash \left (\overline{D^1_\gamma \cup V^+} \right )$.
\end{itemize}

 \begin{figure}
 	\centering
 		\includegraphics[width=0.70\textwidth]{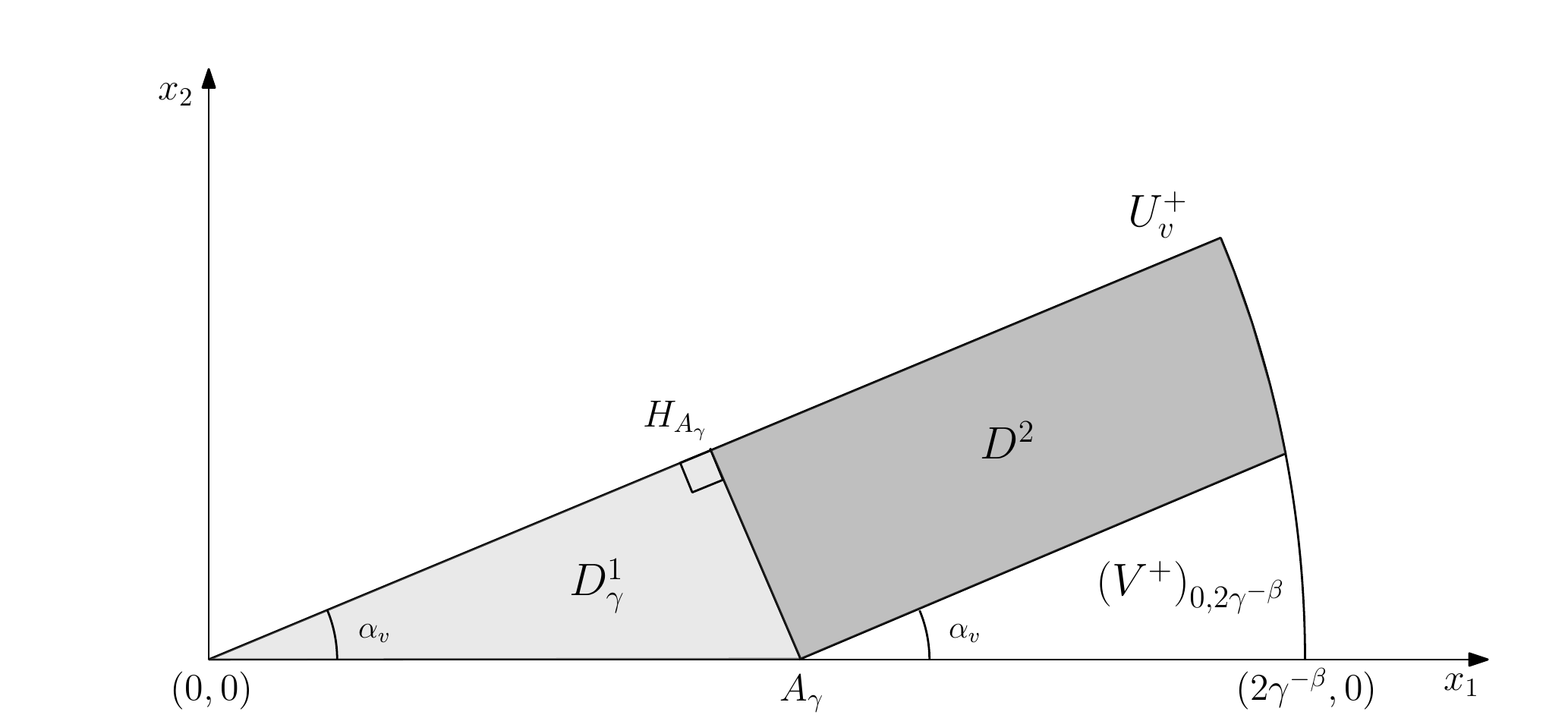}
 	\caption{\label{fig:TS_pol} Partition of $U_v^+$ by the domains $D^1_\gamma$, $D^2$ and $V_{0,2\gamma^{-\beta}}$.}
 \end{figure}

Hence we have, see Figure~\ref{fig:TS_pol},
\[
\overline{(U_v^+)_{0,2\gamma^{-\beta}}} = \overline{ D_\gamma^1 \cup D^2 \cup (V^+)_{0,2\gamma^{-\beta}}},
\]
where $(V^+)_{0,2\gamma^{-\beta}}\coloneqq V^+ \cap B(0, 2\gamma^{-\beta})$.
We consider three new sesquilinear forms associated with this covering, 
\[
h(\psi,\psi)= \int_{(V^+)_{0,2\gamma^{-\beta}}} \lvert \nabla \psi \rvert^2 du, 
\]
with $D(h)\coloneqq \left\{ \psi \in H^1\left((V^+)_{0,2\gamma^{-\beta}}\right), \psi(u) =0\text{ for } u\in \partial (V^+)_{0,2\gamma^{-\beta}}\cap \partial B(0,2\gamma^{-\beta}) \right \}$.
For a domain $D\subset \mathbb R^2$ and $\mu>0$ satisfying $f_\beta(\gamma)\leq \mu \gamma$ for all $\gamma>0$ we define
\[
r^{\mu\gamma}_{D}(\psi,\psi)=\int_{D} \lvert \nabla \psi \rvert^2 du -\mu\gamma \int_{\partial D \cap \partial U_v} \lvert \psi \rvert^2 ds,\quad \psi \in D(r^{\mu\gamma}_{D})\coloneqq H^1(D),
\]
and
\[
p^{\mu\gamma}(\psi,\psi)=\int_{D^2} \lvert \nabla \psi \rvert^2 du -\mu\gamma \int_{\partial D^2\cap \partial U_v^+} \lvert \psi \rvert^2 ds,
\]
with $D(p^{\mu\gamma})\coloneqq \left \{ \psi \in H^1(D^2) : \psi(u) =0 \text{ for } u\in \partial D^2\cap \partial B(0,2\gamma^{-\beta}) \right \}$. 
It is easy to see that $T^{f_\beta(\gamma),+}_{v,2\gamma^{-\beta}}\geq R^{\mu\gamma}_{D^1_\gamma}\oplus H \oplus P^{\mu \gamma}$. 
Moreover,
as the operator $H$ is positive we can write $\mathcal N(H, E\gamma^2+C\gamma^m)=0$ for $\gamma$ large enough.
Thus we obtain
\begin{align}
\label{TS3}
\mathcal N (T^{f_\beta(\gamma),+}_{v,2\gamma^{-\beta}}, E \gamma^2+C\gamma^m) \leq \mathcal N(R^{\mu \gamma}_{D^1_\gamma},  E\gamma^2+C\gamma^m)+ \mathcal N( P^{\mu\gamma}, E \gamma^2+C\gamma^m).
\end{align}
 We perform the change of variables $u=\gamma^{-1}v$ in the sesquilinear form $r^{\mu\gamma}_{D^1_\gamma}$. Then, for all $n \in\mathbb N$ we have 
$
E_n(R^{\mu\gamma}_{D^1_{\gamma}})= \gamma^2 E_n(R^{\mu}_{D^1}),
$ 
where $D^1$ is defined by its vertices $(0,0)$, $A\coloneqq(0,1)$ and $H_A$ being the orthogonal projection of $A$ on $\Sigma^+_{v,2\gamma^{-\beta}}$. 
In particular, 
\[
\mathcal N(R^{\mu\gamma}_{D^1_{\gamma}}, E\gamma^2+C\gamma^m) = \mathcal N(R^{\mu}_{D^1},  E+ C\gamma^{m-2})\leq C,
\]
as $m-2 <0$.
Let us now focus on the operator $P^{\mu\gamma}$. Notice that $D^2$ is included in a rectangle of length $2\gamma^{-\beta}-\gamma^{-1}\cos(\alpha)$ and width $\gamma^{-1} \sin(\alpha)$. Extending $\phi \in D(p^{\mu\gamma})$ by $0$ and using the min-max principle, we obtain
\begin{align*}
\mathcal N(P^{\mu\gamma},  E\gamma^2+ C\gamma^m) \leq N(\mathcal L^{ND}\otimes 1 + 1\otimes \mathscr N^{\mu\gamma}_{\gamma^{-1}\sin \alpha_v}, E\gamma^2 + C\gamma^m),
\end{align*}
where $\mathcal L^{ND}$ is the operator acting on 
$L^2(0, 2\gamma^{-\beta}-\gamma^{-1}\cos(\alpha_v))$ as $f\mapsto -f''$ on
\[
D(\mathcal L^{ND})\coloneqq\{ f\in H^2(0,2\gamma^{-\beta}-\gamma^{-1}\cos \alpha_v), -f'(0)=f(2\gamma^{-\beta}-\gamma^{-1}\cos(\alpha_v))=0\}.
\] 

The operator $  \mathscr N^{\mu\gamma}_{\gamma^{-1}\sin \alpha_v}$  acts on $L^2(0, \gamma^{-1}\sin(\alpha_v))$ as $f\mapsto -f''$ on
\[
D(\mathscr N^{\mu\gamma}_{ \gamma^{-1}\sin \alpha_v})\coloneqq \{H^2(0,\gamma^{-1}\sin\alpha_v), -f'(0)-\mu\gamma f(0)=f'(\gamma^{-1}\sin\alpha_v)=0\}.
\]
We know by \cite[Lemma A.1]{HP} that $E_1(\mathscr N^{\mu\gamma}_{\gamma^{-1}\sin \alpha_v})$ is the unique strictly negative eigenvalue of $\mathscr N^{\mu\gamma}_{ \gamma^{-1}\sin \alpha_v}$ and $E_1(\mathscr N^{\mu\gamma}_{ \gamma^{-1}\sin \alpha_v})=\gamma^2 E_1(\mathscr N^\mu_ {\sin \alpha_v})$.
 Thus, we have
\[
\mathcal N(P^{\mu\gamma}, E\gamma^2 +C\gamma^m) \leq\# \left\{ n\in\mathbb N, E_n(\mathcal L^{ND}) +\gamma^2 E_1(\mathscr N^\mu_{\sin \alpha_v} )<  E\gamma^2 + C\gamma^m\right\},
\]
which implies
\[
\mathcal N(P^{\mu\gamma}, E\gamma^2 +C\gamma^m) \leq \gamma ^{1-\beta}
\frac{\sqrt{\left\lvert  E + E_1(\mathscr N^\mu_{\sin \alpha_v})\right \rvert}}{\pi}+C.
\]
Combining it with \eqref{TSprop1}, \eqref{TS7} and \eqref{TS3} finishes the proof.
\end{proof}


\section{Concluding remarks}\label{prospects}

In Theorem~\ref{asexpcurv}, we proved that the asymptotics of the $\mathcal N^\oplus$ first eigenvalues of the operator $Q^\gamma$ is determined by Robin Laplacians acting on the tangent sectors. The next natural step would be to understand what happens for the next eigenvalues. More precisely, we would like to obtain an asymptotics for $E_{\mathcal N^\oplus+j}(Q^\gamma)$ as $\gamma$ becomes large. 
For now, we can give a first answer stating that the corners do not contribute at the leading order to the the asymptotics.

\begin{proposition}
\label{nexteigen}
For each $j\geq 1$, theres exists $C>0$ such that, for $\gamma$ large enough,
\[
-\gamma^2-C\gamma^{4/3}\leq E_{\mathcal N^\oplus+j}(Q^\gamma) \leq -\gamma^2-\kappa_{\min}\gamma +C,
\]
where $\kappa_{\min}\coloneqq \min_{k=1,...,M}\left(\min_{s\in[0,l_k]}\kappa_k(s)\right)$. Consequently,
\[
E_{\mathcal N^\oplus+j}(Q^\gamma)=-\gamma^2+o(\gamma^2),\quad \text{as}\quad \gamma \to +\infty.
\]
\end{proposition}

\begin{proof}
We obtain the lower bound by Proposition~\ref{asvpcurv}: using \eqref{asvpcurvdeux} with $l=K^\oplus$ we immediately have, for $\gamma$ large enough,
\[
E_{\mathcal N^\oplus +1}(Q^\gamma)\geq -\gamma ^2 -C \gamma^{4/3}.
\]
Let us now focus on the upper bound. We use the notations of Section~\ref{Dirichlet}. Let $k\in\{1,...,M\}$. Recall that $D(q^{\gamma,D}_k)\coloneqq \{\phi\in H^1(\tilde \Omega^k_{a_\gamma}), \phi=0 \text{ on } \partial \tilde \Omega^k_{a_\gamma}\backslash \Gamma_k\}$. As $D(q^{\gamma,D}_k)\subset H^1(\Omega)$ we get by the min-max principle and for all $n\in\mathbb N$,
\[
E_n(Q^\gamma)\leq E_n(Q^{\gamma,D}_k).
\]
Following the steps of the proof of Proposition~\ref{minND}, we know that $E_n(Q^{\gamma,D}_k)\leq E_n(\tilde H^{\gamma,D}_k)$. We now introduce
\begin{align*}
\tilde h^{\gamma,D,-}_k(\phi,\phi)=&\int_{0}^{L_\gamma}\int_{0}^{a_\gamma} \left( (1+a_\gamma C) \lvert \partial_s \phi\rvert ^2 +\lvert \partial_t\phi\rvert^2+C\lvert \phi \rvert^2\right) dsdt\\
&-(\frac{ \kappa_{k,\min}}{2}+\gamma) \int_0^{L_\gamma} \lvert \phi(s,0)\rvert^2 ds ,\quad \phi \in D(\tilde h^{\gamma,D}_k),
\end{align*}
where $\kappa_{k,\min}\coloneqq \min_{s\in[0,l_k]} \kappa_k(s)$.
Then, by the min-max principle,
\[
E_n(Q^{\gamma})\leq E_n(\tilde H^{\gamma,D,-}_k).
\]
By separation of variables it is easy to see that $\tilde H^{\gamma,D,-}_k$ is unitarily equivalent to $\mathcal L^D_{L_\gamma}\otimes 1+ 1\otimes \mathscr D^-$ where the operator $\mathcal L^D_{L_\gamma}$ acts on $L^2(0,L_\gamma)$ as $f\mapsto -(1+a_\gamma C) f''+Cf$ with 
\[
D(\mathcal L ^D_{L_\gamma})\coloneqq H^2(0,L_\gamma)\cap H^1_0(0,L_\gamma).
\]
The operator $\mathscr D^-\coloneqq \mathscr D_{\frac{ \kappa_{k,\min}}{2}+\gamma,a_\gamma}$ defined in Section~\ref{auxop} acts on $L^2(0,a_\gamma)$ as $f\mapsto-f''$ with 
\[
D(\mathscr D^-)\coloneqq\{f\in H^2(0,a_\gamma), -f'(0)-(\frac{\kappa_{k,\min}}{2}+\gamma)=f(a_\gamma)=0\}.
\]
There exists $\gamma_0>0$ such that for all $\gamma>\gamma_0$ we have $(\displaystyle\frac{ \kappa_{k,\min}}{2}+\gamma)a_\gamma >1$.  Then for all $\gamma>\gamma_0$, we know by  Proposition~\ref{LemmaHP} that $ E_1(\mathscr D^-)$ is the unique negative eigenvalue of $\mathscr D^-$ and we have the following estimate
\[
E_1(\mathscr D^-)\leq -(\frac{ \kappa_{k,\min}}{2}+\gamma)^2 +C.
\]
As $\Spec(\mathcal L^D_{L_\gamma})\subset \mathbb R_+$, we then have for $\gamma$ large enough,
\[
E_n(Q^\gamma) \leq E_n(\mathcal L ^D_{L_\gamma})+E_1(\mathscr D^-).
\]
Using the previous estimate on $E_1(\mathscr D^-)$ we get
\[
E_n(Q^\gamma)\leq -\gamma^2- \kappa_{k,\min}\gamma +C.
\]
As it is true for all $k\in\{1,...,M\}$ we can take the minimum over $k$ and obtain the result.

\end{proof}
\begin{Remarques}
 In the present paper, we do not investigate the second term in the asymptotics of the further eigenvalues. However, in the simple case of a square one can see, using separation of variables, that the second term is a constant depending on the length. The general case of a curvilinear polygon need further considerations and this will be discussed elsewhere.
\end{Remarques}

\noindent \textbf{Acknowledgements} 
I would like to thank Konstantin Pankrashkin for his support, his remarks and the fruitful conversations during the achievement of this work.

\appendix
\section{Spectral approximation}\label{speccorol}
\begin{proposition}\label{prop}
Let $\mathcal H$ be a Hilbert space, $A$ a self-adjoint operator acting on $\mathcal H$ and $\lambda \in \mathbb R$. We suppose that there exists $\epsilon >0$ and an orthonormal family $\psi_1,...,\psi_n \in D(A)$ satisfying
\[
\lVert (A-\lambda) \psi_j\rVert< \epsilon, \quad j=1,...,n.
\]
Then,
\[
\dim \Ran P_A(\lambda-\sqrt{n}\epsilon,\lambda+\sqrt{n}\epsilon) \geq n,
\]
where $P_A(a,b)$ stands for the spectral projection of $A$ on the interval $(a,b)\subset \mathbb R$.
\end{proposition}
\begin{proof}
For simplicity we denote $P\coloneqq P_A(\lambda-\sqrt{n}\epsilon,\lambda+\sqrt{n}\epsilon)$.
Let us make a proof by contradiction and suppose that $\dim \Ran P \leq n-1$. Let $\tilde A\coloneqq A_{\lvert \Ran(1-P)}$. Then, $\Spec(\tilde A)\cap (\lambda-\sqrt{n}\epsilon,\lambda+\sqrt{n}\epsilon)=\emptyset$. Moreover, as we assumed $\dim\Ran P \leq n-1$, there exists $\psi \in \Span\{\psi_1,...,\psi_n\}\backslash \{0\}$ such that $\psi \in\Ran(1-P)$. Without loss of generality we assume that there exist $(\alpha_1,...,\alpha_n)\in \mathbb C^n$ such that 
\[
\psi=\sum_{i=1}^n \alpha_i \psi_i \text{ with } \sum_{i=1}^n \lvert \alpha_i\rvert^2 =1.
\]
Then,
\begin{align*}
\lVert (\tilde A-\lambda) \psi \rVert^2 & = \lVert  (1-P) (A-\lambda) \psi \rVert^2 \\
&\leq \lVert (A-\lambda ) \psi \rVert^2 \\
&\leq (\sum_{i=1}^n \lvert \alpha_i\rvert^2) (\sum_{i=1}^n\lVert (A-\lambda) \psi_i\rVert^2).
\end{align*}
By assumptions, we finally get  $\lVert (\tilde A-\lambda) \psi \rVert \leq \sqrt{n} \epsilon$. Then, by the spectral theorem we can conclude that $\tilde A$ admits some spectrum in $(\lambda-\sqrt{n}\epsilon,\lambda+\sqrt{n}\epsilon)$, which is a contradiction.
\end{proof}

\begin{Corollaires}\label{cor}
If there exist $\varphi_1,...,\varphi_n \in D(A)$ linearly independent and satisfying 
\[
\frac{\lVert (A-\lambda) \varphi_j\rVert}{\lVert \varphi_j\rVert} < \epsilon, \quad j=1,...,n,
\]
then 
\[
\dim \Ran P_A(\lambda-n^{\frac{3}{2}}\epsilon\sqrt{\frac{\lambda_{\max}}{\lambda_{\min}}},\lambda+n^{\frac{3}{2}}\epsilon\sqrt{\frac{\lambda_{\max}}{\lambda_{\min}}}) \geq n,
\]
where $\lambda_{\min}$ (resp. $\lambda_{\max}$)
 is the minimal (resp. maximal) eigenvalue of the Gramian matrix of the family $(\varphi_j)_{j=1}^n$. 
 In particular, if $\Specess(A)\cap (\lambda-n^{\frac{3}{2}}\epsilon\sqrt{\frac{\lambda_{\max}}{\lambda_{\min}}},\lambda+n^{\frac{3}{2}}\epsilon\sqrt{\frac{\lambda_{\max}}{\lambda_{\min}}})=\emptyset$ 
 there exist at least $n$ eigenvalues in $(\lambda-n^{\frac{3}{2}}\epsilon\sqrt{\frac{\lambda_{\max}}{\lambda_{\min}}},\lambda+n^{\frac{3}{2}}\epsilon\sqrt{\frac{\lambda_{\max}}{\lambda_{\min}}})$.

\end{Corollaires}

\begin{proof}
The idea consists in using a specific orthonormalized family obtained from $(\varphi_j)_{j=1}^n$ and then use Proposition~\ref{prop}.
We denote by $G$ the Gramian matrix of $(\varphi_j)_{j=1}^n$. It is known that $G$ is a positive hermitian matrix and then there exists an invertible matrix $R$, hermitian and positive such that
\[
G=R^2.
\]
Let us define, for $j=1,...,n$, 
\[
\psi_j=\sum_{l=1}^n (R^{-1})_{jl} \varphi_l.
\]
Then $(\psi_j)_{j=1}^n$ is orthonormal :
\begin{align*}
\langle \psi_j, \psi_k\rangle &=\sum_{l=1}^n\sum_{m=1}^n (R^{-1})_{j,l} \overline{(R^{-1})_{k,m}} \langle \varphi_l,\varphi_m\rangle \\
&=\sum_{l=1}^n\sum_{m=1}^n (R^{-1})_{j,l} (R^{-1})_{m,k} G_{l,m} \\
&=\left ( R^{-1} G R^{-1}\right)_{j,k} \\
&=\left( I_n \right)_{j,k}.
\end{align*}
Moreover, for all $j=1,...,n$,
\begin{align*}
\lVert (A-\lambda) \psi_j\rVert^2  
\leq \epsilon ^2 \sum_{k=1}^n \lvert (R^{-1})_{j,k} \rvert^2 \lVert \varphi_k\rVert^2 + 2\Re \sum_{l<k} (R^{-1})_{j,l} \overline{(R^{-1})_{j,k}} \langle (A-\lambda)\varphi_l, (A-\lambda)\varphi_k\rangle.
\end{align*}
Applying Cauchy Schwartz we get 
\[
\left \lvert \sum_{l<k} (R^{-1})_{j,l} \overline{(R^{-1})_{j,k}} \langle (A-\lambda)\varphi_l, (A-\lambda)\varphi_k\rangle\right \rvert \leq \epsilon^2 \sum_{l<k}  \lvert (R^{-1})_{j,l} (R^{-1})_{j,k}\rvert \lVert \varphi_l\rVert \lVert \varphi_k\rVert,
\]
which implies 
\begin{align*}
\lVert (A-\lambda) \psi_j\rVert^2 \leq \epsilon^2 \left ( \sum_{k=1}^n \lvert (R^{-1})_{j,k}\rvert \lVert \varphi_k\rVert\right )^2 
\leq  \epsilon^2 \left(\sum_{k=1}^n \lvert (R^{-1})_{j,k}\rvert^2\right) \left (\sum_{k=1}^n  \lVert \varphi_k\rVert^2 \right ).
\end{align*}
By definition $\lVert \varphi_k\rVert^2=G_{k,k} \leq \lambda_{\max}$ and $\lvert  (R^{-1})_{j,k}\rvert^2=(R^{-1})_{j,k} (R^{-1})_{k,j} = G^{-1}_{j,j}$ where $G^{-1}$ is the inverse of $G$. Then,  $\lvert  (R^{-1})_{j,k}\rvert^2\leq \displaystyle\frac{1}{\lambda_{min}}$ and we get
\[
\lVert (A-\lambda) \psi_j\rVert \leq \epsilon n \sqrt{\frac{\lambda_{\max}}{\lambda_{\min}}},
\]
which allows us to conclude using Proposition~\ref{prop}.
\end{proof}

\end{document}